\documentclass{article}
\usepackage[utf8]{inputenc}
\usepackage[english]{babel}
\usepackage{mystyle}
\usepackage{tikz}
\usepackage{bm}
\usepackage{hyperref}
\usepackage{caption}
\usepackage[skip=0.5cm]{caption} 

\allowdisplaybreaks

\newcommand{\eps}{\varepsilon}
\newcommand{\horver}{\theta^{hor/ver}}
\newcommand{\hor}{\theta^{hor}}
\newcommand{\ver}{\theta^{ver}}
\newcommand{\ahor}{\alpha^{hor}}
\newcommand{\aver}{\alpha^{ver}}
\newcommand{\dhor}{\delta^{hor}}
\newcommand{\dver}{\delta^{ver}}
\newcommand{\opt}{\theta_{opt}}
\newcommand{\curl}{\textup{curl}^{d,\eps}}
\newcommand{\curln}{\textup{curl}^{d,\eps_n}}
\newcommand{\Rot}{\textup{Rot}}
\newcommand{\sgn}{\textup{sign}}
\newcommand{\INNhor}{\mathcal{I}_{\textup{NN}}^\textup{hor}}
\newcommand{\INhor}{\mathcal{I}^{\textup{hor}}_{\textup{N}}}
\newcommand{\INNver}{\mathcal{I}_{\textup{NN}}^\textup{ver}}
\newcommand{\INver}{\mathcal{I}_{\textup{N}}^{\textup{ver}}}
\newcommand{\derhor}{\partial_1^{d,\eps}}
\newcommand{\derver}{\partial_2^{d,\eps}}

\newcommand{\dy}{\textup{d}y}

\newcommand{\Ind}{\mathcal{I}}
\newcommand{\pc}{\mathcal{PC}}
\newcommand{\pa}{\mathcal{PA}}
\newcommand{\E}{E_{\eps,\Delta}}
\newcommand{\Ehor}{E^{hor}_{\eps,\Delta}}
\newcommand{\Ever}{E^{ver}_{\eps,\Delta}}
\newcommand{\Eiso}{E_{\eps,\delta}}
\newcommand{\Eisohor}{E^{hor}_{\eps,\delta}}
\newcommand{\Eisover}{E^{ver}_{\eps,\delta}}
\newcommand{\Hn}{H_n}
\newcommand{\Hnhor}{H_n^{hor}}
\newcommand{\Hnver}{H_n^{ver}}
\newcommand{\dom}{\textup{dom}}
\renewcommand{\H}{H_{\sigma,\gamma}}
\newcommand{\Sbound}{\mathcal{S}_0(\Omega,\eps)}
\newcommand{\Sboundn}{\mathcal{S}_0(\Omega,\eps_n)}

\newcommand{\T}{T^{\Delta}}
\newcommand{\Sboundper}{\mathcal{S}_0^{\vert per\vert }(\Omega,\eps)}

\DeclareMathOperator{\argmin}{argmin}

\title{Microstructures in a
two-dimensional frustrated spin system: Scaling regimes and a discrete-to-continuum
limit }
\author{Janusz Ginster, Melanie Koser and Barbara Zwicknagl \\
\it\normalsize
Humboldt-Universit\"at zu Berlin, Institut f\"ur Mathematik, Unter den Linden 6, 10099 Berlin
}
\date{\today}

\begin{document}

\maketitle
\begin{abstract}
We study pattern formation within the $J_1$-$J_3$ - spin model on a two-dimensional square lattice in the case of incompatible (ferromagnetic) boundary conditions on the spin field. We derive the discrete-to-continuum $\Gamma$-limit at the helimagnetic/ferromagnetic transition point, which turns out to be characterized by a singularly perturbed multiwell energy functional on gradient fields. Furthermore, we study the scaling law of the discrete minimal energy. The constructions used in the upper bound include besides rather uniform or complex branching-type patterns also structures with vortices. Our results show in particular that in certain parameter regimes the formation of vortices is energetically favorable. 
\end{abstract}

\section{Introduction} 

We study the formation of complex magnetization patterns in certain magnetic compounds. Such structures can often be explained in the framework of statistical mechanics as a result of competing ferromagnetic and antiferromagnetic interactions (see e.g. \cite{B11-2013-Diep}). We focus on an anisotropic version of the classical $J_1-J_3$-spin model on regular square lattices which arises from the classical $J_1-J_2-J_3$-model in the case $J_2=0$. 
Precisely, we set $\Omega \coloneqq [0,1)^2$ and consider the square lattice $\Omega \cap\eps\Z^2$ of side length $\eps>0$. To each atom $(i\eps, j\eps) \in \Omega \cap\eps\Z^2$ we associate a spin vector $u_{i,j}\in \mathbb{S}^1$, where $\mathbb{S}^1$ denotes the unit sphere in $\R^2$. The configurational energy of a spin field $u \colon \Omega \cap\eps\Z^2 \rightarrow \mathbb{S}^1$, $u(i\eps,j\eps)=u_{i,j}$, is given by
\begin{equation}
    \label{eq:J1J3model}
    \begin{split}
     I_{\alpha, \eps}(u) \coloneqq & 
     \alpha^{hor}\sum_{(i,j)} u_{i,j}\cdot u_{i+1, j} -\alpha^{ver}\sum_{(i,j)} u_{i,j}\cdot u_{i, j+1}  
     + \beta^{hor}\sum_{(i,j)} u_{i,j}\cdot u_{i+2, j}+ \sum_{(i,j)} u_{i,j}\cdot u_{i, j+2}\end{split}
\end{equation}
for interaction parameters $\alpha^{hor}, \alpha^{ver}, \beta^{hor}>0$. Here $x\cdot y$ denotes the scalar product of two vectors $x,y\in\R^2$. The first two terms in the energy are ferromagnetic and favor aligned neighboring spin vectors, whereas the other two terms are anti-ferromagnetic and favor antipodal next-to-nearest neighboring spins. 
Throughout this work, we set for simplicity $\beta^{hor}=1$, and consider the case of (incompatible) Dirichlet boundary conditions on the spin field. In this sense, our work complements the more local results from \cite{B11-2015-CS} in the one-dimensional and from \cite{B11-2019-CFO} in the two-dimensional setting, and studies explicitly the influence of boundary conditions, which is a step towards the understanding of the influence of external fields. Our two main results are the following. We first present the variational continuum limit at the helimagnetic/ferromagnetic transition point of the anisotropic model \eqref{eq:J1J3model} in the sense of $\Gamma$-convergence (see Theorem \ref{thm:gamma}), which turns out to be a singularly perturbed multi-well functional for gradient fields with diffuse interfaces. Additionally, we consider the scaling behaviour of the minimal energy \eqref{eq:J1J3model} in the case $\alpha^{hor}=\alpha^{ver}$ (see Theorem \ref{thm:scaling}), and show in particular that there is a parameter regime in which optimal structures have topological defects. \\

Let us briefly explain our results in some more detail and how they complement the literature, the precise statements can be found in Section \ref{sec:mainresults} after the notation is introduced. The results of this paper are essentially contained in the Ph.D. thesis \cite{koser-diss-24} where also more discussion of related literature can be found. 
If no boundary conditions are assigned, the minimizers of the energy \eqref{eq:J1J3model} with $\beta^{hor}=1$ are well-known (see \cite{B11-2015-CS,B11-2019-CFO}). As the energy decouples into a contribution of vertical  interactions and one of horizontal interactions, minimizing configurations are characterized by their behaviour in vertical and horizontal direction, respectively, given as minimizers of one-dimensional energies as studied in \cite{B11-2015-CS}. Precisely, if $\alpha^{hor}\geq 4$ or $\alpha^{ver}\geq 4$ then minimizers are (up to boundary effects), {\em ferromagnetic} in horizontal or vertical direction, respectively, i.e. the spin field is constant in the respective direction. If, however, $\alpha^{hor}$ or $\alpha^{ver} \in (0,4)$ then minimizers are (up to boundary effects) {\em helical} in the respective direction, i.e., spin vectors rotate with the optimal angle $\opt^{hor} \coloneqq \pm \arccos(\alpha^{hor}/4)$ and $\opt^{ver} \coloneqq \pm \arccos(\alpha^{ver}/4)$ between neighboring lattice points in horizontal and vertical direction, respectively. Note that if both, $\alpha^{hor}, \alpha^{ver}\in (0,4)$ then there are (up to a global rotational invariance) four optimal structures, corresponding to clockwise or counterclockwise rotations with the optimal angle in vertical and horizontal direction, respectively, see Figure \ref{fig:optimal_profiles}.\\

{\bf Gamma-convergence. } 
We first focus on the helimagnetic regime $\alpha^{hor}, \alpha^{ver}\in (0,4)$ and study the energy \eqref{eq:J1J3model} with incompatible boundary conditions on the spin field $u$, i.e., setting $u(0,j\eps)=(0, 1)^T$ 
for all $j\eps\in [0,1)\cap\eps\Z$. In this case, the explicit characterization of minimizers is very challenging as one expects the formation of complex magnetization patterns near the boundary, mixing regions with optimal helical structures to approximate the boundary condition. We first focus on the regime where $\eps\to 0$, i.e., the continuum limit, and $\alpha\nearrow 4$, i.e., approaching the helimagnetic/ferromagnetic transition point. Following \cite{B11-2015-CS,B11-2019-CFO}, we reformulate the problem in terms of the order parameters $\sqrt{\frac{2}{\delta^{hor / ver}}} \sin\left( \frac{\horver}{2}\right)$, where $\delta^{hor / ver} = \frac{4-\alpha^{hor / ver}}4$ and
$\horver_{i,j}$ is the oriented angle between neighboring spin vectors in horizontal and vertical direction, respectively. The normalization is chosen such that for $(\hor, \ver)  \in \{(\pm \opt^{hor}, \pm \opt^{ver})\}$, the order parameters take values $(\pm 1,\pm 1)$, for details see Section \ref{sec:preliminaries}. \\
As minimizers for \eqref{eq:J1J3model} are explicitly known, one can proceed in the spirit of $\Gamma$-development (see \cite{Braides-Truskinowski}). A heuristic computation from \cite{B11-2019-CFO} based on \cite{B11-2015-CS} suggests that in the above mentioned parameter regime, the behaviour of a properly renormalized version of \eqref{eq:J1J3model} behaves like a Modica-Mortola type functional in the order parameters, with a typical transition length between regions of different chirality being of order $\eps/\sqrt{\delta^{hor/ver}}$. We recall the heuristics for the reader's convenience in the appendix. In the case without boundary conditions, this behaviour has been confirmed rigorously in \cite{B11-2015-CS} in the one-dimensional setting for $\eps\to 0$, and in \cite{B11-2019-CFO} in the two-dimensional setting in the regime $\eps\to 0$ and $\eps/\sqrt{\delta}\to 0$ for $\delta:=\delta^{hor}=\delta^{ver}$. Here, as expected from the heuristics, the continuum model takes the form of an anisotropic perimeter functional subject to a differential inclusion (see Remark \ref{rem:marco}). We point out that a major difficulty in the two-dimensional setting arises from a compatibility condition for the horizontal and vertical angle changes. More precisely, reformulating the energy in terms of the order parameters induces a discrete curl condition, which essentially ensures that if one starts at the spin vector at a given lattice point, then going along a lattice cell, the angular changes have to add up to a multiple of $2\pi$, i.e., yielding the same spin vector at the beginning and the end of the closed path. 
For related asymptotic results on this and other variants of the $J_1-J_2-J_3$-model in different parameter regimes, we refer, e.g., to \cite{cicalese2022variational,kubin-lamberti,scilla-valocchia,B11-2009-AC,B11-2018-BCDLP,B11-2015-ACR, CICALESE2023112929} and the references therein. \\
We consider here the situation with Dirichlet boundary conditions on the spin field, for which the above described energy becomes infinite. Therefore, we study the regime $\eps\to 0$, $\delta^{hor/ver} \to 0$ and $\lim \eps/\sqrt{2\delta^{ver}}>0$ in the anisotropic setting 
$\delta^{hor}/\delta^{ver}\in [0, \infty)$. We confirm the expectation from the heuristics by rigorously deriving the $\Gamma$-limiting functional, which is of Modica-Mortola type for gradient fields, i.e., a singularly perturbed multiwell energy with incompatible Dirichlet boundary conditions (see Theorem \ref{thm:gamma} for the precise statement). An important step in the proof is to show that the limiting field is curl-free and hence a gradient field (c.f. also \cite{B11-2019-CFO}). On the discrete level, this requires to control the number of {\em vortices}, i.e., lattice cells in which the angular changes do not add up to $0$ but to $\pm 2\pi$.  \\
\begin{figure}[!ht]
    \centering
    \includegraphics[scale = 0.5]{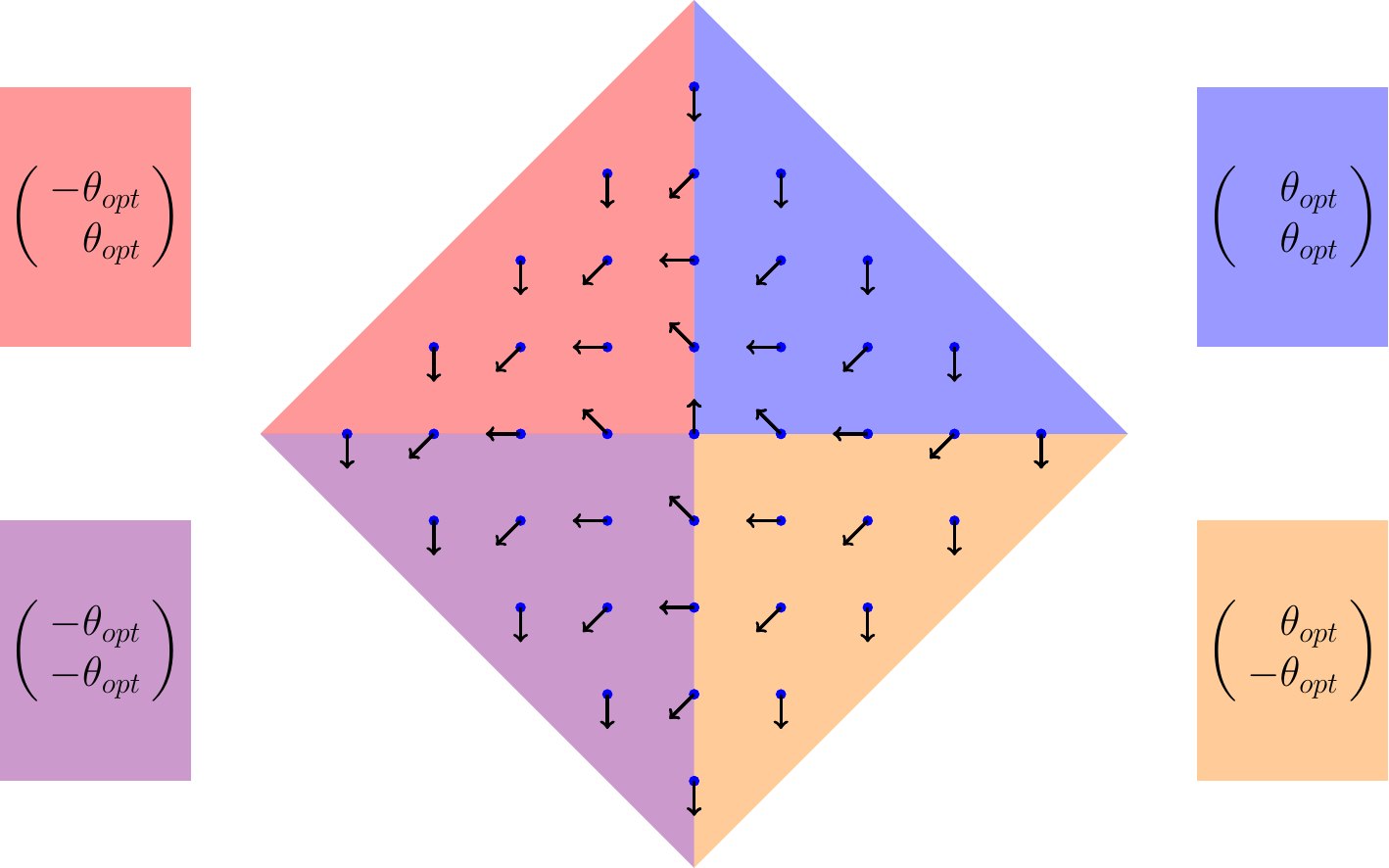}
    \caption{Sketch of a possible spin field  
    with $\alpha \coloneqq \alpha^{hor} = \alpha^{ver} \in (0,4)$.}
    \label{fig:optimal_profiles}
\end{figure}

{\bf Scaling laws.} Singularly perturbed problems related to the limit problem have been studied in \cite{GiZw22,GiZw23_Martensites,Gin23}, and in particular,  the scaling behaviour of the minimal energy has been characterized. The constructions used in the upper bounds there suggest that in some regimes, uniform patterns seem optimal, while in others, fine mixtures of chiralities forming complex branching-type patterns are 
expected. Qualitatively similar behaviour is well-known also in related singularly perturbed models for other pattern forming systems (see \cite{B11-kohn:06}), starting with the seminal work on shape-memory alloys (see e.g. \cite{kohn-mueller:92,B11-1994-KM} and also  e.g. \cite{conti:00,conti:06,zwicknagl:14,knuepfer-kohn-otto:13,chan-conti:14-1,bella-goldman:15,simon-17,simon2018quantitative,knuepfer-otto:18,ContiDiermeierMelchingZwicknagl2020,seiner2020branching,ruland2021energy,CDKZ21} and references therein) but also for continuum models from micromagnetics (see e.g. \cite{choksi-kohn:98,choksi-et-al:99,DKMO:02,ChoksiKohnOtto2004,otto-viehmann:10,knuepfer-muratov:11,DaVeJa20,Dabade-et-al:19,ried2023domain} and the references therein), pattern formation in elastic films (see e.g. \cite{JinSternberg2,BCDM02,bella-kohn:14,BCM2017}), or super-conductors (see e.g. \cite{ChoksiKohnOtto2004,ContiOttoSerfaty2016,ContiGoldmanOttoSerfaty2018,goldman:20}), among many others. However, to the best of our knowledge, this is the first time to study scaling laws for branching-type patterns in discrete frustrated spin systems. 

As we have a $\Gamma$-convergence result with associated compactness, we expect the analogue scaling behaviour (at least asymptotically) in the respective parameter domain for the discrete energies \eqref{eq:J1J3model}. We make this more precise by providing upper and lower bounds on the minimal discrete energy in terms of the problem parameters $\delta^{hor}=\delta^{ver}$ and $\eps$. For the upper bound, on the one hand, we use discrete variants of constructions from \cite{GiZw22}, see Fig.~\ref{fig:ferromagnet} and Fig.~\ref{fig:helimagnet}, and prove a matching lower bound in the setting that there are no vortices in the spin field. However, we also prove that there is a parameter regime in which optimal configurations necessarily contain vortices, which can be viewed as topological defects, see Fig.~\ref{fig:vortex} for a sketch of a construction using vortices. Note that in the parameter regime in which vortices are expected the $\Gamma$-convergence result is not valid which is reflected by the fact that the limiting model does not incorporate defects. For the precise statement we refer to Theorem \ref{thm:scaling}. The conjectured phase diagram is sketched in Fig.~\ref{fig:phase_diagram}. In addition to the theoretical result presented here, note that in certain parameter regimes the formation of branching-type patterns or vortices (as in Figure \ref{fig:helimagnet} and Figure \ref{fig:vortex}) were also observed as the result of a numerical optimization \cite{koser-diss-24, BeKo24}. In particular, these numerical experiments inspired the construction using vortices for the upper bound in Theorem \ref{thm:scaling}. The existence of magnetic vortices in helimagnets was experimentally observed, see e.g. \cite{helimagnet_evidence_Uchida,Yu2011_Vortices}. More discussion can be found e.g. in \cite{Muehl09_Vortices, Muenz10_vortices}. Mathematically, the appearance of vortices and their connection to dislocation theory have been studied for related discrete models in \cite{Cic2011_screw, B11-2009-AC, AlLuGaPo2014, AlGaLuPo2016, AlBrCiLu2022, Alicandro_Braides_Cicalese_Solci_2023,CICALESE2023112929}. We note that our construction using vortices (see Figure \ref{fig:vortex} and proof of Proposition \ref{lem:upper}) bears similarities to configurations with dislocations near semicoherent interfaces (see \cite{AbGinZwi}).\\
\begin{figure}[t!]
\centering
   \begin{minipage}[t]{.43\linewidth}
   \centering
   \captionsetup{width=0.95\linewidth}
      \includegraphics[width=0.8\linewidth]{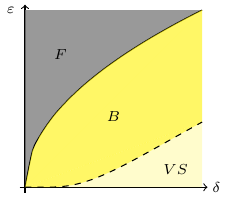}
      \caption{Conjectured phase diagram for the scaling law with the ferromagnetic (F), branching (B), and vortex (VS) phase, c.f.~Fig.~\ref{fig:ferromagnet}, Fig.~\ref{fig:helimagnet} and Fig.~\ref{fig:vortex}}
      \label{fig:phase_diagram}
   \end{minipage}
   \hspace{.1\linewidth}
   \begin{minipage}[t]{.43\linewidth} 
   \centering 
      \captionsetup{width= \linewidth}
      \includegraphics[width=0.7\linewidth]{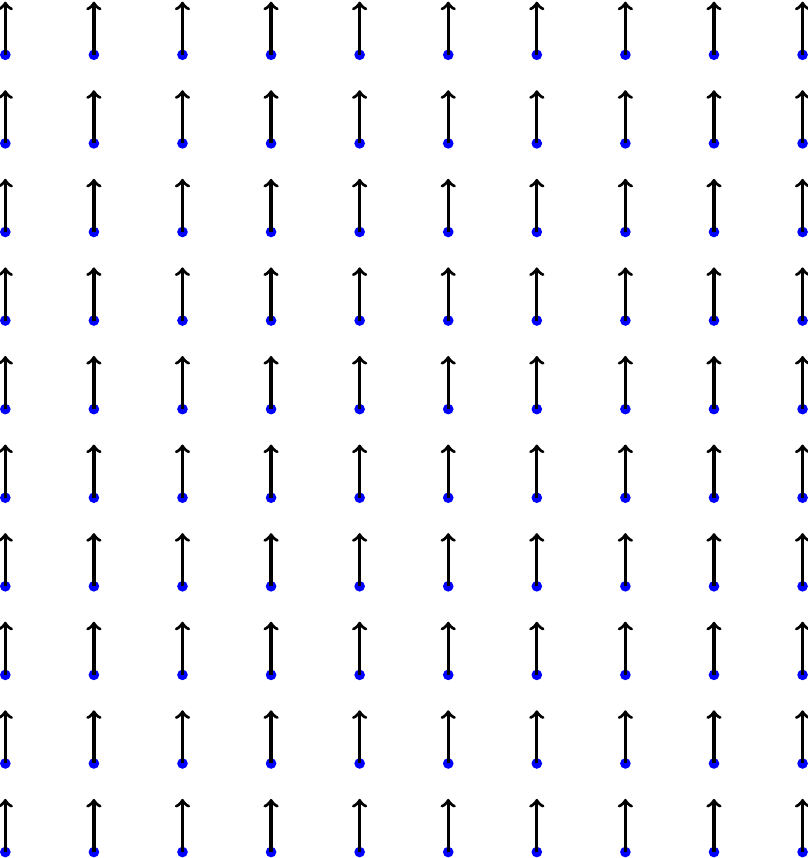}
      \caption{Sketch of the ferromagnetic (F) competitor.}
      \label{fig:ferromagnet}
   \end{minipage}
\end{figure}
\begin{figure}[ht!]
\centering
   \begin{minipage}[t]{.43\linewidth} 
   \centering
      \captionsetup{width=0.95\linewidth}
      \includegraphics[width=0.8\linewidth]{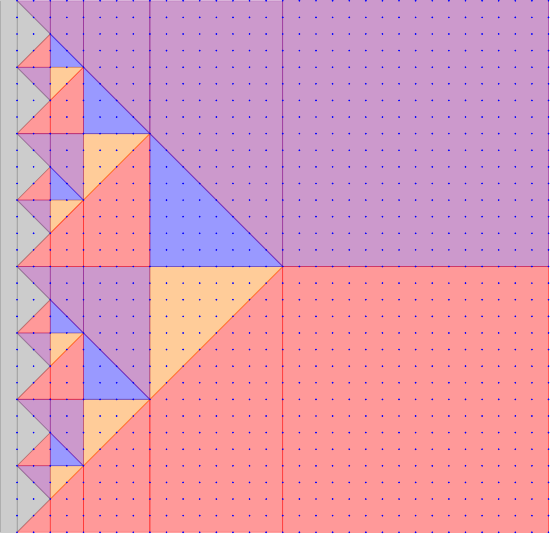}
      \caption{Sketch of the branching-type construction (B) using the optimal profiles from Figure \ref{fig:optimal_profiles}.}
      \label{fig:helimagnet}
   \end{minipage}
   \hspace{.1\linewidth}
   \begin{minipage}[t]{.43\linewidth}
   \centering
      \captionsetup{width=0.95\linewidth}
      \includegraphics[width=0.8\linewidth]{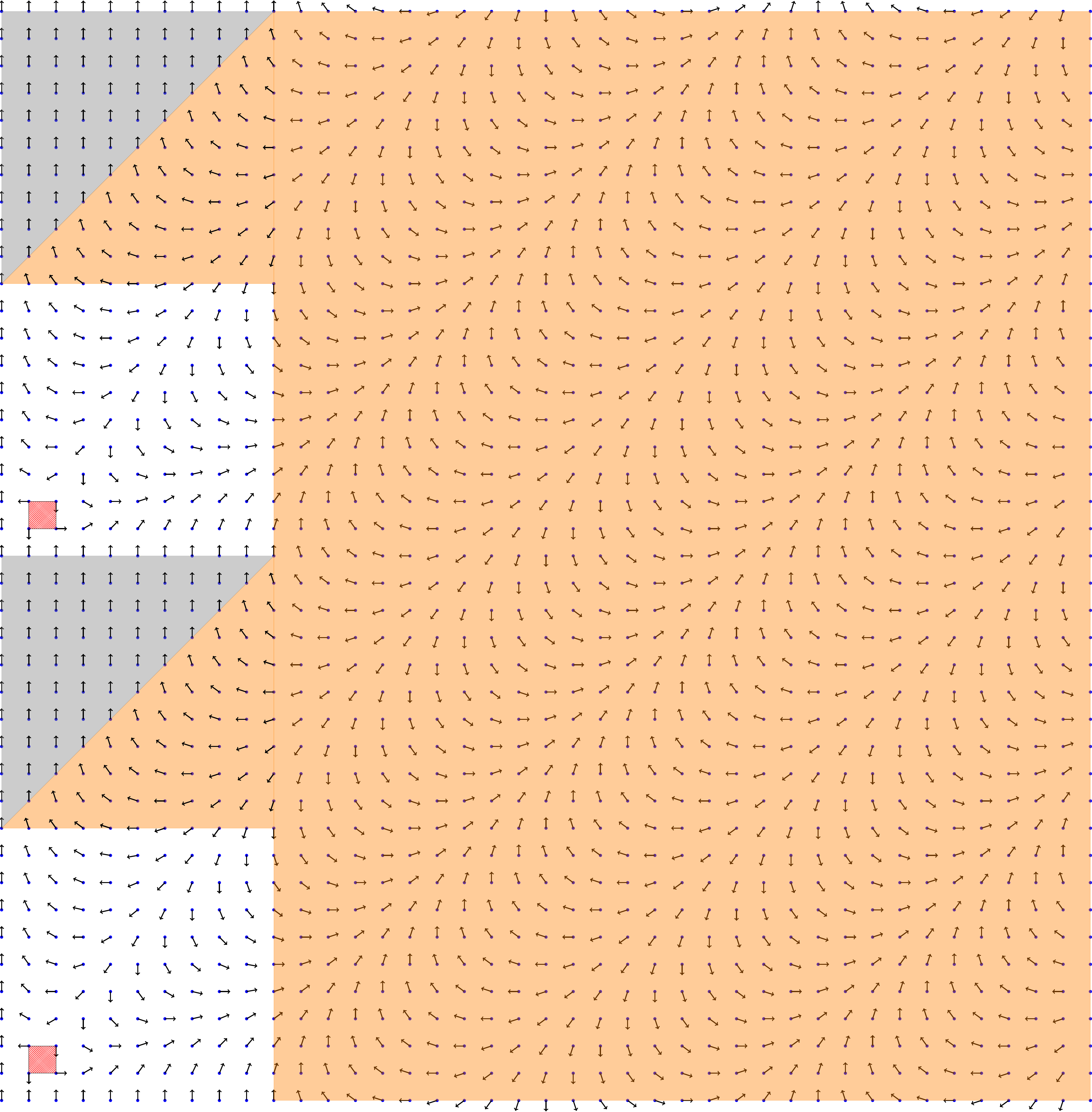}
      \caption{Sketch of a spin field with vortices. Around the red squares the angular changes (measured in $[-\pi,\pi)$) do not add up to $0$. This corresponds to vortices centered in the lower left corners of the squares, c.f.~\eqref{def:vortices}. }
      \label{fig:vortex}
   \end{minipage}
\end{figure}

{\bf Outline.} The remainder of the text is organized as follows. We introduce the precise model and our main results in Section \ref{sec:model}. In Section \ref{sec:preliminaries}, we collect some technical preliminaries needed in the proofs. The main parts are Section \ref{sec:Gamma} where we prove the $\Gamma$-limit result given in Theorem \ref{thm:gamma}, and Section \ref{sec:scaling} where we prove the scaling results for the minimal discrete energy given in Theorem \ref{thm:scaling}.

\section{The mathematical model and main results}\label{sec:model}
\subsection{Notation}
 Let $\Omega \coloneqq [0,1)^2, \ \eps \in (0, 1/2), \ \ahor, \aver \in (0, 4)$,  $\delta^{hor/ver} \coloneqq (4-\alpha^{hor/ver})/4$, 
 $\Delta \coloneqq (\dhor,\dver)$ and $\gamma \coloneqq \sqrt{\frac{\dhor}{\dver}}$. Furthermore, let $\sgn(x)\in\{\pm 1\}$ be the sign of $x\in \R$ with the convention $\sgn(0) = -1$. We define for vectors $a,b\in\R^2$ the Euclidean scalar product by $a\cdot b \coloneqq a^T b$ and the cross product by $a\times b\coloneqq a_1b_2 - a_2 b_1.$  We introduce several sets of indices, namely
\begin{equation}\begin{split}
&\Ind \coloneqq  \{ (i,j) \in \Z^2 \colon (i\eps, j\eps) \in \Omega \cap \eps\Z^2\}, \\
&\INhor \coloneqq \left\{ (i,j) \in\Ind \colon (i+1, j) \in \Ind \right\}, \quad 
\INver \coloneqq  \left\{ (i,j) \in \Ind \colon (i,(j+1))\in  \Ind \right\}, \\
&\INNhor \coloneqq \left\{ (i,j) \in \Ind \colon (i+2, j)\in \Ind\right\} , \quad 
\INNver \coloneqq  \left\{ (i,j) \in \Ind \colon (i,j+2) \in \Ind  \right\}.
\end{split}\label{def:indices}
\end{equation}
  Furthermore, let $Q_\eps(i,j) \coloneqq [i\eps, (i+1)\eps) \times [j\eps, (j+1)\eps)$ for $(i, j) \in \INhor\cap \INver.$ 
Note that we identify a lattice point $(i\eps, j\eps)\in \Omega\cap\eps\Z^2$ with its indices $(i,j)\in \Z^2$. Accordingly, for a function $g \colon \Omega\cap\eps\Z^2\rightarrow\R$ we write $g_{i,j} \coloneqq g(i\eps, j \eps)$ for $(i,j) \in \Ind$.
Moreover, we write for the discrete derivatives of a function $g \colon \Omega\cap\eps\Z^2\rightarrow\R^m$ (with $m \in \N$)
\begin{equation}
\partial_1^{d,\eps} g_{i,j} \coloneqq \frac{g_{i+1, j} - g_{i,j}}{\eps} \quad \textup{ and } \quad  \partial_2^{d,\eps} g_{i,j} \coloneqq \frac{g_{i, j+1} - g_{i,j}}{\eps}
\end{equation}
for $(i,j) \in \INhor$ and $(i,j)\in \INver$, respectively. Additionally, we define the discrete curl of a discrete vector field $(g,h) \colon \Omega \cap\eps\Z^2 \rightarrow \R^2$ via
\begin{equation}
    \nonumber \curl ((g,h)) \coloneqq \derhor h - \derver g \textup{\quad on\quad } \INhor \cap \INver.
\end{equation}
Note that for $(i,j) \in \INhor\cap\INver$, we have
\begin{equation}
    \nonumber \eps\curl ((g,h)_{i,j}) = h_{i+1, j} - h_{i,j} - g_{i,j+1} + g_{i,j}.
\end{equation}

\subsection{Spin fields and their angular velocity field}\label{sec:spin_intr}
In this section, we clarify the relation between a spin field and its angular velocity field. Moreover, we give an estimate of the number of large angles of an angular velocity field in terms of its energy. Note that throughout this manuscript we consider a static situation. The terminology 'angular velocity' refers to the angular changes of the spin field between neighboring points in $\eps \Z^2$ and \emph{not} to angular changes of the spin field in time. We can describe a spin field $v \colon \Omega \cap\eps\Z^2 \rightarrow \mathbb{S}^1$ by its angular velocities
\begin{equation}
    \label{def:horizontal-angle-velocity}
    \hor_{i,j}\coloneqq \sgn(v_{i,j} \times v_{i+1, j}) \arccos(v_{i,j} \cdot v_{i+1, j})\text{\qquad for $(i, j) \in \INhor$, and}
\end{equation}

\begin{equation}
    \label{def:vertical-angle-velocity}
    \ver_{i,j} \coloneqq \sgn(v_{i,j}\times v_{i,j+1}) \arccos(v_{i,j} \cdot v_{i,j+1})\text{\qquad for $(i, j)\in \INver$.}
\end{equation}
Here, $\arccos: [-1,1] \to [0,\pi]$ denotes as usual the inverse of $\cos: [0,\pi] \to [-1,1]$.
 We can retrieve the spin field by setting for $ (i,j) \in \INhor$
\begin{equation}
\nonumber     v_{i,j}= \Rot\left(\sum\limits_{k=0}^{i-1} \hor_{k,j} \right) v_{0,j} \quad \textup{ with } \quad\Rot(\theta) \coloneqq \left( \begin{array}{cc}
   \cos\theta  &-\sin\theta  \\
    \sin\theta & \cos\theta 
\end{array}\right)\quad  \textup{ for } \theta \in \R.
\end{equation}
However, not for every angular velocity field $(\hor, \ver) \colon \Omega \cap\eps\Z^2\rightarrow [-\pi, \pi)^2$ there exists an associated spin field $v \colon \Omega \cap \eps \Z^2 \rightarrow \mathbb{S}^1.$ A necessary and sufficient condition to recover a spin field from an angular velocity field $(\hor,\ver)\colon \Omega \cap\eps\Z \rightarrow [-\pi, \pi)^2$ is the discrete curl condition 
\begin{equation}
    \label{intr:curl_cond}
    \eps\curl ((\hor,\ver)) \in \{-2\pi, 0, 2\pi\}.
\end{equation}
This condition ensures that rotating a spin vector with angular changes according to the field $(\hor,\ver)$ within an $\eps$-cell $Q_\eps(i,j)  \subseteq \Omega$ yields the starting spin vector i.e., \[\Rot(\eps \curl((\hor, \ver))v = v.\] 
Let 
\begin{equation}
    \label{def:vortices}
    V\coloneqq \left\{  (i,j) \in \INNhor\cap \INNver \colon \vert \eps \curl((\hor,\ver)_{i,j})\vert = 2\pi \right\}
\end{equation}
be the set of {\em vortices} of an angular velocity field $(\hor,\ver)\colon \Omega \cap\eps\Z^2 \rightarrow [-\pi, \pi)^2$. 
If it holds $\vert \eps \curl((\hor,\ver)_{i,j})\vert = 2\pi $
for some $(i,j) \in \INNhor\cap\INNver$ we know that at least one of the four angular velocities $\hor_{i,j},\hor_{i,j+1}, \ver_{i,j}$, and $\ver_{i+1,j}$ can in absolute value not be smaller than $\pi/2$ i.e., 
\begin{equation}
\label{size-angles-vortex}
    \max\left\{\vert \hor_{i,j} \vert, \vert \hor_{i,j+1}\vert , \vert \ver_{i, j}\vert, \vert \ver_{i+1,j}\vert \right\}\geq \frac\pi2 .
\end{equation}
Hence, an upper bound on the number of large angular velocities also implies an upper bound on the number of vortices.

\subsection{The chirality parameters}
Next, we introduce the so-called chirality order parameters for an angular velocity field and to reformulate the two-dimensional lattice energy with respect to these parameters. This will be useful in the formulation of the $\Gamma$-convergence result, see Theorem \ref{thm:gamma}.
Let $u \colon  \Omega \cap\eps \Z^2 \rightarrow \mathbb{S}^1$ be a spin field and $(\hor, \ver)\colon \Omega \cap\eps\Z^2 \rightarrow [-\pi, \pi)^2$ the associated angular velocity field, see Section \ref{sec:spin_intr}. The chirality order parameters are then given by 
\begin{equation}
    \label{def:chirality_order_parameters}
   w^{\Delta}\coloneqq  w^{\Delta}(u) \coloneqq \sqrt{\frac{2}{\dhor}}\sin\left(\frac{\hor}{2}\right) \quad \textup{ and }\quad   z^{\Delta} \coloneqq  z^{\Delta}(u) \coloneqq \sqrt{\frac{2}{\dver}}\sin\left(\frac{\ver}{2}\right)
\end{equation}
on $\Omega \cap\eps\Z^2$. The signs of the chirality order parameters determines if a spin field rotates clockwise or counterclockwise in horizontal and vertical direction, respectively. Due to the trigonometric identity $2\sin^2(\frac\theta2) = 1- \cos(\theta)$ for $\theta \in [-\pi, \pi),$ one has $(w, z)\in \left\{ (\pm 1, \pm 1) \right\} $ if and only if the associated angular velocity field satisfies 
\[(\hor, \ver) \in \{ (\pm \theta_{opt}^{hor}, \pm \theta_{opt}^{ver}) \}\quad \textup{ with } \quad \theta_{opt}^{hor} = \arccos(1-\dhor) \textup{ and }\theta_{opt}^{ver}= \arccos(1-\dver).\] 
Moreover, we associate to a spin field $u:   \Omega \cap\eps \Z^2 \rightarrow \mathbb{S}^1$ the mapping \[T^{\Delta}(u) := (\gamma w^{\Delta}(u),z^{\Delta}(u)) \colon  \Omega \cap\eps \Z^2 \rightarrow \R^2.\]

\subsection{Extensions of discrete functions}

For a discrete function $g\colon \Omega \cap\eps\Z^2\rightarrow \R^m$ we denote by $\pc(g)$ its piecewise constant extension, i.e.,
\begin{equation}
    \label{def:gamma:constant_extension}
    \pc(g)(x,y) \coloneqq g_{i,j} \quad \textup{ for all } (x,y) \in Q_{\eps}(i,j)\text{\quad and $i, j = 0, \dots, \lfloor 1/\eps\rfloor.$}
\end{equation}

Additionally, we write $\pa(g)$ for the continuous, piecewise affine extension of $g$, precisely 
\begin{equation}
    \label{def:gamma:triangulation}
    \pa(g)(x,y) \coloneqq \begin{cases}
\partial_1^{d,\eps}g_{i,j} (x-i\eps) +\partial_2^{d,\eps}g_{i,j} (y-j\eps) + g_{i.j} \hfill&\textup{ if } (x,y) \in T^-_{\eps}(i,j),\\
\partial_1^{d,\eps}g_{i, j+1} (x- (i+1)\eps) +  \partial_2^{d,\eps}g_{i+1, j} (y-(j+1)\eps) + g_{i+1, j+1} \hfill \  & \textup{ if } (x,y) \in T^+_{\eps}(i,j),
\end{cases}
\end{equation}
where 
\[ T_{\eps}^-(i,j) \coloneqq \{ (x,y)\in Q_{\eps}(i,j) \colon 0 \leq y-j\eps\leq \eps - (x-i\eps)\} \quad \textup{ and }\quad T_{\eps}^+(i,j) \coloneqq Q_{\eps}(i,j) \setminus T_{\eps}^-(i,j), \]
see Figure \ref{fig:single_cube_triangle} for a sketch of $Q_\eps(i,j)$ and $T_\eps^{\pm}(i,j)$ and Figure \ref{fig:extension_discrete} for an illustration of the extensions $\pc(g)$ and $\pa(g)$ of a discrete function $g\colon \Omega \cap \eps\Z^2 \rightarrow \R^m$.
\begin{figure}[h]
    \centering
    \includegraphics[scale = 0.75]{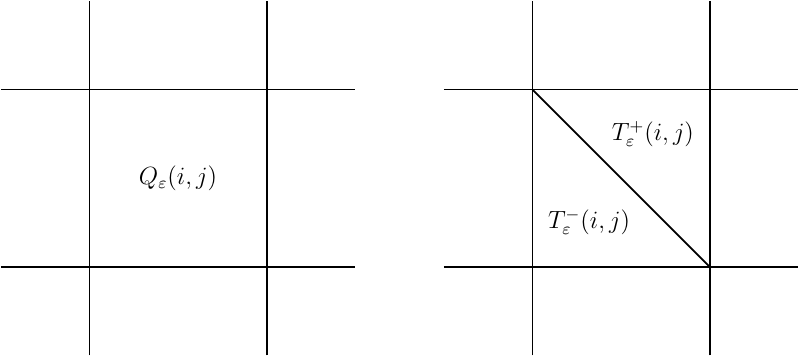}
    \caption{Left: Sketch of square with width $\eps > 0$. Right: Triangulation of this square into $T_{\eps}^-(i,j)$ and $T_{\eps}^+(i,j).$ }
    \label{fig:single_cube_triangle}
\end{figure}

We will sometimes write

\[ \pc(g)(i\eps, j\eps) = \pc(g)_{i,j} = g_{i,j} \quad\textup{ and } \quad \pa(g)(i\eps, j\eps) = \pa(g)_{i,j} = g_{i,j} \quad \textup{ for } i,j = 0, \dots, \left\lfloor \frac1\eps \right\rfloor.\]

\begin{figure}[h]
\centering
      \includegraphics[width=0.45\linewidth]{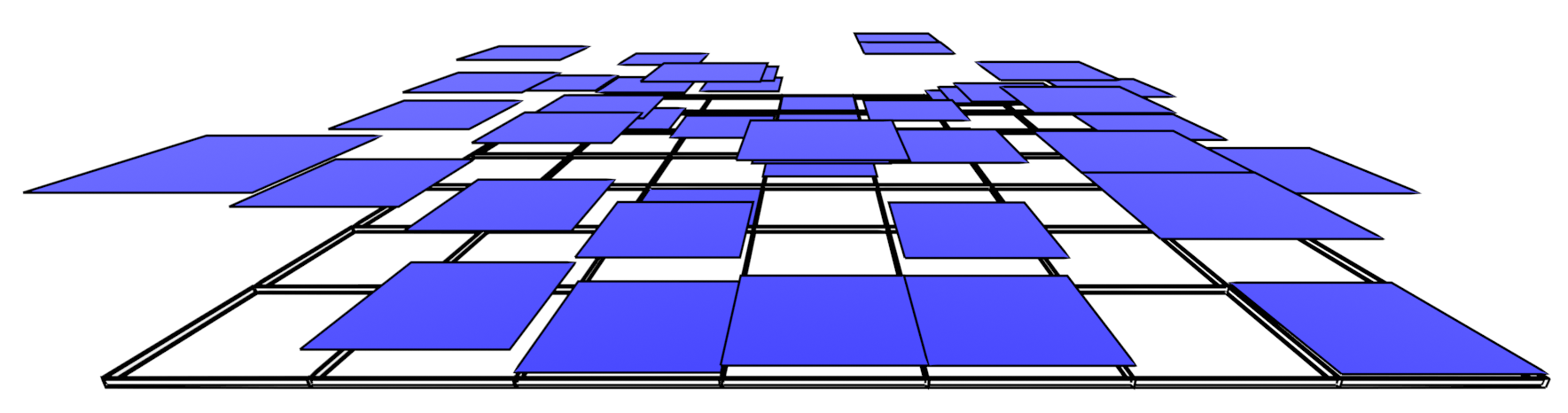}
      \includegraphics[width=0.45\linewidth]{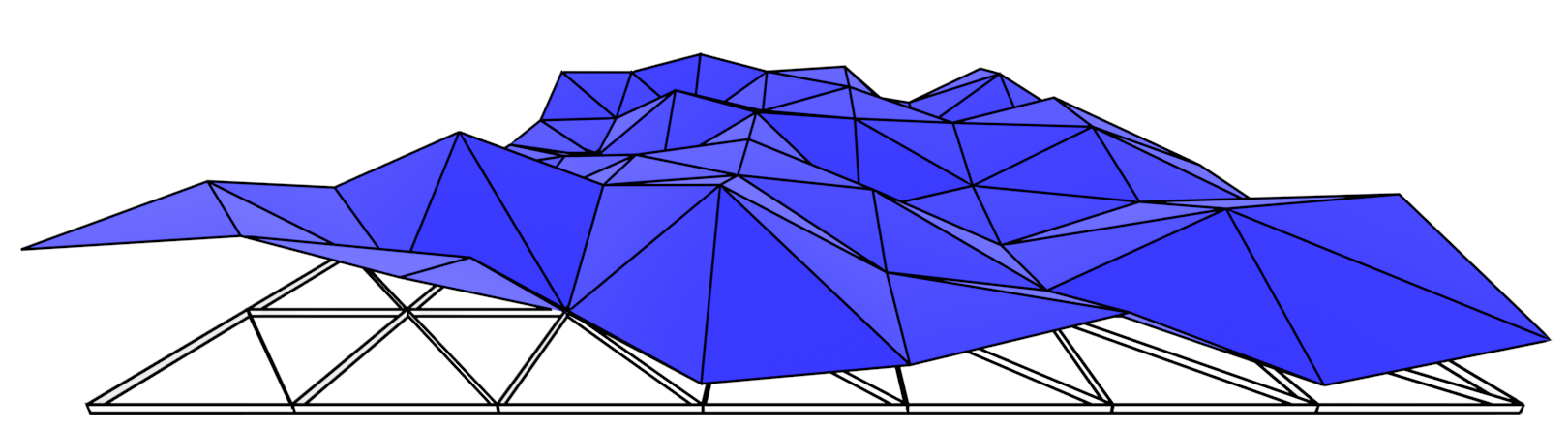}
      \caption{Left: Illustration of the piecewise constant extension $\pc(g) \colon \Omega \rightarrow \R$ of a discrete function $g \colon \Omega \cap\eps\Z^2 \rightarrow \R$ for $\eps = 1/7.$ The squares model the cells $Q_\eps(i,j)$. Right: Sketch of the piecewise affine and continuous extension $\pa(g)\colon \Omega\rightarrow\R.$}
      \label{fig:extension_discrete}
\end{figure}

\subsection{Reformulation of the energy}
Let $u: \Omega \cap \eps \Z^2 \to \mathbb{S}^1$. 
We define the energy $\E(u, \Omega) := \Ehor(u, \Omega) + \Ever(u, \Omega)$, where
\begin{align}
&\Ehor(u, \Omega) := \frac{\eps^2}{2}\sum_{(i,j)\in \INNhor} \vert u_{i,j} - 2(1-\dhor) u_{i+1,j} + u_{i+2,j}\vert^2  \\ 
\text{ and }\quad &\Ever(u, \Omega) := \frac{\eps^2}{2}\sum_{(i,j)\in \INNver} \vert u_{i,j} - 2(1-\dver) u_{i,j+1} + u_{i,j+2}\vert^2.
\end{align}
If $\delta = \dhor = \dver$ we also denote the energies $\E, \Ehor,$ and $\Ever$ by $\Eiso$, $\Eisohor$, and $\Eisover$, respectively. Let $B \subseteq \Omega$. The energy $E(\cdot, B)$ stands for the the $E$ restricted to $\INNhor\cap B$ and $\INNver\cap B$, respectively, for $E\in \{ \E, \Ehor, \Ever, \Eiso, \Eisohor, \Eisover\}.$ 
In the following, we discuss a reformulation of the energy $\E(\cdot, \Omega)$ from \cite{B11-2015-CS, B11-2019-CFO}. \\
   
   By applying the identity 
   \begin{align} &\vert u_{i,j} - 2(1-\dhor) u_{i+1, j} + u_{i+2, j}\vert^2 \\ = &2 + 4(1-\dhor)^2 - 4(1-\dhor) \cos(\hor_{i,j}) - 4(1-\dhor) \cos(\hor_{i+1, j}) + 2 \cos(\hor_{i, j} + \hor_{i+1, j})
   \end{align}
   to each summand of the horizontal energy contribution $\Ehor$, we obtain  
   \begin{equation}
       \begin{split}
           &\Ehor(u, \Omega) = \frac{\eps^2}{2} \sum_{(i,j) \in \INNhor} \vert u_{i,j} - 2(1-\dhor) u_{i+1, j} + u_{i+2, j}\vert^2\\
           = &\sum_{(i,j) \in \INNhor} \bigg[ 1 + 2(1-\dhor)^2 - 2(1-\dhor) \cos(\hor_{i,j}) - 2(1-\dhor) \cos(\hor_{i+1, j})  
           + \cos(\hor_{i, j} + \hor_{i+1, j})\bigg]. 
       \end{split}
   \end{equation}
   Adding and subtracting $\sin^2(\hor_{i,j}) + \sin^2(\hor_{i+1, j}) -1 $ for each $(i,j) \in \INNhor$ yields
\begin{equation}
    \begin{split}
        \Ehor(u, \Omega)  =  &\eps^2\sum_{(i,j) \in \INNhor} \bigg[ 2 + 2(1-\dhor)^2 - 2(1-\dhor) \cos(\hor_{i,j}) - 2(1-\dhor) \cos(\hor_{i+1, j}) \\ & \qquad \qquad- \sin^2(\hor_{i,j}) - \sin^2(\hor_{i+1, j})\bigg] + \eps^2\sum_{(i,j) \in \INNhor}  q(\hor_{i, j}, \hor_{i+1, j}),
    \end{split}
\end{equation}
where 
\begin{equation}
    \label{def:q}
    q(\theta_1, \theta_2) \coloneqq \sin^2(\theta_1) + \sin^2(\theta_2) - (1-\cos(\theta_1 +\theta_2)) \quad \textup{ for } (\theta_1, \theta_2) \in \R^2.
\end{equation}
In turn, using $\sin^2 = 1 -  \cos^2$ we obtain
 \begin{equation}
     \begin{split}
          \Eisohor(u, \Omega) = & \eps^2\sum_{(i,j) \in \INNhor} \bigg[  2(1-\dhor)^2 - 2(1-\dhor) \cos(\hor_{i,j}) - 2(1-\dhor) \cos(\hor_{i+1, j}) \\ & \quad+ \cos^2(\hor_{i,j}) + \cos^2(\hor_{i+1, j})\bigg]
         + \eps^2 \sum_{(i,j) \in \INNhor}  q(\hor_{i, j}, \hor_{i+1, j})\\
         = & \eps^2\sum_{(i,j) \in \INNhor} (1-\dhor - \cos(\hor_{i,j}))^2 + \sum_{(i,j) \in \INNhor} (1-\dhor - \cos(\hor_{i+1,j}))^2 \\ &\quad  + \eps^2\sum_{(i,j) \in \INNhor}  q(\hor_{i, j}, \hor_{i+1, j}).
     \end{split}
 \end{equation}

Furthermore, we introduce 
\begin{equation}
    p(\theta_1, \theta_2) \coloneqq \begin{cases}
    \frac{\sin^2(\theta_1) +\sin^2(\theta_2) -(1-\cos(\theta_1 +\theta_2))}{2(\sin(\frac{\theta_2}{2}) - \sin(\frac{\theta_1}{2}))^2}\quad & \textup{ if } \theta_1 \neq \theta_2,\\
    1 & \textup{ else},
    \end{cases} \quad \textup{ for } (\theta_1, \theta_2) \in \R^2. \label{def:p}
\end{equation}
The equality
 \begin{equation}
     q(\theta_1, \theta_2) = 2 p(\theta_1, \theta_2) \left\vert \sin\left(\frac{\theta_2}{2} \right) - \sin\left(\frac{\theta_1}{2} \right) \right\vert^2 \quad \textup{ for } (\theta_1, \theta_2) \in [-\pi, \pi)^2
 \end{equation}
 and the definition of the chirality parameter $w^\Delta$ yield the equality
 \begin{equation}
    \eps^2 \sum_{(i,j) \in \INNhor}  q(\hor_{i, j}, \hor_{i+1, j}) = \dhor\eps^2\sum_{(i,j) \in \INNhor} p(\hor_{i,j}, \hor_{i+1, j})\left\vert w_{i+1, j}^\Delta- w_{i,j}^\Delta\right\vert^2.
 \end{equation}
 By the identity $1-\cos(\theta) = 2\sin^2(\theta/2)$ for $\theta \in \R$, we obtain 
 \begin{equation}
 \begin{split}
  & \eps^2\sum_{(i,j) \in \INNhor} (1-\dhor - \cos(\hor_{i,j}))^2 + \sum_{(i,j) \in \INNhor} (1-\dhor - \cos(\hor_{i+1,j}))^2\\
  & \quad = (\dhor)^2 \eps^2 \sum_{(i,j) \in \INNhor} \left(1- (w^\Delta_{i+1,j})^2 \right)^2 + (\dhor)^2 \eps^2 \sum_{(i,j) \in \INNhor} \left(1- (w^\Delta_{i,j})^2 \right)^2.\end{split}
 \end{equation}
Altogether, we obtain the reformulation
\begin{equation}
\begin{split}
    \Ehor(u, \Omega) = & \eps^2\sum_{(i,j) \in \INNhor} (1-\dhor- \cos(\hor_{i,j}))^2 + \eps^2\sum_{(i,j) \in \INNhor} (1-\dhor- \cos(\hor_{i+1,j}))^2 \\ &+\eps^2\sum_{(i,j) \in \INNhor} q(\hor_{i,j}, \hor_{i+1, j}),\\
     = & (\dhor)^2 \eps^2\sum_{(i,j) \in \INNhor} (1- (w_{i.j}^\Delta)^2)^2 + (\dhor)^2\eps^2\sum_{(i,j) \in \INNhor} (1-(w_{i,j}^\Delta)^2)^2\\
    & \quad + \dhor\eps^2\sum_{(i,j) \in \INNhor} p(\hor_{i,j}, \hor_{i+1, j})\left\vert w_{i+1, j}^\Delta- w_{i,j}^\Delta\right\vert^2,\end{split} \label{id:energy_hor}
\end{equation}
and the analogous statement holds for the vertical energy contribution $\Ever$. Assuming the periodic boundary conditions
\begin{equation}
   u_{0,k} \cdot u_{1, k} = u_{\lfloor \frac1\eps \rfloor - 1, k} \cdot u_{\lfloor \frac1\eps \rfloor, k} \quad  \textup{ and } \quad u_{k, 0}\cdot u_{k, 1}= u_{k, \lfloor \frac1\eps \rfloor -1} \cdot u_{k, \lfloor \frac1\eps \rfloor} \quad \textup{ for all } k = 0, \dots, \left\lfloor \frac1\eps\right\rfloor,   \label{gamma:spins:periodic_boundary}
\end{equation}
the reformulation of the horizontal energy contribution \eqref{id:energy_hor} yields 
\begin{equation}
\begin{split}
 \Ehor(u, \Omega) = &(\dhor)^2 \eps^2 \sum_{(i,j)\in \INNhor} (1- (w^{\Delta}_{i,j})^2)^2+ (\dhor)^2 \eps^2 \sum_{(i+1,j)\in \INNhor} (1- (w^{\Delta}_{i+1,j})^2)^2 \\
& + \dhor \eps^4 \sum_{(i,j)\in \INNhor} p(\hor_{i,j}, \hor_{i+1, j}) \vert \partial_1^{d,\eps} w^{\Delta}_{i,j}\vert^2\\
 = &2(\dver)^2 \eps^2 \sum_{(i,j)\in \INNhor} (\gamma^2- (\gamma w^{\Delta}_{i,j})^2)^2  + \dver \eps^4 \sum_{(i,j)\in \INNhor} p(\hor_{i,j}, \hor_{i+1, j}) \vert  \gamma \partial_1^{d,\eps} w^{\Delta}_{i,j}\vert^2.
\end{split}\label{gamma:spins:energie_one}
\end{equation}
Analogously, if \eqref{gamma:spins:periodic_boundary} holds, we obtain for the vertical energy contribution
\begin{equation}
\Ever(u, \Omega) = 2(\dver)^2 \eps^2 \sum_{(i,j)\in \INNver} (1- (z^{\Delta}_{i,j})^2)^2 + \dver \eps^4 \sum_{(i,j)\in \INNver} p(\ver_{i,j}, \ver_{i, j+1}) \vert \partial_2^{d,\eps} z^{\Delta}_{i,j}\vert^2. \label{spin:gamma:reformulation_vertical}
\end{equation}
Next, we define the set of spin fields with ferromagnetic Dirichlet boundary conditions as 
\[
\Sbound := \left\{ u \colon \Omega \cap\eps\Z^2 \rightarrow \mathbb{S}^1 \colon u(0, \cdot) = (0,1)^T \right\}.
\]

For sequences $(\eps_n)_n, (\delta_n^{hor})_n, (\dver_n)_n \subseteq (0,1)$, we will often write $E_n^{hor}$ and $E_n^{ver}$ instead of $E^{ver}_{\eps_n, \Delta_n}$ and $E^{hor}_{\eps_n, \Delta_n}$.
Additionally, we define the renormalized energy $\Hn\coloneqq \Hnhor + \Hnver \colon L^2(\Omega;\R^2) \to [0,+\infty]$ by 
\begin{equation}\label{eq: def Hn}
\Hn(w,z) \coloneqq \Hnhor(w) + \Hnver(z) \coloneqq \begin{cases}
\frac{E_n^{hor}(u, \Omega) + E_n^{ver}(u, \Omega) }{\sqrt{2} \eps_n  (\delta_n^{ver})^{3/2}}  &\text{if } \exists u \in \mathcal{S}_0(\Omega, \eps_n) \text{ that } \text{satisfies }\eqref{gamma:spins:periodic_boundary} \text{ and } \\ & \pc(T^{\Delta_n}(u)) = (w,z),  \\
    +\infty &\text{else.}
\end{cases}
\end{equation}
Lastly, set $\H: L^2(\Omega; \R^2) \to [0,+\infty]$ as 
\begin{equation}
    \H(w,z) = \begin{cases} \frac1\sigma\int_{\Omega} ( w^2 - \gamma^2)^2 + (z^2 - 1)^2  \, d\mathcal{L}^2 + \sigma \int_\Omega (\partial_1 w)^2 +(\partial_2 z)^2 \, d\mathcal{L}^2 & \text{if } (w,z) \in \dom(\H)\\
    + \infty &\text{else,}
    \end{cases}
\end{equation}
and its effective domain 
\[
\dom(\H)\coloneqq\left\{ (w,z) \in W^{1,2}(\Omega;\R^2) \colon \begin{array}{l}
     z(0, \cdot) = 0, \ \vert z(\cdot, 0) \vert = \vert z(\cdot, 1)\vert , \ \vert w(0, \cdot)\vert = \vert w(1, \cdot)\vert, \\
     \textup{and }\textup{curl}(w,z) = 0 
\end{array}  \right\}. 
\]

\subsection{Main Results}\label{sec:mainresults}

As a first result, we prove the following discrete-to-continuum $\Gamma$-convergence result for the renormalized energy.
\begin{theorem}[$\Gamma$-convergence and local compactness]
Let $(\eps_n)_n \subseteq (0, 1/2)$ and $ \ \Delta_n = (\delta_n^{hor}, \delta_n^{ver})$ satisfy 
\begin{equation}
\quad \eps_n \rightarrow 0, \quad  \Delta_n = (\dhor_n, \dver_n) \rightarrow (0,0), \quad \frac{\eps_n}{\sqrt{2\dver_n}} \rightarrow \sigma \in (0,\infty), \quad \text{ and } \quad \gamma_n^2 \coloneqq \frac{\dhor_n}{\dver_n} \rightarrow \gamma^2 \in [0,\infty).  
\end{equation}
Then the following statements are true:\label{thm:gamma}
\begin{itemize}
    \item[\textup{(i)}] Let $(w_n, z_n) \in L^2(\Omega;\R^2)$ be such that for all $n\in \N$  there holds $\Hn(w_n,z_n) \leq K$ for some constant $K > 0$. Then, there is a (not relabeled) subsequence $(w_n, z_n)$ which converges to some $(w,z) \in \textup{ dom }(\H)$ strongly in $L^2(\Omega;\R^2)$. 
    \item[\textup{(ii)}] If a sequence  $(w_n,z_n)_n \subseteq L^2(\Omega; \R^2)$ converges strongly in $L^2(\Omega;\R^2)$ to some $(w,z) \in L^2(\Omega;\R^2)$  then it holds 
    \[ \H(w,z) \leq  \liminf\limits_{n\rightarrow \infty} \Hn(w_n,z_n).\]
    \item[\textup{(iii)}]  For every $(w,z) \in L^2(\Omega;\R^2)$ there is a sequence $(w_n,z_n)_n \subseteq L^2(\Omega;\R^2)$ such that 
    \[ \limsup\limits_{n\rightarrow \infty} \Hn(w_n,z_n) \leq \H(w,z).\]
\end{itemize}
\end{theorem}
\begin{proof}
   Combining Proposition \ref{prop:lower_gamma_per} and Proposition \ref{gamma:spins:upper} yields the claim. 
\end{proof}
\begin{remark}
Recall that $H_n$ yields finite energy only on configurations satisfying the periodicity condition \eqref{gamma:spins:periodic_boundary}. Hence, by  \eqref{gamma:spins:energie_one} and \eqref{spin:gamma:reformulation_vertical}, minimizers of $H_n$ correspond to minimizers of the original energy $I_{\alpha, \eps}$ from \eqref{eq:J1J3model} under the periodicity constraint \eqref{gamma:spins:periodic_boundary}. We do not treat the $\Gamma$-limit of the original energy without the periodicity constraint.
\end{remark}
\begin{remark}\label{rem:marco}
    In \cite{B11-2019-CFO} a $\Gamma$-convergence result for the isotropic case $\delta_n=\delta_n^{hor}=\delta_n^{ver}$ is proven in the limit $\eps_n/(2\delta_n)^{1/2}\rightarrow 0$ as $n\rightarrow \infty $ with respect to strong convergence in $L^1_{loc}(\Omega;\R^2)$. The $\Gamma$-limiting functional is of (anisotropic) perimeter-type, 
    \[ H((w,z), \Omega) \coloneqq \frac43\left( \vert D_1 w\vert(\Omega) + \vert D_2 z\vert (\Omega) \right)\]
    for $(w,z) \in BV(\Omega;\R^2)$ with $(w,z)\in \{-1, 1\}^2$ and $\textup{curl}((w,z)) = 0$ in a distributional sense, and $H((w,z),\Omega)= \infty$ otherwise. In the Master's thesis \cite{FoMaster2018Discrete-to-Continuum}, a corresponding anisotropic result is proven.   
\end{remark}
\begin{remark}
    We note that one-dimensional analogs of Theorem \ref{thm:gamma} and Remark \ref{rem:marco} have been provided in \cite{B11-2015-CS}.
\end{remark}

Our second main result concerns the scaling law for the minimal discrete energy in the case that $\dver = \dhor$. 
In certain parameter regimes for $\eps$ and $\Delta$ the minimal energy of the discrete energy behaves essentially as the minimal energy of the continuum $\Gamma$-limit $H_{\sigma,1}$ which was (up to technical simplifications) studied in \cite{GiZw22}. In particular, in this regime minimizers are curl-free, and we prove matching upper and lower bounds for the minimal energy. 
On the other hand, we show that there is a complementary parameter regime in which minimizers necessarily are not curl-free and vortices occur (for a definition we refer to \eqref{def:vortices}), see Figure \ref{fig:vortex} for a sketch of such a profile. In this regime, we provide an upper bound but we do not yet have a matching lower bound in the full parameter regime.  
	\begin{theorem}[Scaling law of the discrete model]\label{thm:scaling} Let $s\colon (0,1/2)\times(0,1) \rightarrow (0,\infty)$ be given by
		\[ s(\eps,\delta) \coloneqq \min\left\{ \delta^2,\  \eps\delta^{3/2} \left( \left\vert \ln\left( \frac{\eps}{\delta^{1/2}}\right) \right\vert +1\right),\  \eps\delta^{1/2}\right\}. \quad 
  \] 
	The following statements are true for $\delta := \dver = \dhor$.
	\begin{itemize}
	    \item[(i)] There is a constant $C > 0$ such that for all $(\eps, \delta) \in (0,1/2) \times (0,1)$ there holds
	    \[ \min \left\{ \Eiso(u; \Omega) \colon u \in \Sbound \right\} \leq C s(\eps,\delta). \]
	    \item[(ii)] There is $\overline{\delta}\in (0,1)$ and a constant $c > 0$, depending only on $\overline{\delta},$ such that for all\\ $(\eps, \delta) \in (0,1/2) \times [\overline{\delta}, 1)$ it holds
	\begin{equation*}
	   c \eps\delta^{1/2} \leq \min \left\{ \Eiso(u; \Omega) \colon u \in \Sbound \right\},
	\end{equation*}
	and for all $\delta\in (0, \overline{\delta})$ and $\eps \in [\delta, 1/2)$ it holds
 \[ c \min\left\{ \delta^2, \eps\delta^{3/2} \left( \left\vert \ln\left( \frac{\eps}{\delta^{1/2}}\right) \right\vert +1\right)\right\} \leq  \min \left\{ \Eiso(u; \Omega) \colon u \in \Sbound \right\}. \] 
	\end{itemize}
   \end{theorem}

\label{chapter:scaling}

   \begin{proof}
       The result follows directly from Proposition \ref{lem:upper} and Proposition \ref{prop:lower}. 
   \end{proof}
 
   \begin{remark}\label{rem:scaling_function_regimes} 
   Let $\textup{e} \coloneqq \exp(1)$. Then for $(\eps, \delta)\in (0,1/2)\times (0,1)$, we have
   \begin{equation}
       s(\eps,\delta) = \begin{cases}
       \delta^2 \quad & \textup{ if } \delta^{1/2} \leq \eps,\\
       \eps\delta^{3/2} \left( \left\vert \ln\left( \frac{\eps}{\delta^{1/2}}\right) \right\vert +1\right) \quad & \textup{ if } \textup{e}\cdot \delta^{1/2} \exp(-1/\delta) \leq \eps\leq \delta^{1/2},\\
       \eps\delta^{1/2} & \textup{ else}.
       \end{cases}\nonumber 
   \end{equation} 
   \end{remark}
   \begin{remark}
   The decisive step in the proof of the lower bound for $\delta \in (0, \overline{\delta})$ (see Section \ref{sec:lower} and (ii) of the Theorem above) is to show that the claimed lower bound holds for all spin fields $u\in \Sbound$ that are vortex-free. It follows from Lemma \ref{lem:vortices-estimates} below that the number of vortices $\mathcal{H}^0(V)$ of $u$ can be bounded by
    \[ \mathcal{H}^0(V)\leq \frac{\Eiso(u, \Omega)}{\eps^2}. 
    \] 
    Hence, it turns out that the claimed lower bound follows for a larger set of parameters than stated in Theorem \ref{thm:scaling} (ii), see Remark \ref{rem:lowerbound2} for a more quantitative statement.
   \end{remark}

\section{Preliminaries}\label{sec:preliminaries}
In the first part of this section we prove that  the space
\begin{equation}
    \label{gamma:laplace_sobolev} \mathcal{W} \coloneqq \left\{ u\in W^{1,2}(\Omega) \colon \, \partial_{11}u, \partial_{22}u\in L^2(\Omega) \textup{ and }\partial_{12} u, \partial_{21}u \in L^2_{loc}(\Omega)\right\} 
\end{equation}
endowed with the norm
\[ \Vert u\Vert_{\mathcal{W}} =\Vert u \Vert_{W^{1,2}(\Omega)} + \Vert \partial_{11} u\Vert_{L^2(\Omega)}+\Vert \partial_{22} u\Vert_{L^2(\Omega)} \]
is continuously embedded in the Sobolev space $W^{2,2}(\Omega)$.\\ In the second part of this section, we prove several properties of the function $p$ (see \eqref{def:p}). \\

\quad\textbf{Continuous embedding of $\mathcal{W}$ into $W^{2,2}(\Omega)$.} 
For an open, bounded Lipschitz set $\calD \subseteq \R^2$, we set 
\[
\calW_0(\calD):=\left\{u\in W^{1,2}_0(\calD): \partial_{11}u, \partial_{22}u\in L^2(\calD) \textup{ and }\partial_{12} u, \partial_{21}u \in L^2_{loc}(\calD)\right\} 
\]
equipped with the norm $\| u\|_{\calW_0( \calD)} \coloneqq \Vert u \Vert_{W^{1,2}(\calD)} + \Vert \partial_{11} u\Vert_{L^2(\calD)}+\Vert \partial_{22} u\Vert_{L^2(\calD)}$.
\begin{theorem}\label{th:embedding}
 The space $\calW$ defined in \eqref{gamma:laplace_sobolev} is continuously embedded in $W^{2,2}(\Omega)$ i.e., there exists a constant $C>0$ such that for all $u\in \calW$ there holds
 \[\|u\|_{W^{2,2}(\Omega)}\leq C\|u\|_{\calW}. \]
\end{theorem}
 \begin{proof}
 The proof relies on elliptic regularity and the extension result in Lemma \ref{prop:extension}.
 Precisely, for $u\in \calW$ by Proposition \ref{prop:extension} below, there exists an extension  $Eu\in\calW_0(B_{3/2}((1/2,1/2)))$ such that 
 \begin{equation}
 \label{eq:propext}
\|Eu\|_{\calW_0( B_{3/2}(\frac 12,\frac 12))}\leq C\|u\|_{\calW}
 \end{equation}
 with a constant $C>0$ independent of $u$. Since by elliptic regularity, e.g.~\cite[Theoreme IX.25 ]{Brezis:05},
 \begin{equation}\label{eq:ellipticreg}
 \|Eu\|_{W^{2,2}( B_{3/2}(\frac 12,\frac 12))}\leq C \|Eu\|_{\calW_0( B_{3/2}(\frac 12,\frac 12))},
 \end{equation}
 we obtain combining  \eqref{eq:propext} and \eqref{eq:ellipticreg}
 \begin{equation}
 \|u\|_{W^{2,2}(\Omega)}\leq \|Eu\|_{W^{2,2}( B_{3/2}(\frac 12,\frac 12))}
 \leq C\|Eu\|_{\calW_0( B_{3/2}(\frac 12,\frac 12))}
 \leq C \|u\|_{\calW}.
 \end{equation}
 \end{proof}
 
 \begin{lemma}
 \label{prop:extension}
 There is a continuous linear extension operator $E:\calW\to \calW_0(B_{3/2}((1/2,1/2))$ i.e., there is a constant $C>0$ such that for all  $u\in \calW$ we have $E(u)_{|\Omega}=u$ and
 \[\|Eu\|_{\calW_0( B_{3/2}(\frac 12,\frac 12))}\leq C\|u\|_{\calW}.\]
 \end{lemma}
 \begin{proof}
Let us first note that due to the specific geometry of $\Omega$, for $u \in \calW$ the partial derivatives $\partial_1 u$ and $\partial_2 u$ have traces in $L^2(\{0,1\} \times (0,1))$ and $L^2((0,1) \times \{0,1\})$, respectively.  

 \begin{figure}[h]
 \centering 
\includegraphics[scale = 0.6]{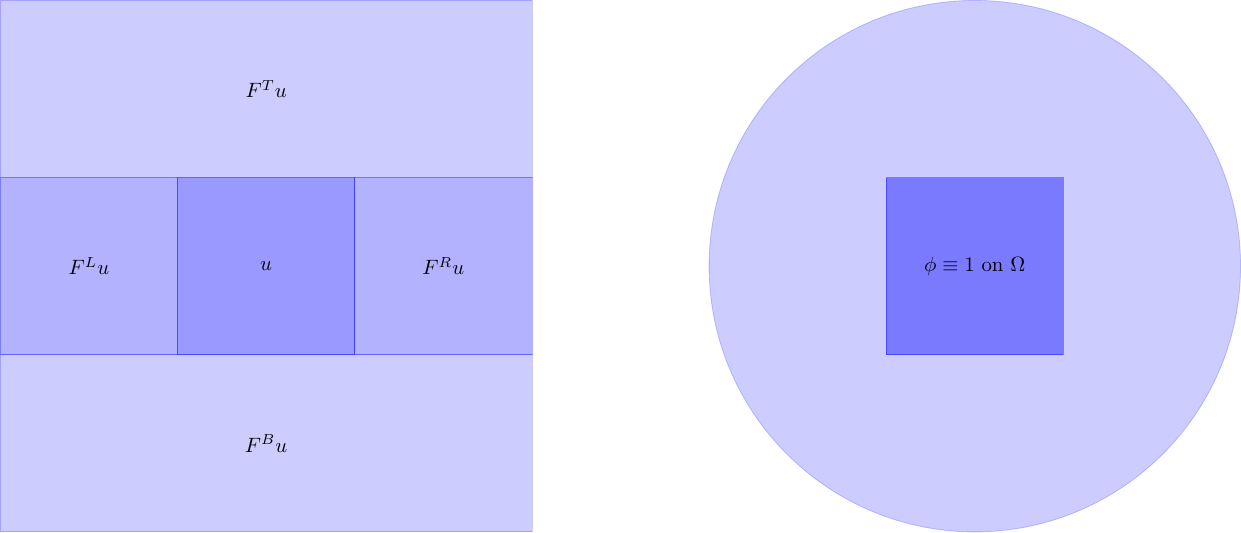}
\caption{Left: Sketch of the construction of the extension operator $F$. We first extend a function $u\in \calW (\Omega)$ by a reflection argument from the dark-blue square to the left and right blue squares by $F^Lu$ and $F^Ru$. In a second step, the left and right extended function is extended similarly as $F^Tu$ and $F^Bu$ to the the upper and lower light-blue rectangles. Right: Sketch of $\phi\in C_c^\infty( B_{3/2}((1/2, 1/2))$ with $\phi_{\vert \Omega} = 1$ and $\phi = 0$ on $\partial B_{3/2}((1/2, 1/2)).$ }
\label{fig:extension_domain}
 \end{figure}
In order to define the extension operator we will use the construction from \cite[proof of Lemma 1 in Section 2.9.1]{Tri:78}. 
First, let $(c_1^L, c_2^L, c_3^L) \in \R^3$ be the unique solution to the linear system
 \begin{equation} \label{linear-system}
 \sum\limits_{i=1}^3 c_i^L = 1, \ -\sum\limits_{i=1}^3 \frac{i}{3} c_i^L = 1,\quad \textup{ and } \quad \sum\limits_{i= 1}^3 \left( \frac{i}{3}\right)^2 c_i^L = 1.
 \end{equation}
 For $u \in \calW$ we define for $(x_1,x_2)\in (-1,1)\times(0,1)$
\begin{equation}
   F^L u (x_1,x_2)=\begin{cases}
   u(x_1,x_2)&\textup{ if }(x_1,x_2)\in \Omega,\\
    \sum\limits_{i=1}^3 c_i^L u\left( -\frac{i}{3} x_1, x_2\right) & \textup{ if } (x_1,x_2) \in (-1, 0]\times (0,1).
    \end{cases}
\end{equation} 
Then $F^L u \in W^{1,2}((-1,1) \times (0,1))$ and $\partial_{22} F^L u \in L^2((-1,1)\times(0,1))$. Using the existence of the trace of $\partial_1 u$ on $\{0\} \times (0,1)$ one can show in addition that $\partial_{11} F^Lu \in L^2((-1,1)\times (0,1))$. 
Moreover, there is a constant $C>0$ such that for all $u\in\calW$
\begin{equation}\label{eq:extensionTriebel}
\|F^L u\|_{W^{1,2}((-1,1) \times (0,1)} + \|\partial_{11} F^L u\|_{L^2((-1,1) \times (0,1))} + \|\partial_{22} F^L u\|_{L^2((-1,1) \times (0,1))} \leq C  
\|u\|_{\calW}.
\end{equation}
Analogously, we define a further extension $F^Ru$ to $(0,2)\times(0,1)$. 
We similarly extend the so-constructed function to $(-1,2)\times(-1,2)$ by an analogous construction in the vertical direction, see Figure \ref{fig:extension_domain}. 
In this way, we obtain for $u \in \calW$ an extension $Fu\in W^{1,2}((-1,2)^2)$ satisfying $\partial_{11} F u, \partial_{22} F u \in L^2((-1,2)^2))$ and for a universal $C>0$ 
\begin{equation}
\label{eq:contextF}
\|F u\|_{W^{1,2}((-1,2)^2)} + \|\partial_{11} F u\|_{L^2((-1,2)^2)} + \|\partial_{22} F_L u\|_{L^2((-1,2)^2)} \leq C  
\|u\|_{\calW}.
\end{equation} 
 Finally, let $\phi\in C^\infty_c(B_{3/2}((1/2,1/2)))$ be a smooth cut-off function with $\phi_{|\Omega}=1$ and extended by $0$ to $\R^2$. Then define  $E:\calW\to\calW_0(B_{3/2}(1/2,1/2))$ by $E(u):=F(u)\phi$.  
 By the product rule we obtain
 \begin{equation}\begin{split}
 &\|E(u)\|_{\calW_0(B_{3/2}(1/2,1/2))} \\
 \leq &C\left( \|F u\|_{W^{1,2}((-1,2)^2)} + \|\partial_{11} F u\|_{L^2((-1,2)^2)} + \|\partial_{22} F^L u\|_{L^2((-1,2)^2)} \right)\|\phi\|_{W^{2,\infty}(\calD)}\\
 \leq &C\|u\|_{\calW}\|\phi\|_{W^{2,\infty}(\calD)}.\end{split}
 \end{equation}
 \end{proof}
 
 \quad \textbf{Properties of the function $p$.} 
 In this section we collect some results on the function $p$ defined in \eqref{def:p}. 
 First, we discuss the behavior of $p$ for small angles. 
 \begin{lemma}\label{lem:p_convergence} The function $p$ is continuous in $(0,0)$, i.e., for every $\eps > 0$ there is $\eta >0$ such that for all $(\theta_1,\theta_2)$ with $\max\{\vert \theta_1\vert, \vert \theta_2\vert\} \leq \eta$ there holds $\vert 1-p(\theta_1, \theta_2) \vert \leq \eps $.  Moreover, it holds 
\[
p(\theta_1, \theta_2)\geq \frac12 \text{ for all } (\theta_1, \theta_2)\in [-\pi/8, \pi/8]^2 \quad \text{ and } \quad  p(\theta_1, \theta_2)\leq 3/2 \text{ for all }(\theta_1,\theta_2)\in [-\pi/2, \pi/2]^2.
\] Eventually, there is a constant $c_p\in \R$ such that 
\[ p(\theta_1, \theta_2) \geq c_p \quad \text{ for all } (\theta_1, \theta_2) \in [-\pi, \pi]^2.\]
\end{lemma}
\begin{proof}
The proof can be found in \cite[Lemma 4.2]{B11-2019-CFO}. 
We sketch it here for the reader's convenience.
Applying several trigonometric identities, one obtains  
\begin{equation}\label{gamma:p:reformulated}
    p(\theta_1, \theta_2) = \begin{cases}
    \frac{\cos^2((\theta_2 -\theta_1)/4) \cos(\theta_1 + \theta_2)}{\cos^2((\theta_1 +\theta_2)/4)} \quad & \textup{ if } \theta_1 \neq \theta_2,\\
    1 & \textup{ else.}
    \end{cases}
\end{equation}
Since $x\mapsto \cos(x)$ and $x\mapsto 1/\cos(x)$ are continuous at $x = 0$, we conclude the continuity of $p$ at $(0,0)$. 
Moreover, we find for all $(\theta_1, \theta_2) \in [-\pi/8, \pi/8]^2$ with $\theta_1 \neq \theta_2$ the estimate 
\[p(\theta_1, \theta_2) \geq \cos^2((\theta_1-\theta_2)/4) \cos(\theta_1 +\theta_2)\geq \cos^2(\pi/16) \cos(\pi/4) \geq \frac12.  \]
Similarly, one estimates for $(\theta_1, \theta_2)\in [-\pi/2, \pi/2]^2$
\[  p(\theta_1,\theta_2) \leq \frac{1}{\cos(\pi/4)}\leq \frac32. \] 
Eventually, we prove the claimed lower bound on $[-\pi,\pi]^2$. Define the function \[g: [-\pi,\pi]^2 \setminus \{(\pi,\pi),(-\pi,-\pi)\} \to \R \text{ as } g(\theta_1,\theta_2) = \frac{\cos^2((\theta_2 -\theta_1)/4) \cos(\theta_1 + \theta_2)}{\cos^2((\theta_1 +\theta_2)/4)}.\] 
As $g$ is continuous in $[-\pi,\pi]^2 \setminus \{(\pi,\pi),(-\pi,-\pi)\}$ and $\liminf_{(\theta_1,\theta_2) \to (\pm \pi,\pm \pi)} g(\theta_1,\theta_2) > 0$ in light of \eqref{gamma:p:reformulated} the lower bound for $p$ follows.
\end{proof}
Next, we estimate the second contribution in the energies $\Ehor$ and $\Ever$ (see \eqref{gamma:spins:energie_one} and \eqref{spin:gamma:reformulation_vertical}).
\begin{lemma} \label{lem:lower_bound_q} Let $\eps \in (0,1/2)$, $\Delta \in (0,1)^2$, $u \in \Sbound$ and $w^{\Delta}, z^{\Delta}$ the associated chirality order parameters. In addition, assume that \[\max\{\arccos(1-\dhor), \arccos(1-\dver)\} \leq \pi/16.\] 
Then it holds
\begin{equation}
\dhor\eps^4 \sum\limits_{(i,j) \in \INNhor} p(\hor_{i+1,j}, \hor_{i,j})\vert \partial_1^{d,\eps}w^{\Delta}_{i,j}\vert^2 \geq \frac12 \dhor \eps^4 \sum\limits_{(i,j) \in \INNhor} \left\vert \partial_1^{d,\eps} w^{\Delta}_{i,j} \right\vert^2 - 2\pi^2(1-2c_p) \eps^2 \mathcal{H}^0(A^{hor})\nonumber 
\end{equation}
where $A^{hor} \coloneqq \{(i,j) \in \mathcal{I} \colon \max\{\vert \hor_{i+1,j}\vert, \vert \hor_{i,j}\vert  \}\geq \pi/8\} $ and $c_p\in \R$ is the constant from the lower bound of $p$ in Lemma \ref{lem:p_convergence}. The analogous statement holds for the vertical contribution, where one has to replace the set $A^{hor}$ by the corresponding set for the vertical angular velocities.
\end{lemma}
\begin{proof}
We follow the lines of \cite[Lemma 4.3]{B11-2019-CFO} and only present a finer analysis regarding the dependence on the parameters. Due to Lemma \ref{lem:p_convergence}, we have  \begin{equation}\label{proof:lower_bound_p:choice_alpha}
    p(\theta_1, \theta_2 ) \geq \frac12 \textup{ for all } \max\{ \vert \theta_1\vert, \vert \theta_2\vert \} \leq \frac\pi8 \quad \textup{ and } \quad p(\theta_1, \theta_2) \geq c_p \textup{ for all } (\theta_1, \theta_2) \in [-\pi, \pi]^2.
\end{equation}
For simplicity, we write $w$ for $w^{\Delta}$ and estimate
\begin{equation}
\begin{split}
& \dhor \eps^4 \sum\limits_{(i,j) \in \INNhor} \left\vert \partial_1^{d,\eps} w_{i,j} \right\vert^2 = \dhor \eps^4 \sum\limits_{(i,j) \in \INNhor\setminus A^{hor}} \left\vert \partial_1^{d,\eps} w_{i,j}\right\vert^2 + \dhor \eps^4 \sum\limits_{(i,j) \in A^{hor}} \left\vert \partial_1^{d,\eps} w_{i,j} \right\vert^2  \\
&\leq 2 \dhor \eps^4 \sum\limits_{(i,j) \in \INNhor\setminus A^{hor}} p(\hor_{i+1,j}, \hor_{i,j})\vert \partial_1^{d,\eps}w_{i,j}\vert^2+   \dhor \eps^4 \sum\limits_{(i,j) \in A^{hor}} (1- 2p(\hor_{i+1,j},\hor_{i,j})) \left\vert \partial_1^{d,\eps} w_{i,j} \right\vert^2 \\
&\phantom{\leq} + 2\dhor \eps^4 \sum\limits_{(i,j) \in A^{hor}} p(\hor_{i+1,j},\hor_{i,j}) \left\vert \partial_1^{d,\eps} w_{i,j} \right\vert^2 \\
& \leq 2 \dhor \eps^4 \sum\limits_{(i,j) \in \INNhor} p(\hor_{i+1,j},\hor_{i,j}) \left\vert \partial_1^{d,\eps} w_{i,j} \right\vert^2 + \dhor \eps^4 \sum\limits_{(i,j) \in A^{hor}} (1- 2c_p) \left\vert \partial_1^{d,\eps} w_{i,j} \right\vert^2\\
& \leq 2 \dhor \eps^4 \sum\limits_{(i,j) \in \INNhor} p(\hor_{i+1,j},\hor_{i,j}) \left\vert \partial_1^{d,\eps} w_{i,j} \right\vert^2 + 2\pi^2(1-2c_p) \eps^2 \mathcal{H}^0(A^{hor}),
\end{split}\label{gamma:p:surface_estimate}
\end{equation}
where we used \eqref{proof:lower_bound_p:choice_alpha} and that for $(i,j) \in \INNhor$
\[\vert \partial_1^{d, \eps} w_{i,j} \vert = \sqrt{\frac{2}{\dhor}}  \left\vert  \frac{\sin(\hor_{i+1, j}/2) - \sin(\hor_{i,j}/2)}{\eps}\right\vert\leq \sqrt{\frac{2}{\dhor}} \frac{1}{2\eps} \left\vert  \hor_{i+1, j} - \hor_{i,j}\right\vert\leq \sqrt{\frac{2}{\dhor}}\frac\pi\eps.
\]  
Arguing analogously for the vertical energy contribution yields the claim.
\end{proof}

Eventually, we will bound the term $\mathcal{H}^0(A^{hor})$ as appearing in the previous Lemma. As outlined in \cite{B11-2019-CFO} (see also Section \ref{sec:spin_intr}), this will be helpful to control the number of vortices, see Remark \ref{rem: vortex} below.

\begin{lemma} \label{lem:vortices-estimates} Let $\eps \in(0, 1/2)$, $\Delta \in (0,1)^2$, $u\in \mathcal{S}_0(\Omega,\eps)$ and $(\hor, \ver)\colon \Omega \cap\eps\Z^2 \rightarrow [-\pi,\pi)^2$ the associated angular velocity field. 
Furthermore, let 
$\beta^{hor/ver} \in (2\arccos(1-\delta^{hor/ver}), \pi)$,
\begin{equation*}
A_{\beta^{hor}}^{hor} \coloneqq \left\{ (i,j) \in \INhor \colon \vert\hor_{i,j}\vert \geq \beta^{hor} \right\},\quad \textup{ and } \quad A_{\beta^{ver}}^{ver} \coloneqq \left\{ (i,j) \in \INver \colon \vert\ver_{i,j}\vert  \geq \beta^{ver} \right\}.
\end{equation*}
Then it holds 
\begin{align}
    \label{ineq:number-of-lattice-points}
    &\eps^2 \mathcal{H}^0(A_{\beta^{hor}}^{hor})  \leq \frac{2\Ehor(u)}{(1-\dhor)^2 (1-\dhor - \cos(\beta^{hor}))^2} \\ \text{ and } \qquad &\eps^2 \mathcal{H}^0(A_{\beta^{ver}}^{ver})  \leq \frac{2\Ever(u)}{(1-\dver)^2 (1-\dver - \cos(\beta^{ver}))^2}.
\end{align}
\end{lemma}
\begin{proof} We only show the first estimate of the assertion, the second one follows analogously. We follow the lines of the proof of  \cite[Lemma 4.1]{B11-2019-CFO}. 
Note that for any unit vectors $u, v, w \in \mathbb{S}^1$ it holds 
\[\left\vert \left\vert u -2(1-\dhor)v\right\vert +1\right\vert \leq 2 + 2(1-\dhor). \] 
Hence, by the triangle inequality
\begin{equation}
     \vert u - 2(1-\dhor)v + w\vert^2 \geq (\vert u - 2(1-\dhor)v\vert -1)^2 \geq \frac{1}{(2 +2(1-\dhor))^2} \left(\vert u- 2(1-\dhor)v\vert^2 -1\right)^2.
 \end{equation}
Next, let $\theta_{u,v}, \theta_{v,w} \in [-\pi, \pi)$ be such that $\cos(\theta_{u,v}) = u\cdot v$ and $\cos(\theta_{v,w}) = v\cdot w$.
Then observe that
 \begin{equation}\begin{split}
    & \frac{1}{(2 +2(1-\dhor))^2} \left(\vert u- 2(1-\dhor)v\vert^2 -1\right)^2\nonumber\\
    \geq &\frac{1}{16} ( 1 -4(1-\dhor)\cos(\theta_{u,v}) + 4(1-\dhor)^2 -1)^2 = (1-\dhor)^2 (1-\dhor- \cos(\theta_{u,v}))^2.\end{split}
\end{equation}
Swapping the roles of $u$ and $v$ we find
\begin{equation*}
    \vert u - 2(1- \dhor)v + w \vert^2 \geq (1-\dhor)^2\max \left\{ (1-\dhor- \cos(\theta_{u,v}))^2, (1-\dhor- \cos(\theta_{v,w}))^2\right\}. 
\end{equation*}
Eventually, we conclude 
\begin{equation*}\begin{split}
\Ehor(u) = &\frac{\eps^2}{2} \sum\limits_{(i,j)\in \INNhor} \vert u_{i,j} -2(1-\dhor)u_{i+1,j} + u_{i+2,j}\vert^2 \\
  \geq &\frac{(1-\dhor)^2}{2} \eps^2\sum\limits_{(i,j)\in A_{\beta^{hor}}^{hor}}  (1-\dhor -\cos(\beta^{hor}))^2 \label{eq: estimate energy reformulation}\\
=& \frac{(1-\dhor)^2}{2} (1-\dhor - \cos(\beta^{hor}))^2\eps^2 \mathcal{H}^0(A_{\beta^{hor}}^{hor}).\end{split} 
\end{equation*} 
\end{proof}
\begin{remark}\label{rem: vortex}
For $\delta = \dhor = \dver \in(0, 1/2)$ and $\beta^{hor} = \beta^{ver} = \pi/2$ we obtain by Lemma \ref{lem:vortices-estimates} for the set of vortices (recall \eqref{def:vortices})
\[ \eps^2 \mathcal{H}^0(V) \leq \eps^2\left( \mathcal{H}^0(A_{\frac\pi2}^{hor})+ \mathcal{H}^0(A_{\frac\pi2}^{ver})\right)  \leq 64 \Eiso(u,\Omega)\]
for any spin field $u\colon \Omega \cap\eps\Z^2 \rightarrow \mathbb{S}^1.$\label{rem:bound_vortices}
\end{remark}
\begin{remark}\label{rem:beta_special}
    Lemma \ref{lem:vortices-estimates} remains true for $\beta^{hor} = 2\pi/3$ or $\beta^{ver} = 2\pi /3$. In this case, we have 
    \[ \eps^2 \mathcal{H}^0(A_{\frac{2\pi}{3}}^{hor}) \leq \frac{8\Ehor(u)}{(1-\dhor)^2} \text{\qquad or \qquad} \eps^2 \mathcal{H}^0(A_{\frac{2\pi}{3}}^{ver}) \leq \frac{8\Ever(u)}{(1-\dver)^2}.\]
\end{remark}

\section{Proof of the $\Gamma$-convergence result}
\label{sec:Gamma}

In this chapter, we prove the $\Gamma$-convergence result at the heli-/ferromagnetic transition point (see Theorem \ref{thm:gamma}).
As outlined above, a main difficulty comes from the discrete curl-condition, and the need to control the number of vortices. 
We start by proving the local compactness and liminf inequality in Section \ref{sec:liminf}, and construct a recovery sequence in Section \ref{sec:upperbound} 

\subsection{Local compactness and lower bound}\label{sec:liminf}
In this section, we prove the local compactness and the liminf inequality stated in (i) and (ii) of Theorem \ref{thm:gamma}. Some steps are inspired by \cite{B11-2019-CFO} where a different parameter regime is considered. We start by proving that sequences $(w_n,z_n)$ that are equibounded in energy can be associated to a sequence of gradients of equibounded $W^{2,2}$-functions.

\begin{lemma}[Existence of bounded potentials]\label{lem:existenc_potentials}
 Let $\eps_n \in (0, 1/2), \ \Delta_n = (\delta_n^{hor}, \delta_n^{ver})\in (0,1)^2$ satisfy 
\begin{equation}
\quad \eps_n \rightarrow 0, \quad  \Delta_n  \rightarrow (0,0), \quad \frac{\eps_n}{\sqrt{2\dver_n}} \rightarrow \sigma \in (0,\infty), \quad \text{ and }\quad\gamma_n^2 \coloneqq  \frac{\dhor_n}{\dver_n} \rightarrow \gamma^2 \in [0,\infty). 
\end{equation}
Let $(w_n, z_n) \in L^2(\Omega;\R^2)$ for $n\in \N.$ If there is $K> 0$ such that 
\begin{equation}
    \Hn(w_n, z_n) \leq K \quad \text{ for all } n\in \N,
\end{equation}
then there is a sequence of associated angular velocity fields $(\hor_n,\ver_n) \colon \Omega \cap \eps_n\Z^2 \rightarrow [-\pi, \pi)^2$ and a subsequence such that the following two statements are true:
\begin{enumerate}
    \item[\textup{(i)}] There is a sequence $u_n\in W^{2,2}(\Omega)$ with $u_n(0,\cdot) \equiv 0$ such that for all $n\in\N$, 
\begin{equation}
    \begin{split}
        &\quad \quad \quad \quad \quad \nabla u_n =  ((2\delta_n^{ver})^{-1/2}\pa(\theta^{ hor}_n),(2\delta_n^{ver})^{-1/2}\pa(\theta^{ ver}_n)),\\
         &\partial_{11} u_n = (2\delta_n^{ver})^{-1/2} \partial_1^{d,\eps_n} \hor_n\  \text{ and }\  \partial_{22} u_n = (2\delta_n^{ver})^{-1/2}\partial_2^{d,\eps_n} \ver_n \quad \text{ on }\Omega \cap\eps_n\Z^2.
    \end{split} 
\end{equation}
\item[\textup{(ii)}] There is $u\in W^{2,2}(\Omega)$ such that $u_n \rightarrow u$ strongly in $W^{1,2}(\Omega)$ and $u_n \rightharpoonup u$ weakly in $W^{2,2}(\Omega)$ as $n\rightarrow \infty.$
\end{enumerate}
\end{lemma}
\begin{proof}
We divide the proof into several steps. First, by definition of the energy we may associate the functions $(w_n, z_n)$ to angular velocity fields $(\hor_n, \ver_n) \colon \Omega \cap \eps_n \Z^2 \rightarrow [-\pi, \pi)^2$ of appropriate spin fields. 
Next, we show that the number of vortices of the fields $(\hor_n,\ver_n)$ vanishes as $n\rightarrow \infty$. Denoting by $\pa(\theta_n^{hor})$ and $ \pa(\theta_n^{ ver})$ the piecewise affine extensions of the discrete functions $\hor_n$ and $\ver_n$, c.f.~\eqref{def:gamma:triangulation}, we show that there is a subsequence such that there is a sequence of potentials $u_n \in \mathcal{W}(\Omega)$ with $\nabla u_n = ((2\delta_n^{ver})^{-1/2}\pa(\hor_n), (2\delta_n^{ver})^{-1/2}\pa(\ver_n))$, $u_n(0,\cdot)\equiv 0$ for $n\in \N$, and being bounded in $\mathcal{W}(\Omega)$, see \eqref{gamma:laplace_sobolev}. Due to Theorem  \ref{th:embedding}, the boundedness of $(u_n)_n$ in $\mathcal{W}(\Omega)$ also implies the boundedness in $W^{2,2}(\Omega)$. Due to the compact embedding of $W^{2,2}(\Omega)$ into $W^{1,2}(\Omega)$, 
there is a subsequence converging weakly in $W^{2,2}(\Omega)$ and strongly in $W^{1,2}(\Omega).$  \\

\quad \emph{\textbf{Step 1. Set-up.}} Since $\Hn(w_n, z_n) \leq K <\infty $ for all $n \in \N$, by definition of $\Hn$ there exists a spin field $v_n \in \Sboundn$ that satisfies \eqref{gamma:spins:periodic_boundary} and $\pc (\T(v_n)) = ( w_n, z_n)$.
Furthermore, let $(\hor_n, \ver_n) \colon \Omega \cap \eps_n \Z^2 \rightarrow [-\pi, \pi)^2$ be the associated angular velocity field. 
In the following, we will identify the corresponding discrete chiralities $w^{\Delta_n}(v_n), z^{\Delta_n}(v_n): \Omega \cap \eps_n \Z^2 \to \R$ as defined in \eqref{def:chirality_order_parameters} with the functions $\gamma_n^{-1} w_n$ and $z_n$.
Moreover, we define the set of vortices 
\[
V_n = \left\{  (i,j) \in \INNhor\cap \INNver \colon \vert \eps \curl((\hor_n,\ver_n)_{i,j})\vert = 2\pi \right\}.
\]

\quad \emph{\textbf{Step 2. Existence of a curl-free subsequence of $\bm{(\hor_n, \ver_n)}$.}} First note that if \\ $(i,j)\in \INNhor\cap\INNver\cap V_n$ then $\max\{\vert \hor_{n,i,j}\vert, \vert \hor_{n, i,j+1}\vert, \vert \ver_{n, i,j} \vert , \vert \ver_{n, i+1,j}\vert  \}\geq \pi/2$. It follows that
\begin{equation} V_n\subseteq \left\{ (i,j) \in \INNhor\colon \max\{\vert \hor_{i,j}\vert, \vert \hor_{i,j+1}\vert\} \geq\frac\pi8\right\} \bigcup \left\{ (i,j)\in \INNver\colon \max\{ \vert \ver_{i,j} \vert , \vert \ver_{i+1,j}\vert \}\geq \frac\pi8\right\} =: \mathcal{A}_n, \label{est:vortices} \end{equation}
Hence, we may estimate
\[ \mathcal{H}^0(V_n) \leq \mathcal{H}^0\left(\left\{(i,j)\in \INNhor \colon \vert \hor_{n, i, j}\vert\geq \frac\pi8 \right\} \right) +  \mathcal{H}^0\left(\left\{(i,j)\in \INNver\colon \vert \ver_{n, i, j}\vert\geq \frac\pi8 \right\} \right). \] 
In turn, Lemma \ref{lem:vortices-estimates} and $\cos(\pi/8) \leq 19/20$ imply 
\begin{equation}
\begin{split}
   \max\{ \mathcal{H}^0(V_n), \mathcal{H}^0(\mathcal{A}_n)\}& \leq   \frac{(2\delta_n^{ver})^{3/2}}{\eps_n}\left(\frac{H^{hor}_n(w_n)}{(1-\delta_n^{hor})^2 (1-\delta_n^{hor}-\frac{19}{20})^2 } + \frac{H_n^{ver}(z_n)}{(1-\delta_n^{ver})^2 (1-\delta_n^{ver}-\frac{19}{20})^2 } \right)  \\
    &\leq  \frac{(2\delta_n^{ver})^{3/2}}{\eps_n}\left(\frac{K}{(1-\delta_n^{hor})^2 (1-\delta_n^{hor}-\frac{19}{20})^2 } + \frac{K}{(1-\delta_n^{ver})^2 (1-\delta_n^{ver}-\frac{19}{20})^2 } \right), \end{split}\label{gamma:lem_existence_pot:vortices}
\end{equation}
where $\Hnhor$ and $\Hnver$ are the renormalized energies defined in \eqref{eq: def Hn}.
The right-hand side of \eqref{gamma:lem_existence_pot:vortices} converges to zero as $n\rightarrow \infty$. Hence, there is a curl-free (not relabeled) subsequence $(\hor_n,\ver_n)$ that satisfies $\mathcal{H}^0(\mathcal{A}_n) = 0$ for all $n\in \N$. From now on, we will argue along this subsequence.\\

\quad \emph{\textbf{Step 3. Existence of potentials $\bm{u_n}$. }}
Let $\pa(\hor_n)$ and $\pa(\ver_n)$ be the piecewise affine extensions of the discrete functions $\hor_n$ and $\ver_n$ introduced in \eqref{def:gamma:triangulation}, see Figure \ref{fig:construction_triangular} for an illustration. 
\begin{figure}[h]
    \centering
    \includegraphics[scale = 0.4]{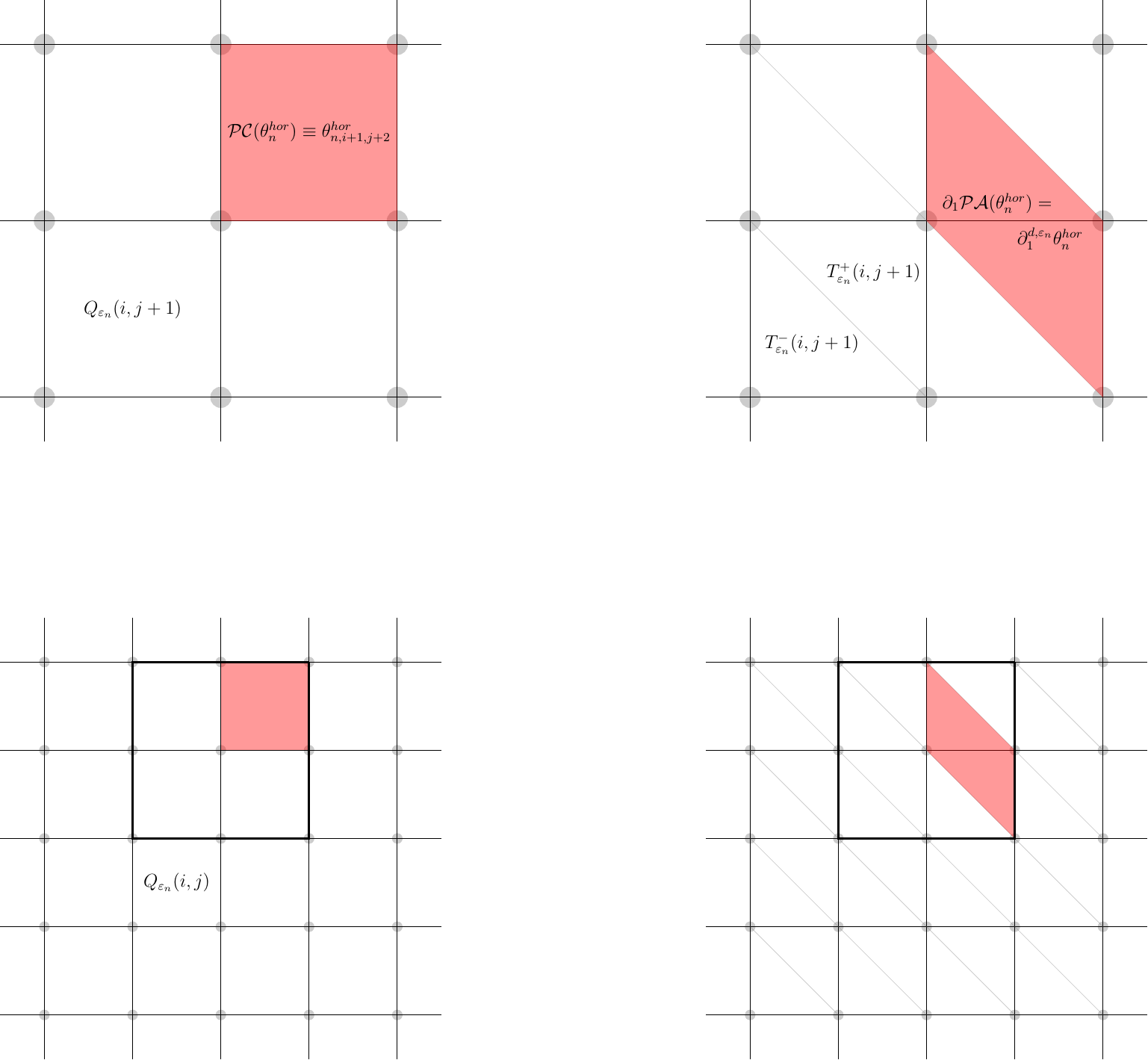}
    \caption{Left: Sketch of the quadratic $\eps_n$-lattice with the associated constant extension $(\pc(\hor_n), \pc(\hor_n))$ of $(\hor_n, \ver_n)$ to $\Omega$. Right: Sketch of the piecewise affine, continuous extension $(\pa(\hor_n), \pa(\ver_n))$ of $(\hor_n, \ver_n)$ to $\Omega$ with respect to the $\eps_n$-triangulation. }
\label{fig:construction_triangular}
\end{figure}
By the piecewise affine structure of $(\pa(\hor_n), \pa(\ver_n))$ it holds $(\pa(\hor_n), \pa(\ver_n)) \in W^{1,2}(\Omega;\R^2)$.
By step 2 we may assume that the sequence $(\hor_n,\ver_n)$ is curl-free (in a discrete sense). 
This carries over to $(\pa(\ver_n), \pa(\ver_n))$ due to the definition of the functions $\pa(\hor_n)$ and $\pa(\ver_n)$, and 
\[\partial_2 \pa(\hor_n) - \partial_1\pa(\ver_n)=\textup{curl}^{d,\eps_n} ((\hor_n,\ver_n)) = 0\quad \textup{ on } \Omega \cap \eps_n \Z^2 .\] 
Since $\Omega \subseteq \R^2$ is simply-connected, there exist scalar functions $u_n\in W^{2,2}(\Omega)$ with \[\nabla u_n = ((2\delta_n^{ver})^{-1/2}\pa(\hor_n), (2\delta_n^{ver})^{-1/2}\pa(\ver_n))\]  for all $n\in \N$. As $\ver_{n, 0,\cdot} \equiv 0$ the functions $u_n(0,\cdot)$ (in the sense of traces) are constant on $(0,1)$. Hence, we may assume that $u_n(0,\cdot) \equiv 0$.\\

\quad \emph{\textbf{Step 4. Boundedness of $\bm{u_n}$ in $\bm{\mathcal{W}(\Omega)}$.}}  We divide the proof of the boundedness into two separate steps. We show that there is $C >0$ such that it holds for all $n\in \N$ and $i=1,2$
\begin{equation}
    \label{gamma:boundedness_partial}
    \Vert \partial_{ii} u_n\Vert_{L^2(\Omega)} \leq  C <\infty \quad \textup{ and } \quad \Vert \partial_i u_n \Vert_{L^2(\Omega)} \leq C <\infty.
\end{equation}
 We only prove the above estimates \eqref{gamma:boundedness_partial} for $i = 1$. The estimates for $i=2$ follow similarly.  \\

\quad \emph{\textbf{Step 4.1: Boundedness of $\bm{(\partial_{11}u_n)_n}$ in $\bm{L^2(\Omega)}$.}}
We observe that for all $\theta_1,\theta_2 \in [-\pi/3,\pi/3]$ there holds
\[ 
\vert \theta_2 -\theta_1 \vert = 2 \left\vert \int_{\sin(\theta_1/2)}^{\sin(\theta_2/2)}\frac{1}{\sqrt{1-z^2}}\textup{ d}z \right\vert \leq 4\left\vert \sin\left( \frac{\theta_2}{2}\right) - \sin \left( \frac{\theta_1}{2}\right)\right\vert.  
\] 
By Step 2 we have $\mathcal{H}^0(\mathcal{A}_n) = 0$ for all $n\in \N$, where the sets $\mathcal{A}_n$ are defined in \eqref{est:vortices}. Applying the reformulation of the energy \eqref{gamma:spins:energie_one} and Lemma \ref{lem:lower_bound_q} then yields  
\begin{equation}
\begin{split}
& \Vert \partial_{11} u_n\Vert_{L^2(\Omega)}^2 \leq (2\delta_n^{ver})^{-1} \eps^2_n \sum_{(i,j)\in \INNhor} \left\vert \partial_{1}^{d,\eps} \hor_{n,i,j}\right\vert^2  = \frac{(2\delta_n^{ver})^{-1}}{\eps_n^2} \eps_n^2 \sum_{(i,j)\in \INNhor} \left\vert \hor_{n,i+1, j} -\hor_{n, i, j}\right\vert^2 \\
& \leq 16\frac{(2\delta_n^{ver})^{-1}}{\eps_n^2} \eps_n^2 \sum_{(i,j)\in \INNhor} \left\vert \sin\left(\frac{\hor_{n,i+1, j}}{2}\right) -\sin\left(\frac{\hor_{n, i, j}}{2}\right)\right\vert^2 = 4 \frac{\delta_n^{hor}}{\delta_n^{ver}}\eps_n^2 \sum_{(i,j)\in \INNhor} \left\vert \partial_1^{d,\eps} \gamma_n^{-1}w_{n, i, j}\right\vert^2  \\
&\leq 4 \eps_n^2 \sum_{(i,j)\in \INNhor} \left\vert \partial_1^{d,\eps} w_{n, i, j}\right\vert^2 
\leq\frac{8(2\delta_n^{ver})^{1/2}}{\eps_n}\Hn(w_n, z_n)\leq \frac{8(2\delta_n^{ver})^{1/2}}{\eps_n} K.\end{split}\label{est:secon_derivatives-liminf} 
\end{equation}
 The right hand side of \eqref{est:secon_derivatives-liminf} is bounded since $\eps_n/(2\delta_n^{ver})^{1/2} \rightarrow \sigma \in (0,\infty)$ as $n\rightarrow \infty$. \\

\quad \emph{\textbf{Step 4.2: Boundedness of $\bm{(\partial_1 u_n)_n}$ in $\bm{L^2(\Omega)}$.}} First, note that it holds $\vert \theta \vert \leq 2\vert \sin(\theta)\vert$ for $\theta \in [-\pi/3, \pi/3]$. Using that $\mathcal{H}^0(\mathcal{A}_n) = 0$, the reformulation of the energy \eqref{gamma:spins:energie_one} and H\"older's inequality, we estimate 
\begin{equation}\begin{split}
& \Vert \partial_1 u_n\Vert_{L^2(\Omega)}^2 \leq  \frac{4\eps_n^2}{2\delta_n^{ver}} \sum_{(i,j)\in \INhor}\vert \hor_{n, i, j}\vert^2\leq  \frac{4\cdot16\eps_n^2}{2\delta_n^{ver}} \sum_{(i,j)\in \INhor}\left\vert \sin\left( \frac{\hor_{n, i, j}}{2} \right) \right\vert^2 = 16 \eps_n^2 \sum_{(i,j)\in \INhor}\left\vert  w_{n,i,j} \right\vert^2  \\
& \leq 16 \left(\eps_n^2 \sum_{(i,j)\in \INhor}\left(\left\vert w_{n,i,j} \right\vert^2-\gamma_n^2\right)\right) + 16 \gamma_n^2 \leq 16 \left(\eps_n^2 \sum_{(i,j)\in \INhor} \left(\left\vert  w_{n,i,j} \right\vert^2-\gamma^2_n\right)^2\right)^{1/2} + 16 \gamma_n^2   \\
& \leq 16\left( \frac{2\eps_n}{(2\delta_n^{ver})^{1/2}}\Hn(w_n, z_n) \right)^{1/2} + 16 \gamma_n^2 \leq 23\left( \frac{\eps_n}{(2\delta_n^{hor})^{1/2}} K \right)^{1/2} + 16 \gamma_n^2,\end{split}
\end{equation}
which is uniformly bounded in $n\in \N$ since $\eps_n/(2\delta_n^{ver})^{1/2} \rightarrow \sigma\in (0,\infty)$ and $\gamma_n^2 \to \gamma^2 \in [0,\infty)$ as $n\rightarrow \infty$. \\

\quad \emph{\textbf{Step 5: Conclusion.} }
By Step 4.1 and Step 4.2 the sequence $(u_n)_n$ is bounded in $\mathcal{W}(\Omega)$. By Theorem \ref{th:embedding} this implies the boundedness in $W^{2,2}(\Omega)$. As $\Omega$ is an extension domain for $W^{2,2}$ (see \cite[Remark 2.5.2]{Ziemer}) we obtain by the Rellich-Kondrachov theorem the existence of $u \in W^{2,2}(\Omega)$ and a (not relabeled) subsequence such that $u_n \rightharpoonup u$ in $W^{2,2}(\Omega)$ and $u_n \to u$ in $W^{1,2}(\Omega)$. 
\end{proof}
We will now prove the assertions (i) and (ii) from Theorem \ref{thm:gamma}.

\begin{proposition}[Local compactness and liminf inequality]\label{prop:lower_gamma_per} Let the assumptions of Theorem \ref{thm:gamma} hold. Then, the statements (i) and (ii) are true for $\Hn\colon L^2(\Omega;\R^2) \rightarrow [0,+\infty]$ and the limit functional $\H\colon L^2(\Omega;\R^2) \rightarrow [0,+\infty]$.
\end{proposition}
\begin{proof}
We divide the proof into two parts. First, we prove the local compactness. Subsequently, we prove the desired liminf inequality for the periodic problem.\\
 
\quad \textbf{Local compactness.} Let $(w_n, z_n)\in L^2(\Omega;\R^2)$ be uniformly bounded in energy, i.e., there is $K> 0$ such that 
\begin{equation}
    \label{gamma:periodic:uniform_lower}
    \Hn(w_n, z_n) \leq K \quad \textup{ for all } n\in \N.
\end{equation}
As in the proof of Lemma \ref{lem:existenc_potentials} we denote by $v_n \in \Sbound$ the corresponding spin fields and by $(\hor_n,\ver_n) \colon \Omega \cap\eps_n \Z^2 \rightarrow [-\pi, \pi)^2$ the associated angular velocity fields. Again, we identify the functions $w_n, z_n: \Omega \to \R$ with the discrete functions $\gamma_n w^{\Delta_n}(v_n)$ and $z^{\Delta_n}(v_n)$ as defined in \eqref{def:chirality_order_parameters}. The strategy will then be to use the potentials $u_n \in \Sbound$ satisfying the assertions of Lemma \ref{lem:existenc_potentials} to conclude the proof of the compactness by showing 
\begin{equation}\left\Vert  \nabla u_n - \left( \begin{array}{c}
   w_n \\
     z_n
\end{array}\right)\right\Vert_{L^2(\Omega;\R^2)}  \longrightarrow 0 \quad \textup{ as } n \rightarrow \infty. \label{convergence_difference_L2}\end{equation}
 
\quad\emph{\textbf{Step 1. Definition of a suitable candidate $\bm{(w,z) \in \dom(\H).}$}} Let $u_n\in W^{2,2}(\Omega)$ be the sequence of associated potentials from Lemma \ref{lem:existenc_potentials} that satisfies the identities
\begin{equation}
    \label{local_comp:id_potential}
    \nabla u_n =  ((2\delta_n^{ver})^{-1/2}\pa(\hor_n) ,(2\delta_n^{ver})^{-1/2} \pa(\ver_n))
\end{equation}
and \[  \partial_{11} u_n = (2\delta_n^{ver})^{-1/2} \partial_1^{d,\eps} \hor_n, \ \partial_{22} u_n = (2\delta_n^{ver})^{-1/2}\partial_2^{d,\eps} \ver_n\quad \textup{ on }\Omega \cap\eps\Z^2. \]
In particular, as in the proof of Lemma \ref{lem:existenc_potentials}, we obtain for
\[
\mathcal{A}_n := \left\{ (i,j) \in \INNhor\colon \max\{\vert \hor_{i,j}\vert, \vert \hor_{i,j+1}\vert\} \geq\frac\pi8\right\} \bigcup \left\{ (i,j)\in \INNver\colon \max\{ \vert \ver_{i,j} \vert , \vert \ver_{i+1,j}\vert \}\geq \frac\pi8\right\}
\]
that
\[
\mathcal{H}^0(\mathcal{A}_n) = 0.
\]
Furthermore, there is $u\in W^{2,2}(\Omega)$ and a subsequence (not relabeled) such that $u_n \rightarrow u$ strongly in $W^{1,2}(\Omega)$ and $u_n\rightharpoonup u$ weakly in $W^{2,2}(\Omega)$ as $n\rightarrow \infty.$ Then define $(w,z)^T \coloneqq \nabla u\in W^{1,2}(\Omega;\R^2).$ It follows $\textup{curl}((w,z)) = 0$ in $\Omega$. Next, we will argue that it holds
 \begin{equation}
     \label{gamma:two:periodic:boundary}
     z(0, \cdot) = 0, \ \vert z(\cdot, 0)\vert = \vert z(\cdot, 1)\vert, \quad \textup{ and }\quad \vert w(0, \cdot) \vert = \vert w(1, \cdot)\vert .
 \end{equation}
 Let $\varphi\in C^1(\Omega)\cap C(\overline{\Omega})$ with $\varphi(\cdot, x) \equiv 0$ for $x\in \{0,1\} $ and $\varphi(1, \cdot) \equiv 0.$ Using integration by parts, the weak convergence $ \nabla u_n \rightharpoonup (w,z)^T$ in $W^{1,2}(\Omega;\R^2)$ and $\partial_2 u_n(0, \cdot) = 0$, we conclude 
 \begin{equation}
\int_0^1 z(0, y) \varphi(0,y)\textup{d}y = \int_\Omega  \varphi \partial_1 z + z\partial_1 \varphi \dx\dy = \lim_{n\rightarrow \infty} \int_\Omega  \varphi \partial_1 \partial_2 u_n + \partial_2 u_n \partial_1 \varphi \dx\dy = 0.
 \end{equation}
 Thus, $z(0, \cdot) = 0.$ Next, we verify the periodicity condition from \eqref{gamma:two:periodic:boundary}. The uniform boundedness of $\Hn (w_n,z_n)$ implies that the associated spin fields $v_n \colon \Omega \cap \eps_n\Z^2\rightarrow \mathbb{S}^1$ satisfy for all  $k = 0, \dots, \left\lfloor \frac1\eps_n\right\rfloor$
 \[    v_{n,0,k} \cdot v_{n,1, k} = v_{n,\lfloor \frac1\eps_n \rfloor - 1, k} \cdot v_{n,\lfloor \frac1\eps_n \rfloor, k} \quad  \textup{ and } \quad v_{n,k, 0}\cdot v_{n,k, 1}= v_{n,k, \lfloor \frac1\eps_n\rfloor -1} \cdot v_{n,k, \lfloor \frac1\eps_n\rfloor}.  \] Hence, 
 \[ \cos(\hor_{n,0,k})  = \cos\left(
 \hor_{n, \lfloor \frac1\eps_n \rfloor-1, k}\right) \quad  \textup{ and } \quad \cos(\ver_{n, k, 0})= \cos\left(\ver_{n,k, \lfloor \frac1\eps_n\rfloor -1}\right)  \quad \textup{ for all } k = 0, \dots, \left\lfloor \frac1\eps_n\right\rfloor,  \]
 which yields $\vert \partial_1 u_n(0, \cdot) \vert = \vert \partial_1 u_n(1, \cdot)\vert$ and $\vert \partial_2 u_n(\cdot, 0)\vert = \vert \partial_2 u_n(\cdot, 1)\vert$ in the sense of traces.
By the weak convergence of $\nabla u_n$ to $(w,z)$ in $W^{1,2}(\Omega; \R^2)$, the previous observations imply $(w,z) \in \dom(\H).$\\
 
 \quad \emph{\textbf{Step 2. Convergence of $\bm{((w_n, z_n))_n}$ towards $\bm{\nabla u}$ in $\bm{L^2(\Omega;\R^2)}$.}} We only prove the strong convergence $w_n -\partial_1 u_n\rightarrow 0$ in $L^2(\Omega)$ as $n\rightarrow \infty$. The convergence $z_n -\partial_2 u_n\rightarrow 0$ in $L^2(\Omega)$ as $n\to \infty$ follows similarly. We observe
 \begin{equation}\begin{split}
 & \Vert w_n - \partial_1 u_n \Vert_{L^2(\Omega)} \leq \Vert w_n -(2\delta_n^{ver})^{-1/2} \pa({\theta}_n^{hor})\Vert_{L^2(\Omega)} + \Vert (2\delta_n^{ver})^{-1/2} \pa(\theta_n^{hor}) -\partial_1 u_n\Vert_{L^2(\Omega)} \\
 & =: A_{n} + B_{n}.\end{split}
 \end{equation}
 We estimate $A_n$ and $B_{n}$ for $n\in \N$ separately. In order to estimate $B_n$ fix $(i,j) \in \INhor$ and a corresponding square $Q_{\eps_n}(i,j)$. 
 By the definition of $\pa(\theta_n^{hor})$, \eqref{def:gamma:triangulation}, and the identity \eqref{local_comp:id_potential}, we find for $(x,y) \in T_{\eps_n}^-(i,j)$
 \begin{equation}
 \begin{split}
& ( 2\delta_n^{ver})^{-1/2} \left\vert \hor_{n,i, j} - \pa(\theta_{n}^{hor}) (x,y) \right\vert \leq ( 2\delta_n^{ver})^{-1/2} \left\vert \partial_1^{d,\eps_n} \hor_{n, i,j} (x-i\eps_n) + \partial_2^{d,\eps_n}\hor_{n, i, j}(y-j\eps_n) \right\vert \\
& \quad \leq \left\vert \partial_{11}u_n (i\eps_n, j\eps_n) \right\vert (x-i\eps_n) + \left\vert \partial_{21} u_n(i\eps_n,j\eps_n) \right\vert (y-j\eps_n) \end{split}
 \end{equation}
 and for $(x,y) \in T_{\eps_n}^+(i,j)$
  \begin{equation}
  \begin{split}
& (2\delta_n^{ver})^{-1/2} \left\vert \hor_{n,i, j} - \pa(\theta_{n}^{ hor}) (x,y) \right\vert\nonumber    \\
& \leq (2\delta_n^{ver})^{-1/2} \left\vert \partial_1^{d,\eps_n} \hor_{n, i,j+1} (x-(i+1)\eps_n) + \partial_2^{d,\eps_n}\hor_{n, i+1, j}(y-(j+1)\eps_n) + \hor_{n, i+1, j+1} - \hor_{n, i, j}\right\vert   \nonumber \nonumber \\
& \leq \left\vert \partial_{11}^{d,\eps_n}u_n (i\eps_n, (j+1)\eps_n) \right\vert (x-(i+1)\eps_n) + \left\vert \partial_{21}^{d,\eps_n} u_n((i+1)\eps_n,j\eps_n) \right\vert (y-(j+1)\eps_n)\nonumber \\
& \quad +\eps_n \left\vert \partial_{11}^{d,\eps_n} u_n(i\eps_n,j\eps_n)\right\vert +\eps_n \left\vert \partial_{21}^{d,\eps_n}u_n((i+1)\eps_n, j\eps_n)\right\vert. \end{split}
 \end{equation}
Consequently, 
 \begin{align}
 B_{n} &\leq \left( \sum_{(i,j)\in\INhor}  \Vert (2\delta_n^{ver})^{-1/2} \pa(\theta_n^{hor}) - \partial_1 u_n\Vert_{L^2(Q_{\eps_n}(i,j))}^2 \right)^{1/2} \\ & \leq C\eps_n\left( \Vert \partial_{11} u_n \Vert_{L^2(\Omega)} + \Vert \partial_{21} u_n \Vert_{L^2(\Omega)} \right).
\end{align}
As the sequence $(u_n)_n$ is uniformly bounded in $W^{2,2}(\Omega)$, see Lemma \ref{lem:existenc_potentials}, we obtain $B_n \rightarrow 0$ as $n\rightarrow \infty.$\\

It remains to show $A_n \rightarrow 0$ as $n\rightarrow \infty$. First, observe that for all $\theta \in [-\pi/3, \pi/3]$ and $\arccos(1-\delta) < \pi/3$ it holds
\begin{equation}
    \label{gamma:two:arccos_inequality}
    \vert \theta^2 -(\arccos(1-\delta))^2\vert\leq 4 \left\vert \int_{\vert \arccos(1-\delta)\vert }^{\vert \theta\vert} \frac12  s\textup{ d}s\right\vert \leq 4 \left\vert \int_{\vert \arccos(1-\delta)\vert}^{\vert \theta\vert} \sin(s) \textup{ d}s\right\vert \leq 4\vert 1-\delta -\cos(\theta) \vert.
\end{equation}
Note that in Step 1 we argued that $\mathcal{H}^0(\mathcal{A}_n) = 0$ for all $n\in \N$. Therefore, there are no indices $(i,j)\in \INhor$ such that $\vert \hor_{n,i,j}\vert \geq \pi/3 $. Moreover, there is a further subsequence that satisfies the condition $\arccos(1-\delta_n^{hor}) < \pi/3$ for all $n\in \N$. By inequality \eqref{gamma:two:arccos_inequality} and $1-\cos(\theta) = 2\sin^2(\theta/2)$ for $\theta \in \R$, we conclude
\begin{equation}
\begin{split}
& \sup\left\{ \left\vert (\hor_{n,i,j})^2 - (\arccos(1-\delta_n^{hor}))^2 \right\vert^2  \colon (i,j) \in \INhor\right\} \\ \leq &4\sup\left\{ \left\vert 1-\delta_n^{hor} -\cos(\hor_{n, i, j}) \right\vert^2 \colon (i,j) \in \INhor\right\} \\
 \leq &\frac{4}{\eps_n^2} \eps_n^2 \sum_{(i,j) \in \INhor}\left\vert 1-\delta_n^{hor} -\cos(\hor_{n, i, j}) \right\vert^2 \\
 = & \frac{4}{\eps_n^2} \eps_n^2 \sum_{(i,j) \in \INhor}\left\vert 2\sin^2\left(\frac{\hor_{n, i, j}}{2}\right)-\delta_n^{hor} \right\vert^2 \\
 = & \frac{4}{\eps_n^2} (\delta_n^{ver})^2\eps_n^2 \sum_{(i,j) \in \INhor}\vert \gamma_n^2 - (w_{n,i,j})^2 \vert^2\nonumber \\
 = & \frac{4}{\eps_n^2} (\delta_n^{hor})^2\eps_n^2 \sum_{(i,j) \in \INhor}\vert 1 - (\gamma_n^{-1}w_{n,i,j})^2 \vert^2\nonumber\\
\leq &\frac{4\sqrt{2}(\delta_n^{ver})^{3/2}}{\eps_n} \Hn(w_n, z_n)\leq \frac{4\sqrt{2}(\delta_n^{ver})^{3/2}}{\eps_n} K \eqqcolon \beta_n\delta^{ver}_n. \end{split}
\end{equation}
Since $\beta_n \rightarrow 4/\sigma$ as $n\rightarrow \infty$, there is $C> 0$ such that $\vert \beta_n\vert \leq C <\infty$ for all $n\in \N$. By the triangle inequality, we find
\begin{equation}\label{proof:lower:uniform_bound}
  s_n^{hor} \coloneqq  \sup_{n\in \N} \vert \hor_{n} \vert^2 \leq \beta_n^{1/2} (\delta_n^{ver})^{1/2} +4\delta_n^{hor} \rightarrow 0\quad \textup{ as } n\rightarrow \infty,
\end{equation}
where we used that $\arccos(1-\delta) \leq 2\delta^{1/2}$ for $\delta\in (0,1)$. Using the estimate $\vert \sin(\theta) -\theta\vert \leq \vert\theta\vert^3 $ for $\theta\in \R$, we conclude 
\begin{equation}\begin{split}
 A_n^2 &\leq \eps_n^2 \sum_{(i,j)\in \INhor} \vert w_{n,i, j}- (2\delta_n^{ver})^{-1/2} \hor_{n, i, j}\vert^2 = \frac{2\eps_n^2}{\delta_n^{ver}} \sum_{(i,j)\in \INhor} \left\vert \sin\left(\frac{\hor_{n,i,j}}{2}\right)- \frac{\hor_{n, i, j}}{2}\right\vert^2  \\
&\leq  \frac{2\eps_n^2}{\delta_n^{ver}} \sum_{(i,j)\in \INhor} \left\vert \hor_{n, i, j}\right\vert^6\leq C \left( \beta_n^{3/2} (\delta_n^{ver})^{1/2} + \gamma_n^2 (\delta_n^{hor})^2 \right) \rightarrow 0 \quad \textup{ as } n\rightarrow \infty.\end{split}
\end{equation}
This concludes the proof of \eqref{convergence_difference_L2}. \\

\quad \textbf{Liminf inequality.} We show that for $(w,z) \in L^2(\Omega;\R^2)$ and any sequence $(w_n, z_n)\in L^2(\Omega;\R^2)$ for $n\in \N$ converging strongly to $(w,z)$ in $L^2(\Omega;\R^2)$ as $n\rightarrow \infty$, we have 
\[
\H(w,z) \leq \liminf\limits_{n\rightarrow \infty} \Hn(w_n, z_n). 
\] 
The proof of the liminf-inequality can be done separately for the horizontal and vertical contributions of the energy.
We will only present the proof for the horizontal terms. 
Precisely, we show
\begin{equation}
  \frac1\sigma  \int_\Omega (\gamma^2-w^2)^2 \textup{ d}x\textup{d}y +\sigma \int_\Omega (\partial_2w)^2 \textup{ d}x\textup{d}y \leq \liminf\limits_{n\rightarrow \infty}\Hnhor(w_n).\label{liminf: horizontal}
\end{equation}
We may assume without loss of generality that $\liminf_{n \rightarrow \infty} \Hn(w_n, z_n) < \infty$, $\liminf_{n \rightarrow \infty} \Hn(w_n, z_n) = \lim_{n \rightarrow \infty} \Hn(w_n, z_n)$ and $\Hn(w_n,z_n) \leq C$ for all $n\in \N$.
By the local compactness, there exists a (not relabeled) subsequence of $(w_{n}, z_{n})$ that converges to some $(w,z)^T \in \dom(\H)$ strongly in $L^2(\Omega;\R^2)$. The field $(w, z)$ satisfies the properties $\textup{curl}((w, z)) \equiv 0$, 
\begin{equation}
   z(0, \cdot) = 0, \ \vert z(\cdot, 0) \vert = \vert z(\cdot, 1) \vert, \quad \textup{ and }\quad \vert w(0, \cdot) \vert = \vert w(1, \cdot)\vert. \label{gamma:two:per:liminf_boundary}
\end{equation} 
Moreover, let $u_n\in W^{2,2}(\Omega)$ be the associated sequence of potentials from Step 1 of the proof of the local compactness.
Recall, that there it holds $(\nabla u_n)_n$ converges to $(w,z)^T$ strongly in $L^2(\Omega;\R^2)$ and weakly in $W^{1,2}(\Omega;\R^2)$. 
First, we observe that it follows by the strong convergence $w_n \to w$ in $L^2$ and $\sigma_n:= \eps_n/\sqrt{2 \delta_n^{ver}} \to \sigma$ that
\begin{equation}
\begin{split}
 \frac1\sigma \int_\Omega (\gamma^2 -w^2)^2\textup{ d}x\textup{d}y = &\lim \limits_{n\rightarrow \infty} \frac{1}{\sigma_n}\int_\Omega (\gamma_n^2 -w_n^2)^2\textup{ d}x\textup{ d}y \\
 =&  \lim\limits_{n\rightarrow \infty} \frac{\sqrt{2}}{\eps_n } (\delta_n^{ver})^{1/2} \eps_n^2 \sum_{(i,j)\in \INhor} ( \gamma_n^2 - w_{n,i,j}^2)^2 \\
= &\liminf\limits_{k\rightarrow \infty}\frac{2}{\sqrt{2}\eps_n(\delta_n^{ver})^{3/2}} (\delta_n^{ver})^2 \eps_n^2 \sum_{(i,j)\in \INhor} ( \gamma_n^2 - w_{n,i,j}^2)^2\\
= &\liminf\limits_{k\rightarrow \infty}\frac{1}{\sqrt{2}\eps_n(\delta_n^{ver})^{3/2}} 2(\delta_n^{hor})^2 \eps_n^2 \sum_{(i,j)\in \INhor} ( 1 - (\gamma_n^{-1}w_{n,i,j})^2)^2.
\end{split}\label{proof:lower:ver:elastic:liminf}
\end{equation}
Next, the weak lower semi-continuity of the $L^2$-norm, the weak convergence $\partial_{1} u_n\rightharpoonup w$ in $W^{1,2}(\Omega)$ as $n\rightarrow \infty$ and the definition of $u_n$ (recall Figure \ref{fig:construction_triangular}) yield 
\begin{equation}
\begin{split}
& \sigma \int_\Omega \vert \partial_1 w\vert^2 \textup{ d}x\textup{d}y\leq \liminf\limits_{n\rightarrow\infty} \sigma_n \int_\Omega \vert \partial_{11} u_n\vert^2 \textup{ d}x\textup{d}y = \liminf\limits_{n\rightarrow \infty} \frac{\eps_n^3}{(2\delta_n^{ver})^{1/2}} \sum\limits_{(i,j)\in \INNhor} \left\vert (2\delta_n^{ver})^{-1/2}\partial_{1}^{d,\eps_n} \hor_{n,i,j}\right\vert^2 \\
& = \liminf\limits_{n\rightarrow \infty} \frac{4\eps_n}{(2\delta_n^{ver})^{3/2}}\sum\limits_{(i,j)\in \INNhor} \left\vert \frac{\hor_{n,i+1,j}}{2}-\frac{\hor_{n,i,j}}{2} \right\vert^2.
\end{split}\end{equation}
By Young's inequality and \eqref{proof:lower:uniform_bound} 
\begin{align}
        &\frac{4\eps_n}{(2\delta_n^{ver})^{3/2}}\sum\limits_{(i,j)\in \INNhor} \left\vert \frac{\hor_{n,i+1,j}}{2}-\frac{\hor_{n,i,j}}{2} \right\vert^2  \\ 
        & \quad = \frac{4\eps_n}{(2\delta_n^{ver})^{3/2}}\sum\limits_{(i,j)\in \INNhor} \left\vert \int^{\frac{\hor_{n,i+1,j}}{2}}_{\frac{\hor_{n,i,j}}{2}} 1\textup{ d}s \right\vert^2 \\
        & \quad \leq  \frac{4\eps_n}{(\cos((s_n^{hor})^{1/2}))^2(2\delta_n^{ver})^{3/2}}\sum\limits_{(i,j)\in \INNhor} \left\vert \sin\left(\frac{\hor_{n,i+1,j}}{2}\right)-\sin\left(\frac{\hor_{n,i,j}}{2}\right) \right\vert^2\\
        & \quad = \frac{2}{(\cos((s_n^{hor})^{1/2})^2\eps_n (2\delta_n^{ver})^{3/2}} \dhor_n \eps_n^2 \eps_n^2 \sum\limits_{(i,j)\in \INNhor} \left\vert \partial_1^{d, \eps}\gamma_n^{-1} w_{n, i, j}\right\vert^2.
 \label{gamma:two:rest}
\end{align}
Since $s_n^{hor} \rightarrow 0$ as $n\rightarrow \infty$, we have $1/(\cos((s_n^{hor})^{1/2}) \rightarrow 1$ as $n\rightarrow \infty$. Together with the uniform convergence $p(\theta_1, \theta_2)\rightarrow 1$ as $(\theta_1, \theta_2) \rightarrow (0,0)$, see Lemma \ref{lem:p_convergence}, we conclude 
\begin{equation}
    \begin{split}
     \sigma \int_\Omega \vert \partial_1 w\vert^2 \textup{ d}x\textup{d}y & \leq \liminf_{n\rightarrow \infty} \frac{2}{(\cos((s_n^{hor})^{1/2})^2\eps_n (2\delta_n^{ver})^{3/2}} \dver_n \eps_n^4 \sum\limits_{(i,j)\in \INNhor} \left\vert \partial_1^{d, \eps} w_{n, i, j}\right\vert^2\\
     & \leq \liminf\limits_{n\rightarrow \infty} \frac{1}{\sqrt{2}\eps_n(\delta_n^{ver})^{3/2}} \dver_n \eps_n^4  \sum\limits_{(i,j)\in \INNhor} p(\hor_{n,i,j}, \hor_{n, i, j+1})\left\vert \partial_{1}^{d,\eps_k}  w_{n,i,j}\right\vert^2. 
    \end{split}
\end{equation}
This finishes the proof of \eqref{liminf: horizontal}.
Arguing similarly for the vertical contribution proves the liminf inequality. 
\end{proof}

\subsection{Upper bound: Existence of a recovery sequence}\label{sec:upperbound}
In this section, we prove the upper bound stated in (iii) of Theorem \ref{thm:gamma}. We start with the construction of recovery sequences for a dense subset of $\dom(\H)$. Then a standard diagonal argument will yield the existence of a recovery sequence for all $(w,z) \in L^2(\Omega;\R^2)$. 
\begin{lemma}[Existence of a recovery sequence for a dense subset] \label{lem:recovery-sequence_smoot} Let $\eps_n\in (0,1/2), \Delta_n = (\delta_n^{hor}, \delta_n^{ver})$ satisfy 
\begin{equation}
\quad \eps_n \rightarrow 0, \quad  \Delta_n = (\dhor_n, \dver_n) \rightarrow (0,0), \quad \frac{\eps_n}{\sqrt{2\dver_n}} \rightarrow \sigma \in (0,\infty), \quad \text{ and } \quad \gamma_n^2= \frac{\dhor_n}{\dver_n} \rightarrow \gamma^2 \in [0,\infty).
\end{equation}
For every $(w,z) \in \dom(H)\cap C^\infty(\overline{\Omega})$ there is a sequence $(w_n, z_n) \in \dom(\Hn)$ for $n\in \N$ such that $(w_n, z_n)\rightarrow (w,z)$ in $L^2(\Omega;\R^2)$ as $n\rightarrow \infty$ and
\begin{equation}\label{limsup:spins:inequality}
    \limsup\limits_{n\rightarrow \infty} \Hn(w_n, z_n) \leq \H(w,z).
\end{equation} 
\end{lemma}
\begin{proof}
Let $(w,z) \in \dom(\H)\cap C^\infty(\overline{\Omega})$. 
By definition of $\dom(\H)$ the field $(w,z)$ satisfies
\[z(0, \cdot) = 0, \ \vert z (\cdot, 0)\vert = \vert z(\cdot, 1)\vert, \quad \textup{ and }\quad \vert w(0, \cdot)\vert = \vert w(1, \cdot)\vert.\] 
We start by introducing a discretization  of $(w,z)$ on $\Omega \cap \eps_n\Z^2$. 
Note that it is not clear that this can be interpreted as the angular velocity field of a spin field.
Hence, we will have to modify this discretization to obtain an angular velocity field $(\hor_n, \ver_n)\colon \Omega \cap\eps_n\Z^2 \rightarrow [-\pi, \pi)^2$ that induces a spin field $v_n \in \Sboundper(\Omega;\eps_n).$ 
The actual recovery sequence will then be given by the sequence of chiralities $(w_n, z_n)$ associated with this spin field $v_n$.  \\

\quad \emph{\textbf{Step 1. The discretization $\bm{(\tilde{w}_n, \tilde{z}_n) \colon \Omega \cap \eps_n \mathbb{Z}^2 \rightarrow {\rm I\!R}^2}$ of $\bm{(w, z)}$ on $\bm{\Omega \cap\eps_n\Z^2}$.}} Define for $(i,j) \in \Ind$
\begin{equation}
    \label{def:wnzn:limsup_spins}
    \tilde{w}_{n, i, j}\coloneqq \begin{cases}
        w(i\eps_n, j\eps_n) \quad & \textup{ if } i \leq \lfloor 1/\eps_n \rfloor -1, \\
        w(1, j\eps_n) &\textup{ if } i = \lfloor 1/\eps_n\rfloor 
    \end{cases} \quad \textup{ and } \quad \tilde{z}_{n, i, j} \coloneqq \begin{cases}
        z(i\eps, j\eps_n) \quad & \textup{ if } j \leq \lfloor 1/\eps_n\rfloor -1 ,\\
       z(i\eps_n, 1) \quad & \textup{ if } j = \lfloor 1/\eps_n\rfloor.
    \end{cases}
\end{equation}

\quad \emph{\textbf{Step 2. Lifting $\bm{(\tilde{w}_n, \tilde{z}_n)} $ to a spin field $\bm{u_n \colon \Omega \cap \eps_n \Z^2 \rightarrow \mathbb{S}^1}$}.} First, note that there exists $N \in \N$ such that for all $n\geq N$
\begin{equation}
    \label{limsup:boundedness}
    (\delta_n^{ver})^{1/2} \max\{ \Vert z\Vert_{L^\infty}, \Vert w\Vert_{L^{\infty}}\}\leq 1.
\end{equation} 
We set $\phi_n^{hor}: \Omega \cap \eps_n \Z^2 \to (-\pi/2,\pi/2)$ as 
\begin{equation}
    \label{def:preliminary_angle_velocity}
    \phi_n^{hor} \coloneqq \begin{cases}
        2 \arcsin\left( \sqrt{\frac{\delta_n^{ver}}{2}}\tilde{w}_n\right)\quad & \textup{ if } n\geq N,\\
        0 & \textup{ else}.
    \end{cases} 
    \end{equation}
We define the spin field $u_n \colon \Omega \cap \eps_n \Z^2 \rightarrow \mathbb{S}^1$ for $(i,j) \in \Ind$ as 
\begin{equation}
    \label{def:spin_field:limsup}
    u_{n, i,j} \coloneqq \begin{cases}(0,1)^T & \textup{ if } i= 0, \\
    \Rot\left( \sum_{k=0}^{i-1} \phi_{n, i, j}^{hor}\right) (0,1)^T \quad & \textup{ else}. \end{cases}
\end{equation}

\quad \emph{\textbf{Step 3. Determining the vertical angular velocity $\bm{\ver_n \colon \Omega \cap\eps_n \Z^2 \rightarrow [-\pi,\pi)}$ of $\bm{u_n.}$}} Let us now define
\begin{equation}
    \phi^{ver}_{n,i,j} \coloneqq \begin{cases} \sum_{k=0}^{i-1}(\phi^{hor}_{n, k, j+1}- \phi^{hor}_{n, k, j})\quad & \textup{ if } i\geq 1, \\
    0 & \textup{ else},
    \end{cases}\label{rec:def:vertical_vel}
\end{equation}
which satisfies \[u_{n, i, j+1} = \Rot( \phi^{ver}_{n, i, j}) u_{n,i,j}\quad \textup{ for all }(i,j) \in \INver.\]  Note that by definition, the angular vertical velocity fields $(\phi_n^{hor},\phi_n^{ver})$ is curl-free in a discrete sense. 
Moreover, the bound \eqref{limsup:boundedness} and $w \in C^{ \infty}(\overline{\Omega})$ imply 
\begin{equation}
    \begin{split}
        \vert  \phi^{ver}_{n,i,j} \vert & \leq  \sum_{k=0}^{i-1}\vert \phi^{hor}_{n, k, j+1}- \phi^{hor}_{n, k, j}\vert  = 2 \sum_{k=0}^{i-1}\left\vert \arcsin\left( \sqrt{\frac{\delta_n^{ver}}{2}}\tilde{w}_{n, k, j+1}\right)- \arcsin\left( \sqrt{\frac{\delta_n^{ver}}{2}}\tilde{w}_{n, k, j}\right)\right\vert \phantom{halloha}\\
        & \leq 2 (\delta_n^{ver})^{1/2} \sum_{k=0}^{i-1} \vert w(k\eps_n, (j+1)\eps_n) - w(k\eps_n, j\eps_n)\vert \leq 2 (\delta_n^{ver})^{1/2} \Vert \partial_2 w \Vert_{L^\infty},
    \end{split}\label{admissibility_vertical}
\end{equation}
where we used that $|\arcsin(s) - \arcsin(t)| \leq \sqrt{2} |t-s|$ for all $t,s \in (-2^{-1/2},2^{-1/2})$.
Since the right-hand side of \eqref{admissibility_vertical} vanishes as $n\rightarrow \infty$, there is $M\in \N$ such that $ \phi^{ver}_{n} \in (-\pi/2, \pi/2)$ for all $n \ge M.$ Then the field 
\[ (\hor_n, \ver_n) \coloneqq \begin{cases}
    (\phi_n^{hor}, \phi^{ver}_n) \quad & \textup{if } n\geq M,\\
    (0,0) & \textup{else}
\end{cases}\]
induces the spin field $u_n\colon \Omega \cap \eps_n\Z^2 \rightarrow \mathbb{S}^1 $ with $\curln((\hor_n, \ver_n)) = 0$ for $n\geq M.$ \\

\quad \emph{\textbf{Step 4. Definition of the recovery sequence $\bm{(w_n, z_n)}$.}} Define  $(\hat{w}_n, \hat{z}_n) \colon \Omega \cap\eps_n \Z^2\rightarrow \R^2$ by 
\begin{equation}
    \hat{w}_n \coloneqq \sqrt{\frac{2}{\delta_n^{ver}}} \sin\left(\frac{\hor_n}{2}\right) \quad \textup{ and } \quad \hat{z}_n \coloneqq \sqrt{\frac{2}{\delta_n^{ver}}} \sin\left(\frac{\ver_n}{2}\right),
\end{equation}
and set $(w_n, z_n) \coloneqq (\pc(\hat{w}_n), \pc(\hat{z}_n))\colon \Omega \rightarrow \R^2$  for $n\in \N$. Note, that $(w_n, z_n) \in \dom(\Hn)$ for all $n\in \N$. In the following, we shall identify $(w_n, z_n)$ with its discrete version $(\hat{w}_n, \hat{z}_n).$\\

\quad \emph{\textbf{Step 5. Strong convergence $\bm{(w_n, z_n) \rightarrow (w, z)}$ in $\bm{L^2(\Omega, {\rm I\!R}^2)}$ as $\bm{n\rightarrow \infty}$.}} We only prove the convergence $z_n \rightarrow z$ in $L^2(\Omega)$ as $n\rightarrow \infty$ since the convergence $w_n \to w$ in $L^2(\Omega)$ is actually easier to show due to the simpler form of $\phi_n^{hor}$. 
Let $(x,y) \in Q_{\eps_n}(i,j)$. We estimate 
\begin{equation}
    \begin{split}
        &\vert z_n(x,y) - z(x,y) \vert \leq \vert z_{n, i, j}- z(i\eps_n,j\eps_n) \vert + \vert z(i\eps_n, j\eps_n) - z(x,y )\vert.
    \end{split}\label{stron_conv_rec_first}
\end{equation}
The second term on the right-hand side of \eqref{stron_conv_rec_first} is bounded by $\sqrt{2}\eps_n \Vert \nabla z \Vert_{L^\infty}$. 
For the first term of \eqref{stron_conv_rec_first}, using Lipschitz continuity of the $\sin$-function, we obtain
\begin{equation}
    \begin{split}
       &\left\vert  z_{n, i, j} - z(i\eps_n, j\eps_n) \right\vert \leq  \sqrt{\frac{2}{\delta_n^{ver}}} \left\vert \frac12 \left(\sum_{k=0}^{i-1}\left( \phi^{hor}_{n, k, j+1}-\phi^{hor}_{n, k, j}\right)\right) - \arcsin\left(\sqrt{\frac{\delta_{n}^{ver}}{2}} z(i\eps_n, j\eps_n)\right)\right\vert \\
       & \leq   \sqrt{\frac{2}{\delta_n^{ver}}} \sum_{k=0}^{i-1} \left\vert  \arcsin\left( \Psi^w_{n, k, j+1}\right)- \arcsin\left( \Psi^w_{n, k, j}\right)- \left(\arcsin\left(\Psi^z_{n, k+1, j}\right) - \arcsin\left(\Psi^z_{n, k, j}\right)\right) \right\vert,
    \end{split}
\end{equation}
where \[\Psi^w_{n, i, j} \coloneqq \sqrt{\frac{\delta_n^{ver}}{2}}w(i\eps_n, j\eps_n) \quad \textup{ and } \quad \Psi^z_{n, i, j} \coloneqq \sqrt{\frac{\delta_{n}^{ver}}{2}} z(i\eps_n, j\eps_n) \quad \textup{ on } \Omega \cap \eps_n \Z^2. \]
For a function $g\in C^{1}([l\eps_n,(l+1)\eps_n];(-1,1))$ and $l\in \Z$, it holds
\begin{equation}
    \begin{split}
        &\arcsin(g((l+1)\eps_n)) - \arcsin(g(l\eps_n)) =\int_{l\eps_n}^{(l+1)\eps_n}\frac{1}{\sqrt{1-g^2(s)}} g^\prime(s) \textup{ d}s.
    \end{split}
\end{equation}
Hence, 
\begin{equation}
    \begin{split}
       &  \sqrt{\frac{2}{\delta_n^{ver}}} \left(\arcsin\left( \Psi^w_{n, k, j+1}\right)- \arcsin\left( \Psi^w_{n, k, j}\right)- \left(\arcsin\left(\Psi^z_{n, k+1, j}\right) - \arcsin\left(\Psi^z_{n, k, j}\right)\right)\right)\\
       & = \int_{j\eps_n}^{(j+1)\eps_n} \frac{1}{\sqrt{1- \delta_n^{ver}w^2(k\eps_n, y)/2}} \partial_2 w(k\eps_n, y) \textup{ d}y - \int_{k\eps_n}^{(k+1)\eps_n} \frac{1}{\sqrt{1- \delta_n^{ver}z^2(x, j\eps_n)/2}} \partial_1 z(x, j\eps_n) \textup{ d}x\\
       & = \int_{j\eps_n}^{(j+1)\eps_n} \left(\frac{1}{\sqrt{1- \delta_n^{ver}w^2(k\eps_n, y)/2}} -1\right) \partial_2 w(k\eps_n, y)\textup{ d}y+ \int_{j\eps_n}^{(j+1)\eps_n} \partial_2 w(k\eps_n, y) \textup{ d}y\\
       & \phantom{ = } - \int_{k\eps_n}^{(k+1)\eps_n}\left( \frac{1}{\sqrt{1- \delta_n^{ver}z^2(x, j\eps_n)/2}} -1\right)\partial_1 z(x, j\eps_n) \textup{ d}x  - \int_{k\eps_n}^{(k+1)\eps_n}  \partial_1 z(x, j\eps_n) \textup{ d}x.
    \end{split}\label{limsup_equality_arcsin}
\end{equation}
The first and third term of the right hand-side can each be bounded by \[\eps_n \left\vert \frac{1}{\sqrt{1- \delta_n^{ver}\max\{\Vert w \Vert_{L^\infty}, \Vert z\Vert_{L^\infty}\}^2/2}} -1\right\vert \max\{ \Vert \partial_2 w\Vert_{L^\infty}, \Vert \partial_1 z \Vert_{L^\infty}\}:= \eps_n h_n.\]
From the boundedness of $w$, $z$, $\partial_2 w$ and $\partial_1 z$ it follows that $h_n \rightarrow 0$ as $n\rightarrow \infty.$ Moreover, $\textup{curl}((w, z)) = 0$ implies $\partial_2 w= \partial_1 z$ and thus 
\begin{equation}\label{limsup:second_termin}
    \begin{split}
      &\left\vert   \int_{j\eps_n}^{(j+1)\eps_n} \partial_2 w(k\eps_n, y) \textup{ d}y - \int_{k\eps_n}^{(k+1)\eps_n}  \partial_1 z(x, j\eps_n) \textup{ d}x\right\vert\\
     \leq & \frac{1}{\eps_n} \int_{k\eps_n}^{(k+1)\eps_n}\int_{j\eps_n}^{(j+1)\eps_n} \left \vert \partial_1 z (k\eps_n, y) - \partial_1 z(x, j\eps_n) \right\vert \textup{ d}y\textup{ d}x \\ \leq &\eps_n^2 (\Vert \partial_{12}z \Vert_{L^\infty} + \Vert \partial_{11} z\Vert_{L^\infty}).
    \end{split}
\end{equation}
Combining the estimates above we find that 
\begin{equation}    \begin{split}
     \vert z_{n,i,j} - z(i\varepsilon,j \varepsilon) \vert  \leq 2 h_n + 2 \eps_n \left( \Vert \partial_{12}z \Vert_{L^\infty} + \Vert \partial_{11} z\Vert_{L^\infty} \right) \to 0.
    \end{split}
\end{equation}
Thus, by \eqref{stron_conv_rec_first} we have $(w_n, z_n) \rightarrow (w, z)$ in $L^\infty(\Omega)$ as $n\rightarrow\infty$, which implies strong convergence in $L^2(\Omega;\R^2)$. 
Similarly, one can show that
\begin{equation}
    (\partial_1^{d, \eps_n} w_n, \partial_2^{d, \eps_n} z_n) \rightarrow (\partial_1 w, \partial_2 z) \quad \textup{ in } L^2(\Omega; \R^2) \quad \textup{ as } n\rightarrow \infty.
\end{equation}

\quad \emph{\textbf{Step 6. Conclusion.}} Using the strong convergences
\[ (w_n, z_n) \rightarrow (w, z) \textup{ in } L^\infty(\Omega; \R^2)\quad  \textup{ and }\quad  (\partial_1^{d, \eps_n} w_n, \partial_2^{d, \eps_n} z_n) \rightarrow (\partial_1 w, \partial_2 z) \textup{ in } L^2(\Omega; \R^2) \text{\ as $n\rightarrow \infty$,}\]  the boundedness of $ ((w_n, z_n))_n$ in $L^\infty(\Omega)$ and the boundedness of $ ((\partial_1^{d, \eps_n}w_n, \partial_2^{d, \eps_n}w_n))_n$ in $L^\infty(\Omega; \R^2)$, the estimate \eqref{limsup:spins:inequality} follows.
\end{proof}

We now prove the existence of recovery sequences for general admissible functions.

\begin{proposition}[Existence of a recovery sequence and limsup inequality] \label{gamma:spins:upper}
 Let $\eps_n\in (0, 1/2)$ and $ \Delta_n = (\delta_n^{hor}, \delta_n^{ver})$ satisfy 
\begin{equation}
\quad \eps_n \rightarrow 0, \quad  \Delta_n = (\dhor_n, \dver_n) \rightarrow (0,0), \quad \frac{\eps_n}{\sqrt{2\dver_n}} \rightarrow \sigma \in (0,\infty), \quad \text{ and }\quad \gamma_n^2=\frac{\dhor_n}{\dver_n} \rightarrow \gamma^2 \in [0,\infty). 
\end{equation}
Then for every $(w, z) \in L^2(\Omega,\R^2)$ there is a sequence $(w_n, z_n) \in L^2(\Omega; \R^2)$ such that $(w_n, z_n)\rightarrow (w, z)$ in $L^2(\Omega;\R^2)$ as $n\rightarrow \infty$ and the limsup inequality holds
\[ \limsup\limits_{n\rightarrow \infty} \Hn(w_n, z_n) \leq \H(w, z).\] 
\end{proposition}
\begin{proof}
Let $(w, z)\in L^2(\Omega;\R^2)$. There are two cases.  \\

\quad \emph{\textbf{Case A: Assume $\bm{\H(w, z) = \infty}$.}} We simply set $(w_n, z_n)\coloneqq (w, z) $ for all $n\in \N.$ Then $(w_n, z_n) \rightarrow (w, z)$ in $L^2(\Omega;\R^2)$ as $n\rightarrow \infty$ and  
\[\limsup\limits_{n\rightarrow \infty} \Hn(w_n, z_n)\leq \infty = H_{\sigma, \gamma}(w, z). \]

\quad \emph{\textbf{Case B: Assume $\bm{(w, z) \in \dom(\H).}$}} 
By Lemma \ref{lem:recovery-sequence_smoot} there is a recovery sequence for every $(w,z) \in D\coloneqq \dom(H_{\sigma,\gamma}^{per})\cap C^\infty(\overline{\Omega})$. By mollification, for every $(w,z) \in \dom(\H)$ there exists a sequence $(w_n,z_n) \in \dom(\H) \cap C^{\infty}(\overline{\Omega};\R^2)$ such that $(w_n,z_n) \to (w,z)$ in $W^{1,2}(\Omega;\R^2)$. Then the existence of a recovery sequence follows from Lemma \ref{lem:recovery-sequence_smoot} and a diagonal argument.
\end{proof}

\section{Proof of the scaling law}\label{sec:scaling}
In this section, we prove Theorem \ref{thm:scaling}. 
 Hence, from now on we will assume that $\delta = \dver = \dhor \in (0,1)$.

\subsection{Upper Bound}
In this section, we prove the upper bounds of Theorem \ref{thm:scaling}.  We refer to Figure \ref{fig:optimal_profiles} for an illustration of the admissible curl-free interfaces between optimal profiles. The latter also play a crucial role in the continuum setting, see \cite{B11-2019-CFO,GiZw22,GiZw23_Martensites}.

\label{sec:upper}
\begin{proposition}[Upper bound] There is a constant $c_u > 0$ such that for all $(\eps,\delta)\in (0,1/2)\times (0,1)$ there holds\label{lem:upper}
\begin{equation*}
    \min \left\{\Eiso(u,\Omega) \colon u\in \Sbound\right\} \leq c_u\min\left\{ \delta^2, \eps\delta^{3/2} \left( \left\vert \ln \left( \frac{\eps}{\delta^{1/2}}\right)\right\vert + 1 \right), \eps\delta^{1/2}\right\}.
\end{equation*}
\end{proposition}
\begin{proof}
We provide separate constructions for the parameter regimes spelled out in Remark \ref{rem:scaling_function_regimes}.\\

\quad \textbf{Ferromagnet. } Let $u^{(1)} \colon \Omega\cap\eps\Z^2 \rightarrow \mathbb{S}^1$ be given by $u^{(1)}_{i,j} = (0,1)^T$ for all $(i,j)\in \Ind$, see Fig.~\ref{fig:ferromagnet}. Then 
\begin{eqnarray}\label{eq:energy1}
\Eiso(u^{(1)}, \Omega) \leq 4\delta^2.
\end{eqnarray}

\quad \textbf{Branching-type construction.}
We construct a spin field $u^{(2)}$ satisfying the energy estimate $\Eiso(u^{(2)}, \Omega) \leq C_2 \delta^{3/2} \eps ( |\ln ( \eps / \delta^{1/2})| + 1 )$. 
  First note, that it holds $\delta^2 \leq 2\delta^{3/2} \eps ( |\ln ( \eps / \delta^{1/2})| + 1 )$ if $\delta^{1/2} \leq 2 \eps$ . 
  Hence, in light of the ferromagnetic competitor above it is enough to consider the case $2\eps \leq \delta^{1/2}$.
  
  Let $N\in \N, N\geq 1$, with $2^{-N} \leq \frac{\eps}{\delta^{1/2}}$ and  $\lambda \in (0,2^{-N})$ be fixed. 
  We will construct a spin field $u^{(2)}_{N,\lambda} \colon \Omega \cap\eps\Z^2 \rightarrow \mathbb{S}^1$ with associated angular velocity field $(\hor_{N,\lambda}, \ver_{N, \lambda}) \colon \Omega \cap \eps \Z^2 \rightarrow [-\pi, \pi)^2$ that satisfies the energy estimate
  \begin{equation}
      \label{ineq:energy_constr_A}
      \Eiso (u^{(2)}_{N,\lambda}, \Omega ) \leq C \left( N \delta^2 \lambda + \delta^2 2^{-N} + \frac{\eps^2 \delta}{\lambda} N + \frac{\eps^2 \delta}{\lambda^2} 2^{-N} \right).
  \end{equation}
    Setting 
  \begin{equation*}
     N \coloneqq \left\lceil\left\vert  \frac{\ln \left( \frac{\eps}{\delta^{1/2}}\right)}{\ln 2}\right\vert \right\rceil  \quad  \textup{   and   }\quad  \lambda \coloneqq \frac{\eps}{2\delta^{1/2}}\in (0,2^{-N})\subseteq(0,1),
  \end{equation*}
  yields the energy estimate 
\begin{equation}
    \label{eq:energy2} 
     \Eiso (u^{(2)}_{N,\lambda}, \Omega ) \leq C_2 \eps \delta^{3/2}\left( \left\vert \ln \left(\frac{\eps}{\delta^{1/2}}\right) \right\vert + 1 \right)
\end{equation}
with $C_2 := 8 C / \ln(2) $.\\
For that, we construct a curl-free angular velocity field $(\theta_{N,\lambda}^{hor}, \theta_{N,\lambda}^{ver})$ that satisfies the boundary conditions
  $\theta_{N,\lambda}^{ver} = 0$ on $(\{0\}\times (0,1))\cap\eps \Z^2$, which then determines the spin field. We use a self-similar branching-type construction (see Figure \ref{fig:branching_construction}) which is motivated by a construction from \cite{GiZw22} for a related sharp-interface continuum model. \\
  
  \begin{figure}[h]
      \centering
      \includegraphics[scale = 0.35]{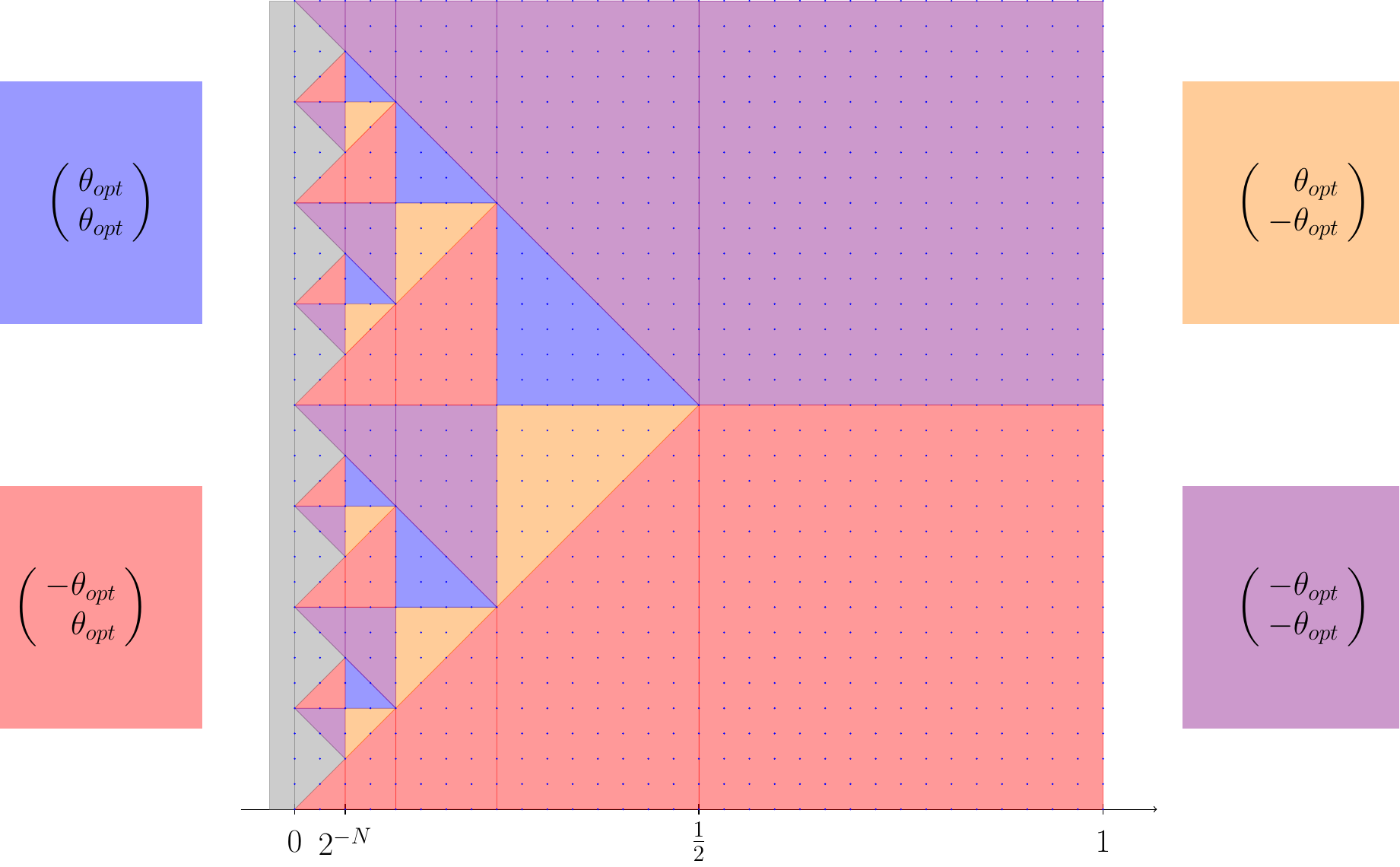}
      \caption{Sketch of the construction for $\eps = 1/32$ and $N = 3$. The different colored domains mark the subdomains in which we turn with one of the four optimal angular velocity fields (recall Figure \ref{fig:optimal_profiles}).}
      \label{fig:branching_construction}
  \end{figure}
 
  \quad \emph{\textbf{Step 1. Definition of the building block.}}
  Let $\opt \coloneqq \arccos(1-\delta)$. We introduce the angular velocity field $(\Psi^{hor}, \Psi^{ver}) \colon [1/4,1/2]\times \R \rightarrow [-\pi, \pi)^2$ given by
  \begin{equation}
      \label{def:bulding_block}
    \Psi^{hor}(x, y) = \begin{cases}
    -\theta_{opt} \quad & \textup{ if } y\leq x, \\
    \phantom{-}\theta_{opt} & \textup{ if } x\leq y \leq \frac12,\\
    \phantom{-}\opt & \textup{ if } \frac12\leq  y \leq 1 - x,\\
    -\opt & \textup{ if } \frac12 \leq 1 - x \leq y,
    \end{cases}\quad  \textup{ and }\quad
        \Psi^{ver}(x, y) = \begin{cases}
    \phantom{-}\theta_{opt} \quad & \textup{ if } y\leq x, \\
    -\theta_{opt} & \textup{ if } x\leq y \leq \frac12,\\
    \phantom{-}\opt & \textup{ if } \frac12\leq  y \leq 1 - x,\\
    -\opt & \textup{ if } \frac12 \leq 1 - x \leq y,
    \end{cases}
  \end{equation}
  for $(x,y) \in [1/4,1/2]\times [0,1]$ and extend $(\Psi^{hor}, \Psi^{ver})$ one-periodically with respect to the $y$-variable to $[1/4,1/2]\times \R$.
  Note that $(\psi^{hor},\psi^{ver})$ is curl-free.
  Hence, for $\eps \in \{ 2^{-n} \colon n \geq N\}$ we have \[\curl((\Psi^{hor}, \Psi^{ver}))= 0 \quad \textup{ on } ([1/4,1/2] \times \R) \cap \eps\Z^2,\]
  where we identify $(\psi^{hor},\psi^{ver})$ with a function on $([1/4, 1/2]\times \R) \cap \eps \Z^2$ through restriction.\\
  
  \quad \emph{\textbf{Step 2. Refinement on $\bm{[2^{-N}, \infty) \times {\rm I\!R}}$.}} First, let $\eps \in \{2^{-n} \colon n\geq N\}$. We define the angular velocity field $(\tilde{\Psi}_N^{hor}, \tilde{\Psi}^{ver}_N) \colon [2^{-N}, \infty)\times \R \rightarrow [-\pi, \pi)^2$ by 
  \begin{equation}
      \nonumber 
      \tilde{\Psi}^{hor}_N (x,y) \coloneqq \Psi^{hor}(2^k x, 2^k y) \quad \textup{ and } \quad \tilde{\Psi}^{ver}_N (x, y) \coloneqq \Psi^{ver}(2^k x, 2^k y) 
  \end{equation}
  for $(x, y) \in [2^{-k-2}, 2^{-k-1}) \times \R $ and $k \in \{0, \dots, N-2\}$. 
  Furthermore, we set 
  \begin{equation}
      \tilde{\Psi}_N^{hor}(x,y) \coloneqq   \Psi^{hor}(1/2, y) \quad \textup{ and } \quad    \tilde{\Psi}_N^{ver}(x,y) \coloneqq  \Psi^{ver}(1/2, y) \quad \textup{ for } \frac12 \leq x. 
  \end{equation}
This extension to $[2^{-N}, \infty)\times \R$ satisfies $\eps \textup{curl}^{d,\eps} ((\tilde{\Psi}_N^{hor}, \tilde{\Psi}_N^{ver}))= 0 $ on $([2^{-N}, \infty)\times \R)\cap \eps \Z^2$.\\
  \begin{figure}[h]
      \centering
      \includegraphics[width =0.95\linewidth]{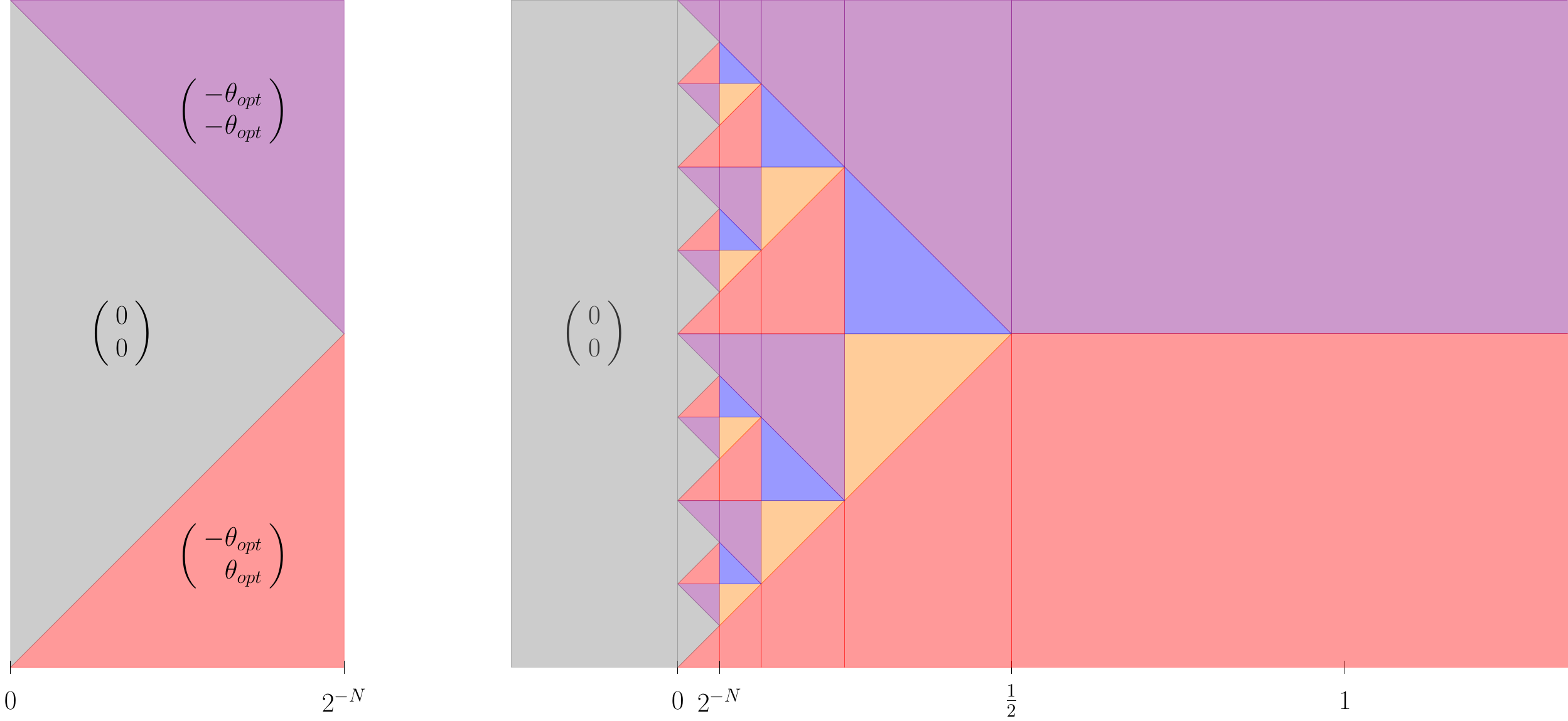}
      \caption{Left: The angular velocity field $(\Psi^{hor}, \Psi^{ver})$ as constructed close to the boundary. Right: Sketch of the extension $(\Psi^{hor}_N, \Psi^{ver}_N)$ of the construction for $N= 3.$ }
      \label{fig:branching_continuous} 
  \end{figure}
  
For arbitrary $\eps \in (0,1/2)$ there is $n\geq N$ such that $\eps\in [2^{-n}, 2^{-(n-1)})$. We define
\begin{equation}
    \Psi_N^{hor}\coloneqq \tilde{\Psi}_N^{hor}\left(\frac{2^{-n}x}{\eps}, \frac{2^{-n}y}{\eps} \right) \quad \textup{ and }\quad  \Psi_N^{ver}\coloneqq \tilde{\Psi}_N^{ver} \left(\frac{2^{-n}x}{\eps}, \frac{2^{-n}y}{\eps} \right)
\end{equation}
for $(x,y)\in [2^{n-N}\eps, \infty)\times \R.$ Note that this definition guarantees that the interfaces are aligned with the $\eps \Z^2$-lattice. Again, we have $\eps \curl (( \Psi_N^{hor}, \Psi^{ver}_N)) = 0$ on $ ([2^{n-N}\eps, \infty)\times \R )\cap \eps\Z^2.$ From now on, for simplicity of notation we assume $\eps \in \{2^{-n}\colon n\geq N\}$. \\

\quad \emph{\textbf{Step 3. Construction in the boundary layer. }}We set for $(x,y) \in [0, 2^{-N}] \times [0, 2\cdot 2^{-N}]$ 
  \begin{equation}
      \label{def:bulding_block_boundary}
    \Psi_N^{hor}(x, y) \coloneqq \begin{cases}
    -\theta_{opt} \quad & \textup{ if } 0\leq y\leq x\leq 2^{-N},\\
    -\theta_{opt} & \textup{ if } 0\leq 2\cdot 2^{-N} -x\leq y \leq 2\cdot 2^{-N}, \\
    \phantom{-}0 & \textup{ else},
    \end{cases} \nonumber 
  \end{equation}
  and 
  \begin{equation}
            \Psi_N^{ver}(x, y) \coloneqq \begin{cases}
    \phantom{-}\theta_{opt} \quad & \textup{ if }0\leq y\leq x\leq 2^{-N}, \\
       -\opt & \textup{ if } 0\leq 2\cdot 2^{-N} -x\leq y \leq 2\cdot 2^{-N},\\
    \phantom{-}0& \textup{ else}.
    \end{cases}\nonumber 
  \end{equation}
  Again, we extend $\Psi_N^{hor}$ and $\Psi_N^{ver}$ with respect to the $y$-variable $2\cdot 2^{-N}$-periodically, and eventually, set $(\Psi_N^{hor}, \Psi_N^{ver}) \coloneqq (0,0)$ on $(-\infty, 0] \times \R.$ Then again $\curl(\Psi_N^{hor}, \Psi_N^{ver})) = 0$. 
 Eventually, note that there is a universal constant $C>0$ such that the length of the jump set of $(\psi^{hor}_N,\psi^{ver}_N)$, $J_{(\psi^{hor}_N,\psi^{ver}_N)}$, is estimated by 
 \begin{equation}\label{eq: estimate surface psi}
 \mathcal{H}^1(J_{(\psi^{hor}_N,\psi^{ver}_N)} \cap (0,2) \times (-1,1)) \leq C N.
 \end{equation}
 We refer to Figure \ref{fig:branching_continuous} for a sketch of the described construction.\\
  
\quad \emph{\textbf{Step 4. Mollification.}} Let $(\Psi_N^{hor}, \Psi^{ver}_N) \colon \R^2 \rightarrow [-\pi, \pi)^2$ be the field from Step 2 and Step 3. Furthermore, let $p_\lambda\colon \R \rightarrow [0,\infty)]$ be a standard mollifier, i.e.~  
  \begin{equation}
     \label{prop:mollifier}
      p_\lambda \in C_c^\infty(\R;[0,\infty)), \quad \textup{supp}(p_\lambda) \subseteq (-\lambda, \lambda), \quad \int_\R p_\lambda \textup{ d}x = 1, \quad \text{ and } \quad \int_\R \vert p_\lambda^\prime\vert  \textup{ d}x \leq \frac6\lambda.
  \end{equation}
 We define the mollified angular velocity field $(\Psi_{N,\lambda}^{hor}, \Psi_{N, \lambda}^{ver})\colon \R^2 \rightarrow [-\pi, \pi)^2$ by (see Figure \ref{fig:smooth_spins} for a sketch)
  \begin{equation}
  \label{def:smoothed_branching}
  \begin{array}{rcl}
  \Psi_{N,\lambda}^{hor} (x,y) &\coloneqq &\int\limits_{\R^2} p_{\lambda}(s) p_\lambda(t) \Psi_N^{hor}(x-\lambda-s, y-t) \textup{ d}s \textup{d}t,\text{\quad and}\\
    \Psi_{N,\lambda}^{ver} (x,y) &\coloneqq & \int\limits_{\R^2} p_{\lambda}(s) p_\lambda(t) \Psi_N^{ver}(x-\lambda-s, y-t) \textup{ d}s \textup{d}t\text{\quad for $(x,y) \in \R^2$.}
  \end{array}
  \end{equation}
\begin{figure}[h]
    \centering
    \includegraphics[scale = 0.8]{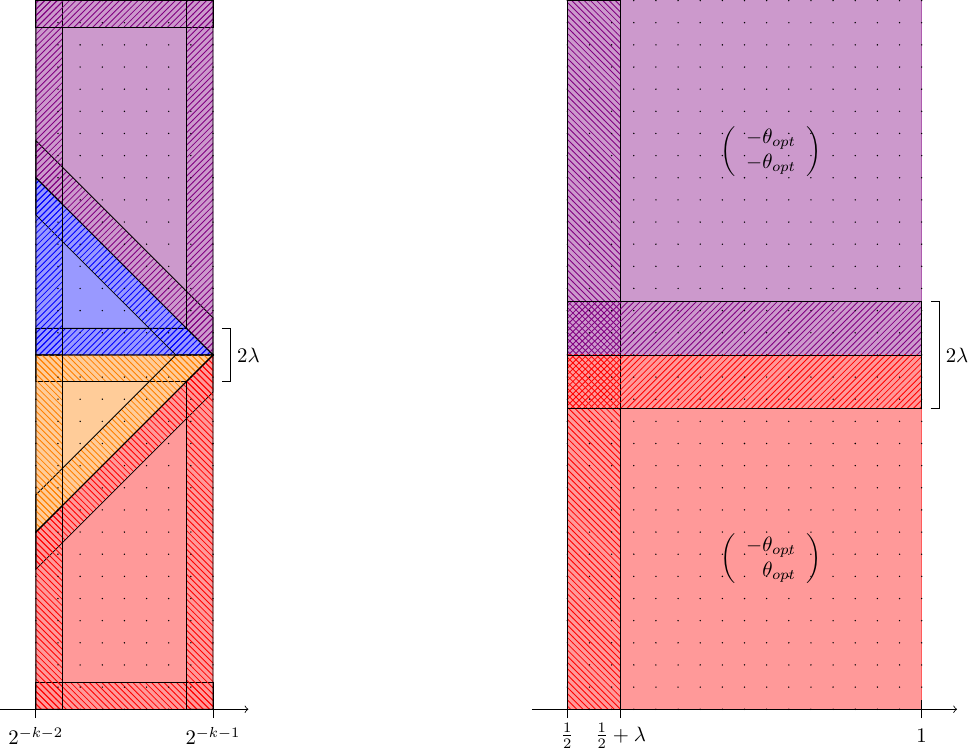}
    \caption{Illustration of the $(\Psi_{N,\lambda}^{hor}, \Psi_{N, \lambda}^{ver})$. The patterned subdomain indicates the set in which the angular velocities are not necessarily contained in $\{-\opt, \opt\}$.}
    \label{fig:smooth_spins}
\end{figure}

Since $\curl((\Psi_{N}^{hor}, \Psi_{N}^{ver})) = 0$ it follows 
  \begin{equation}
      \label{prop:curl_condition}
      \eps \curl((\Psi_{N,\lambda}^{hor}, \Psi_{N,\lambda}^{ver})) = 0 \quad \text{ on } \R^2 \cap \eps \Z^2, 
  \end{equation}
  where we identify the functions $(\Psi_{N,\lambda}^{hor}, \Psi_{N,\lambda}^{ver})$ with their restrictions to $\R^2 \cap \eps \Z^2$.
  Moreover, by construction $\Psi_N^{ver}= 0$ on $(-\infty, 0] \times \R$ and therefore also
  \[ \Psi_{N,\lambda}^{ver}(0, y) = \int_{R^2} \Psi^{ver}_N(0-\lambda-s, y-t) p_\lambda(s) p_\lambda(t) \textup{ d}s\textup{d}t = 0.\]\\
 
 \quad \emph{\textbf{Step 5. Definition of $\bm{(\theta_{N,\lambda}^{hor}, \theta_{N,\lambda}^{ver}) }$.}}
We set
     $(\theta_{N,\lambda}^{hor}, \theta_{N,\lambda}^{ver}) \coloneqq 
     (\Psi_{N,\lambda}^{hor}, \Psi_{N,\lambda}^{ver})$ on $\Omega \cap \eps \Z^2$.
By construction, $\theta_{N,\lambda}^{ver}(0, \cdot) = \Psi_{N,\lambda}^{ver}(0, \cdot)=0$ and $\eps \curl( (\theta_{N,\lambda}^{hor}, \theta_{N,\lambda}^{ver})) =0$ on $\Omega \cap \eps\Z^2$. Hence, there exists an associated spin field $u_{N,\lambda}^{(2)}\in \Sbound$. \\
 
\quad \emph{\textbf{Step 6. Estimating the energy of $\bm{u_{N,\lambda}^{(2)}}$.}} We present only the estimate for the horizontal contribution, the vertical contributions can be treated similarly. Since $\Psi_N^{hor}$ is piecewise constant by construction, the function $\Psi_N^{hor, \lambda}$ is also piecewise constant outside a $\lambda$-neighborhood of the interfaces and away from $(0,\lambda) \times (0,1)$. Also, outside these regions $\Psi_N^{hor, \lambda} \in \{\pm \opt\}$.
We define the set
\[
\mathcal{I}_\lambda \coloneqq \left\{ (i,j)\colon (\theta^{hor}_{N,\lambda})_{i, j} \notin \{\pm \opt\} \text{ or } (\theta^{hor}_{N,\lambda})_{i+1, j} \notin\{\pm \opt\} \text{ or } (\theta^{hor}_{N,\lambda})_{i, j} \neq (\theta^{hor}_{N,\lambda})_{i+1, j} \right\}.\]
By the reasoning above and \eqref{eq: estimate surface psi}, we have $\mathcal{H}^0(A_{\lambda}) \leq C\left( N \frac{\lambda}{\eps^2}  + 2^{-N} \frac{1}{\eps^2} \right)$. Moreover, the estimate $|\theta^{hor}_{N,\lambda}| \leq \theta^{opt} = \arccos(1-\delta) \leq 2 \delta^{1/2}$ holds. It follows by Lemma \ref{lem:p_convergence} that \[p\left((\theta^{hor}_{N,\lambda})_{ i+1, j}, (\theta^{hor}_{N,\lambda})_{i,j}\right)\leq 3/2 \quad \text{ for all } (i, j)\in \INNhor.\]
Eventually, note that by the properties of $\rho_{\lambda}$ (recall \eqref{prop:mollifier}) we have $\left\vert \theta^{hor}_{N,\lambda, i+1, j} - \theta^{hor}_{N,\lambda, i, j} \right\vert \leq 12 \eps \delta^{1/2} / \lambda$. 
Hence, we may estimate
\begin{equation}\begin{split}
 \Eisohor (u^{(2)}_{N,\lambda},\Omega) = & \eps^2\sum_{(i,j) \in \INhor} (1-\delta -\cos((\theta^{hor}_{N,\lambda})_{i,j}))^2 +\eps^2\sum_{(i+1,j) \in \INhor} (1-\delta -\cos((\theta^{hor}_{N,\lambda})_{i+1,j}))^2 \nonumber\\
& + 2\eps^2 \sum_{(i,j)\in \INNhor} p((\theta^{hor}_{N,\lambda})_{i,j}, (\theta^{hor}_{N,\lambda})_{i+1, j}) \left\vert \sin\left(\frac{(\theta^{hor}_{N,\lambda})_{i+1,j}}{2}\right)- \sin\left(\frac{(\theta^{hor}_{N,\lambda})_{i,j}}{2}\right)\right\vert^2 \\
\leq &2 \eps^2 \delta^2 \mathcal{H}^0(A_\lambda) + 3 \eps^2 \sum_{(i,j)\in \mathcal{I}_\lambda} \left\vert \theta^{hor}_{N,\lambda, i+1, j} - \theta^{hor}_{N,\lambda, i, j} \right\vert^2 \\
\leq &  C\mathcal{H}^0(A_\lambda) \left( 2 \eps^2 \delta^2 + \frac{\eps^4 \delta }{\lambda^2}  \right) \\
\leq& C \left( N \delta^2 \lambda + \delta^2 2^{-N} + \frac{\eps^2 \delta}{\lambda} N + \frac{\eps^2 \delta}{\lambda^2} 2^{-N} \right).
\end{split}\label{reformulation:angle-velocity}
\end{equation}
This concludes the proof of \eqref{ineq:energy_constr_A}, and hence of the upper bound \eqref{eq:energy2}. \\

\quad \textbf{Vortex structure.} We construct a competitor $u^{(3)}\colon \Omega \cap\eps\Z^2 \rightarrow \mathbb{S}^1$ by introducing the associated angular velocity field $(\hor_3, \ver_3) \colon \Omega \cap\eps\Z^2 \rightarrow [-\pi,\pi)^2$. Again, we start by defining a building block that we will then use to construct the global angular velocity field. \\

\quad \emph{\textbf{Step 1. Construction of the building block.}} Let $M\coloneqq \left\lfloor \pi/\opt\right\rfloor\geq 2$. We define the angular velocity field $(\Psi^{hor},\Psi^{ver})\colon ([0, M\eps] \times [0, 2M\eps]) \cap\eps\Z^2 \rightarrow [-\pi,\pi)^2$ by 
\begin{equation}
    \Psi^{hor}_{i,j} \coloneqq \begin{cases}
    \frac{\pi}{j}\quad & \textup{ if } 0\leq i < j\leq M, \\
      j\frac{\pi}{i(i+1)}\quad  & \textup{ if } 0\leq j \leq i, \\
     \opt & \textup{ if } M+1\leq j \leq M+i,\\
     \phantom{-}0 & \textup{ else,}
    \end{cases}
\end{equation}
and 
\begin{equation}
    \Psi^{ver}_{i,j} \coloneqq \begin{cases}
   - i\frac{\pi}{j (j+1)}\quad & \textup{ if } 0\leq i\leq j\leq M-1,  \\
     -i\left(\frac{\pi}{M} -\opt\right) \quad & \textup{ if } j= M,\\ 
     -\frac{\pi}{i}\quad  & \textup{ if } 0\leq j< i \textup{ and } 1 < i \\
     -\opt & \textup{ if } M+1\leq j \leq M+i,\\
     \phantom{-}0 & \textup{ else.}
    \end{cases}
\end{equation}
Furthermore, we extend $(\Psi^{hor}, \Psi^{ver})$ $2M\eps$-periodically in the vertical direction to $([0, M\eps] \times [0,1] )\cap\eps\Z^2$. The field $(\Psi^{hor}, \Psi^{ver})$ satisfies the discrete curl condition $\eps\curl (\Psi^{hor}, \Psi^{ver})_{i,j} = 2\pi\delta_{i1} \delta_{j1}$ on $([0, M\eps] \times [0, 1])\cap\eps\Z^2$, i.e.~we create one vortex per building block, see Figure \ref{fig:building_block_vortices} for a sketch. Let $u^{(3)} \colon ([0, M\eps]\times [0,1])\rightarrow \mathbb{S}^1$ be the associated spin field. \\ 
\begin{figure}[h]
    \centering
    \includegraphics[scale = 0.3]{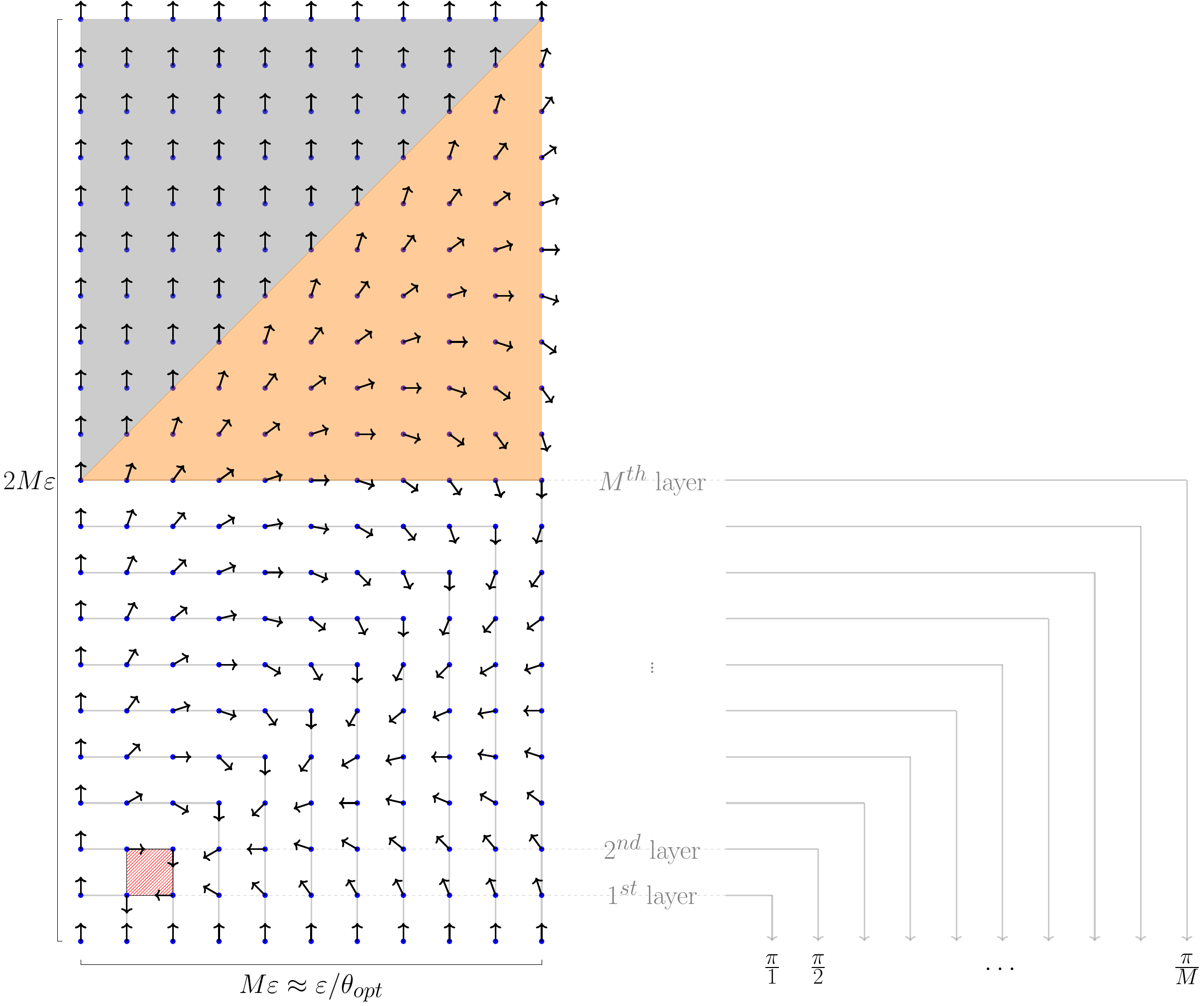}
    \caption{Left: Sketch of the competitor $u^{(3)} \colon ([0, M\eps] \times [0, 2M\eps] )\cap\eps\Z^2\rightarrow \mathbb{S}^1$ using $M=\lfloor \pi / \opt\rfloor$ vortices. The red $\eps$-cells mark cells containing a vortex. Right: Illustration of the angular velocity field of $u^{(3)}$ along an edge in $[0, M\eps]^2\cap\eps\Z^2$ with $\pi/(M+1)\leq \opt \leq \pi/M$. Along the $j$-th layer the associated spin field rotates with the angle $\pi/j$.} 
    \label{fig:building_block_vortices}
\end{figure}

\quad \emph{\textbf{Step 2. Definition of $\bm{(\hor_3, \ver_3).}$}} We define  
\begin{equation}
    (\hor_3, \ver_3) \coloneqq \begin{cases}
        (\Psi^{hor}, \Psi^{ver}) \quad & \textup{ on } [0, M\eps]\times [0,1],\\
        (\opt, -\opt) & \textup{ else.}
    \end{cases}
\end{equation}
The angular velocity field $(\hor_3, \ver_3)$ satisfies the discrete curl-condition and  $\ver_{0, \cdot} = 0$. Hence, there is a spin field $u^{(3)}\in \Sbound$ whose angular velocity field is $(\hor_3, \ver_3)$, see Figure \ref{fig:vortex} for a sketch of $u^{(3)}.$\\

\quad \emph{\textbf{Step 3. Estimating the energy of the building block.}} We start by computing the energy of $(\hor_3,\ver_3)$ on $[0, M\eps] \times [0, 2M\eps]$. Again, we use the reformulation of the energy \eqref{gamma:spins:energie_one} and \eqref{spin:gamma:reformulation_vertical}. \\

First, we compute the energy of $(\hor_3, \ver_3)$ on $[0, M\eps]\times [(M+1)\eps, 2M\eps]$. Subsequently, we estimate the energy on $[0, M\eps] \times [0, M\eps]$.  Let 
\[T^u \coloneqq \{ (i,j) \in \{0, \dots, M\} \times \{M+1, \dots, 2M\}\colon i\leq j-M \} \] be the upper triangle and \[T^l \coloneqq \{ (i,j) \in \{0, \dots, M\} \times \{M+1, \dots, 2M\}\colon j-M\leq i \} \] the lower triangle of the upper part of the building block. We have $(\hor_3,\ver_3) \equiv (\opt, -\opt)$ in $T^l$, and $(\hor_3, \ver_3) \equiv (0,0)$ in $T^u.$ We denote the boundary between $T^l$ and $T^u$ by 
 \[\partial T \coloneqq \{(i, M+i) \colon i\in \{1, \dots, M\}\}.\] The horizontal energy contribution is bounded by 
\begin{equation}\begin{split}
& \Eisohor(u^{(3)} , [0, M\eps] \times [(M+1)\eps, 2M\eps]) \\
& \leq 2\eps^2 \sum_{(i,j)\in T^u} (1-\delta -\cos(0))^2 + 2\eps^2 \sum_{(i,j) \in \partial T } \vert p(\Psi^{hor}_{i+1,j}, \Psi^{hor}_{i,j}) \vert \left\vert \sin\left(\frac{\Psi^{hor}_{i+1,j}}{2}\right) - \sin\left(\frac{\Psi^{hor}_{i,j}}{2}\right)\right\vert^2 \\
& \leq 2\eps^2 \delta^2 \mathcal{H}^0(T^u) + 3 \eps^2 \sum_{(i,j) \in \partial T }  \left\vert \hor_{i+1,j} - \hor_{i,j}\right\vert^2 \\ &\leq M(M+1) \eps^2 \delta^2 + 12 (M+1) \eps^2 \delta \leq 2 (M\eps)^2 \delta^2 + 24 M\eps^2 \delta, \end{split}
\end{equation}
where we used $\vert p(\theta_1, \theta_2)\vert \leq 3/2 $ for all $(\theta_1, \theta_2) \in [-\pi/2, \pi/2]$, see Lemma \ref{lem:p_convergence}, and $\arccos(1-\delta) \leq 2\delta^{1/2}$. The vertical energy contribution can be estimated analogously. Thus, 
\begin{equation}
    \Eiso(u^{(3)}, [0, M\eps]\times [M\eps, 2M\eps] )\leq 4(M\eps)^2 \delta^2 + 48 M\eps^2 \delta.\label{upper:spins:building_upper}
\end{equation}

Next, we calculate the energy on $[0, M\eps]^2$. As before, we introduce the upper triangle \[D^u \coloneqq \{ (i,j) \in \{0, \dots, M\}^2 \colon 0\leq i\leq j \}\] and the lower triangle \[D^l \coloneqq \{ (i,j) \in \{0, \dots, M\}^2 \colon 0\leq j\leq i \}\] of this domain. We start by calculating the vertical contribution of the energy on $[0, M\eps]^2$. Note that we do not have $\opt = \pi/M$ in general. Therefore, the vertical angular velocity is not optimal on $\{0, \dots, M\eps\} \times \{ M\eps\} $ but satisfies the identity 
\[ \Psi^{ver}_{i,M} = i\left( \opt- \frac{\pi}{M}\right) \in \left\lbrack -\frac{\pi}{M+1}, 0\right\rbrack\subseteq [-\opt, 0], \] where we used that $\pi/(M+1)\leq \opt\leq \pi/M.$ Hence, 
\begin{align}
 \Eisover(u^{(3)}, D^u) &\leq 2\eps^2 \sum_{(i,j)\in D^u } (1-\delta -\cos(\Psi^{ver}_{i,j}))^2 + 6 \eps^2 \sum_{(i,j)\in D^u} (\Psi^{ver}_{i,j+1} - \Psi^{ver}_{i,j})^2  \\
& \leq 2\eps^2 \sum_{(i,j)\in D^u }  \left( \int_{\opt}^{i\pi/(j(j+1))} \sin(s) \textup{ d}s \right)^2 +6\eps^2 \sum_{(i,j)\in D^u} \left(\frac{i\pi}{(j+1)(j+2)} - \frac{i\pi}{j(j+1)}\right)^2 \\
& \leq \eps^2 \sum_{(i,j)\in D^u}  \left( \int_{\opt}^{i\pi/(j(j+1))} s \textup{ d}s \right)^2 +24 \eps^2 \sum_{(i,j)\in D^u} \left(\frac{i\pi}{j(j+1)(j+2)}\right)^2 \\
& = \frac12 \eps^2 \sum_{(i,j)\in D^u }  \left( \left(\frac{i\pi}{j(j+1)}\right)^2 - \opt^2 \right)^2 +24 \eps^2 \sum_{(i,j)\in D^u} \left(\frac{i\pi}{j(j+1)(j+2)}\right)^2 \phantom{hallo hallo}\\
& \leq   \eps^2 \sum_{j=1 }^M \sum_{i=1}^j   \left(\frac{i\pi}{j(j+1)}\right)^4+ 8(M\eps)^2\delta^2 + 24 \eps^2 \sum_{j=1}^M \sum_{i=1}^j  \left(\frac{i\pi}{j(j+1)(j+2)}\right)^2\\
& \leq \pi^4 \eps^2 \sum_{j=1}^M \frac{1}{j^3} + 8(M\eps)^2 \delta^2 + 24\pi^2 \eps^2 \sum_{j=1}^M \frac{1}{j^3}\leq 8\pi^4 \eps^2+ 8(M\eps)^2 \delta^2.
\end{align}
Similar estimates hold for the horizontal energy contribution, and the contributions from the lower triangle $D^l.$ Combining these estimates with \eqref{upper:spins:building_upper} yields a universal $C>0$ such that 
\begin{equation}
    \label{upper:energy_block_vortex}
    \Eiso(u^{(3)}, [0,M\eps]\times [0, 2M\eps])\leq C ( \eps^2 + M^2 \eps^2 \delta^2 + M \eps^2 \delta).
\end{equation}

\quad \emph{\textbf{Step 4. Global energy estimate.}} Note that the only energy contribution comes from the boundary region $[0,M\eps] \times [0,1]$, which contains the building block at most $\lceil 1/(2M\eps)\rceil \leq 1/ (M \eps)$ times. 
Therefore, the total energy is bounded by 
\begin{equation}\label{eq:energy3}
    \Eiso(u^{(3)}, \Omega) \leq C \left( \frac{\eps}M + \eps M \delta^2 + \eps \delta \right) \leq 7 C \eps \delta^{1/2}, 
\end{equation}
where we used that $M \leq \pi / \opt \leq \pi / \delta^{1/2}$ and $M \geq \pi / (2\opt) \geq \pi/(4\delta^{1/2})$. Finally, we set $C_3 = 7C$. \\

Combining \eqref{eq:energy1}, \eqref{eq:energy2} and \eqref{eq:energy3} concludes the proof of the lemma with $c_u\coloneqq \max\{4, C_2, C_3\}$.
\end{proof}
\begin{remark}
Note that the number of vortices used by the third competitor $u^{(3)} \colon \Omega \cap\eps\Z^2 \rightarrow \mathbb{S}^1$ constructed above is $2\opt/(2\pi\eps) \approx 2\delta^{1/2}/(\pi\eps)$. On the other hand, by Lemma \ref{lem:vortices-estimates} it holds for the set of vortices $V$ for any spin field that $\eps^2 \mathcal{H}^0(V) \lesssim \Eiso(u,\Omega)$. Hence, for the competitor $u^{(3)}$ we obtain theoretically that $\eps^2\mathcal{H}^0(V) \lesssim \eps \delta^{1/2}$. In particular, the estimate from Lemma \ref{lem:vortices-estimates} is (up to multiplicative constants) sharp for the competitor $u^{(3)}$.
\end{remark}

\subsection{Lower Bound}
\label{sec:lower}

In this section, we aim to prove the lower bound of the scaling law. For this, we will treat the situation of large $\delta \in [\overline{\delta}, 1)$ and the case of small $\delta \in (0, \overline{\delta})$ for some $\overline{\delta}\in (0,1)$ separately. The critical value $\overline{\delta}$ will be chosen so that Lemma \ref{lem:lower_bound_q} holds. 
 We build on the strategy from the continuum setting (see \cite{GiZw22} and the references given there). As the continuum argument was developed for gradient fields this strategy can only be successfully applied in the absence of vortices. The main result of this section reads as followes. 
\begin{proposition}[Lower bound] \label{prop:lower} There is $\overline{\delta} \in (0,1) $ such that the following is true. There is a constant $c>0$, only depending on $\overline{\delta}$, such that the following two statements are true:
\begin{itemize}
    \item[(i)] For all $(\eps, \delta) \in (0,1/2) \times [\overline{\delta}, 1)$ it holds
    \[ \min\{ \Eiso(u, \Omega) \colon u \in \Sbound \} \geq c \eps \delta^{1/2}.\]
    \item[(ii)] For all $\delta \in (0, \overline{\delta})$ and $\delta \leq \eps $ it holds
    \[ \min \{ \Eiso(u, \Omega) \colon u \in \Sbound\} \geq c \min\left\{ \delta^2, \eps\delta^{3/2} \left( \left\vert \ln \left( \frac{\eps}{\delta^{1/2}}\right)\right\vert + 1 \right) \right\}.\]
\end{itemize}
\end{proposition}
\begin{proof}
    Let $\eta\in (0,1)$ be the constant introduced in Lemma \ref{lem:lower_bound_no_vortices} and $\overline{\delta} \coloneqq \min\{\eta^2/16, 1-\cos(\pi/16)\}$. The first statement (i) follows from Lemma \ref{lem:lower_close_to_one} and Lemma \ref{lem:lower_close_inbetween} as follows. Let $\tilde{c}$ be the maximal constant from both Lemmas. For $(\eps,\delta) \in (0,1/2) \times (63/64, 1)$ we obtain from Lemma \ref{lem:lower_close_to_one} the lower bound 
    \begin{equation}
        \Eiso(u, \Omega) \geq \tilde{c} \eps \geq\tilde{c}\eps\delta^{1/2} \quad \textup{ for all } u \in \Sbound.
    \end{equation}
    Lemma \ref{lem:lower_close_inbetween} implies for $(\eps,\delta) \in (0,1/2) \times [\overline{\delta}, 63/64)$
    \[ \Eiso(u, \Omega) \geq \tilde{c} \eps \delta^2 \geq \tilde{c} \overline{\delta}^{3/2} \eps\delta^{1/2}.\]
    For the second statement we separately discuss the case large and small $\eps$ i.e., $\eps \in [\eta/4, 1/2)$ and $\eps \in (0, \eta/4).$ If $(\eps, \delta) \in [\eta/4, 1/2)\times (0, \overline{\delta})$ we apply Lemma \ref{lem:lower_close_inbetween} to obtain 
    \[ \Eiso(u, \Omega) \geq \tilde{c} \eps \delta^2 \geq \tilde{c} \frac\eta4 \delta^2 \geq   \tilde{c} \frac\eta4 \min\left\{ \delta^2, \eps\delta^{3/2} \left( \left\vert \ln \left( \frac{\eps}{\delta^{1/2}}\right)\right\vert + 1 \right) \right\}.\]
   The lower bound for $(\eps, \delta) \in (0, \eta/4)\times (0, \overline{\delta})$ with $\delta \leq \eps$ is proven in Lemma \ref{lem:lower_bound_no_vortices} for a constant constant $c_l >0$. Then (i) and (ii) follow for $c \coloneqq \min\{ \tilde{c}\overline{\delta}^{3/2}, \tilde{c}\eta/4,  c_l\}$. 
\end{proof}

\subsubsection{Lower bound for large $\delta \in [\overline{\delta},1)$}
Using the definition of the energy $\Eiso(\cdot, \Omega)$ and the boundary condition $u_{0,\cdot} = (0,1)^T$, one easily obtains from the energy on indices of the form $(0,j)$ the estimate $\Eiso(u,\Omega) \geq \eps\delta^2/2$ for any spin field $u\in \mathcal{S}_0(\Omega;\eps)$. However, we will prove the same lower bound even without considering the energy contained in the left boundary. 
\begin{lemma}\label{lem:lower_close_to_one}
There is a constant $c > 0$ such that for all $(\eps,\delta) \in (0, 1/2)\times  (63/64, 1)$ and $u\in \Sbound$ it holds
\[\Eiso(u, [0,2\eps]\times [0,1]) \geq c\eps \]
\end{lemma}
\begin{proof}
Let $u\in \Sbound$. Recall, for any unit vectors $u,v,w\in \mathbb{S}^1$ with $\theta_{u,v} = \sgn(u\times v)\arccos(u\cdot v)$ and $\theta_{v,w} = \sgn(v\times w) \arccos(v\cdot w)$ the identity 
\begin{equation}\begin{split}
& \vert u -2(1-\delta) v + w\vert^2 = 2 + 2\cos(\theta_{u,v} +\theta_{v,w}) -4(1-\delta) \cos(\theta_{u,v}) - 4(1-\delta) \cos(\theta_{v,w}) +4(1-\delta)^2  \\
& \quad = 4\cos^2\left(\frac{\theta_{u,v} +\theta_{v,w}}{2} \right)-4(1-\delta) \cos(\theta_{u,v}) - 4(1-\delta) \cos(\theta_{v,w}) +4(1-\delta)^2, \end{split}
\end{equation}
where we used that  $\cos(2\theta) = 2\cos^2(\theta) -1$ for $\theta \in \R$. Let $(\hor, \ver)\colon \Omega \cap\eps\Z^2 \rightarrow [-\pi, \pi)^2$ the angular velocity field associated to $u$ and $j\eps\leq 1-2\eps\in [0,1)$. We will show that 
\begin{equation}\label{eq: lower bound large delta 1}
  \Eiso(u,[0,2\eps]\times [j\eps, (j+2)\eps]) \geq \frac1{16} \eps^2.  
\end{equation} 
First we argue that we may assume that
\begin{equation}
    \label{ass:two_angles}
    \frac{\vert \hor_{0,j+k} +\hor_{1,j+k}\vert}{2}\in \left(\frac{5\pi}{12}, \frac{7\pi}{12} \right)\quad  \textup{ for all } k\in \{0,1,2\}. 
\end{equation}
 Otherwise, there is $k\in \{0, 1, 2\}$ such that the preceding assumption does not hold. Assume that $\vert \hor_{0, j+k} +\hor_{1, j+k}\vert \leq 5\pi/6$. Then, we can estimate the energy on $[0,2\eps] \times [j\eps, (j+2)\eps]$ by 
\begin{equation}\begin{split}
& \Eiso(u,[0,2\eps]\times [j\eps, (j+2)\eps]) \\
& \geq \frac{\eps^2}{2}\left( 4\cos^2\left(\frac{\hor_{0,j+k} +\hor_{1,j+k}}{2} \right)-4(1-\delta) \cos(\hor_{0,j+k}) - 4(1-\delta) \cos(\hor_{1,j+k}) +4(1-\delta)^2 \right) \\
& \geq \frac{\eps^2}{2}\left(4\cos^2\left(\frac{5\pi}{12}\right) -8(1-\delta) \right) \geq \frac{\eps^2}{2}\left( \frac14  - 8(1-\delta)\right)  \geq \frac{1}{16}\eps^2,  \end{split} 
\end{equation}
where we used $4\cos^2(5\pi/12) \geq 1/4$ and $1-\delta \leq 1/64$. If $\vert \hor_{0, j+k} +\hor_{1, j+k}\vert\geq 7\pi/6$, one argues similarly. Now, assume that \eqref{ass:two_angles} holds. Since $(\hor,\ver)$ is the angular velocity field of a spin field, we have for $k=0,1$
\begin{equation}
    \label{lower:spins:curl_identity}
\ver_{2, j+k} = \hor_{0, j+k+1} + \hor_{1, j+k+1} - \hor_{1, j+k} -\hor_{0, j+k} -\eps\curl((\hor,\ver)_{0,j+k})-\eps\curl((\hor, \ver)_{1, j+k}),
\end{equation}
 where we used $\ver_{0,\cdot} \equiv 0$. We claim that it holds 
\begin{equation}
\vert \ver_{2, j+k} \vert \leq \frac\pi3 \quad \textup{ for } k=0,1.\label{lower:spins:big_delta:vertical_bound}    
\end{equation}
Define $\theta^u_k \coloneqq \hor_{0, j+k+1} + \hor_{1, j+k+1}$ and $\theta^l_k \coloneqq \hor_{0, j+k} + \hor_{1, j+k} $.
First, assume that $\textup{sgn}(\theta^u_k) = \textup{sgn}(\theta^l_k)$ and estimate using \eqref{ass:two_angles}
\begin{align*}
|\ver_{2,j+k}| &\geq |\eps\curl((\hor_{0,j+k},\ver_{0,j+k})) + \eps\curl((\hor_{1, j+k}, \ver_{1,j+k}))|   - |\theta_k^u - \theta_k^l| \\ & \geq |\eps\curl((\hor_{0,j+k},\ver_{0,j+k})) + \eps\curl((\hor_{1, j+k}, \ver_{1,j+k}))| - \frac{\pi}6.
\end{align*}
Since $\ver_{2, j+k}\in [-\pi, \pi)$ and $\eps\curl((\hor_{0,j+k},\ver_{0,j+k})), \eps\curl((\hor_{1, j+k}, \ver_{1,j+k})) \in \{0, \pm 2\pi\}$ it follows that $\eps\curl((\hor_{0,j+k},\ver_{0,j+k})) + \eps\curl((\hor_{1, j+k}, \ver_{1,j+k})) = 0$.
In turn, this implies by \eqref{lower:spins:curl_identity} that it holds $|\ver_{2,j+k}| \leq |\theta_k^u - \theta_k^l| \leq \pi / 6 \leq \pi/3$. 
Now, assume $\textup{sgn}(\theta^u_k) \neq \textup{sgn}(\theta^l_k)$. 
Without loss of generality, we may assume $\textup{sgn}(\theta^u_k) = 1.$ Then, 
\begin{equation}
    \begin{split}
       \ver_{2, j+k} &= \theta^u_k -\theta^l_k -\eps\curl((\hor_{0,j+k},\ver_{0,j+k}))-\eps\curl((\hor_{1, j+k}, \ver_{1,j+k}))\\
       &\leq \frac{14}{6}\pi  -\eps\curl((\hor_{0,j+k},\ver_{0,j+k}))-\eps\curl((\hor_{1, j+k}, \ver_{1,j+k}))
    \end{split}
\end{equation}
and 
\begin{equation}
    \begin{split}
       \ver_{2, j+k} &= \theta^u_k -\theta^l_k -\eps\curl((\hor_{0,j+k},\ver_{0,j+k}))-\eps\curl((\hor_{1, j+k}, \ver_{1,j+k}))\\
       &\geq \frac{10}{6}\pi  -\eps\curl((\hor_{0,j+k},\ver_{0,j+k}))-\eps\curl((\hor_{1, j+k}, \ver_{1,j+k})).
    \end{split}
\end{equation}
As $\ver_{2, j+k}\in [-\pi, \pi)$ and $\eps\curl((\hor_{0,j+k},\ver_{0,j+k})), \eps\curl((\hor_{1, j+k}, \ver_{1,j+k})) \in \{0, \pm 2\pi\}$ it follows that $\eps\curl((\hor_{0,j+k},\ver_{0,j+k})) + \eps\curl((\hor_{1, j+k}, \ver_{1,j+k})) = 2\pi$.
Thus, it holds $|\ver_{2, j+k}| \leq \pi/3$.
Then we estimate using \eqref{lower:spins:big_delta:vertical_bound}
\begin{equation}\begin{split}
& \Eiso(u,[0,2\eps]\times [j\eps, (j+2)\eps]) \\
& \geq \frac{\eps^2}{2}\left( 4\cos^2\left(\frac{\ver_{1,j} +\ver_{1,j+1}}{2} \right)-4(1-\delta) \cos(\ver_{1,j}) - 4(1-\delta) \cos(\ver_{1,j+1}) +4(1-\delta)^2 \right) \\
& \geq \frac{\eps^2}{2} \left( 4\cos^2\left( \frac\pi3\right) - 8(1-\delta)\right)\geq \frac{\eps^2}{2}\left( 1  - 8(1-\delta)\right)  \geq \frac{1}{16} \eps^2,\end{split}   
\end{equation}
where we used $\cos(\pi/3) = 1/2.$ This shows \eqref{eq: lower bound large delta 1}. 
Eventually, we sum estimate \eqref{eq: lower bound large delta 1} to find
\begin{align*}
 \Eiso(u, [0,2\eps]\times [0,1]) &\geq \frac13 \sum_{j=0}^{\left\lfloor 1/\eps\right\rfloor-2}\Eiso(u, [0,2\eps]\times [j\eps,j\eps +2\eps]) \\
 &\geq \left( \left\lfloor \frac{1}{\eps}\right\rfloor - 1 \right) \frac{\eps^2}{48} \geq \frac{\eps}{192}.
\end{align*}
\end{proof}
We continue by discussing the case of $\delta$ being bounded away from one and away from zero.
\begin{lemma}\label{lem:lower_close_inbetween}
There is a constant $c > 0$ such that for all $(\eps,\delta)  \in (0, 1/2)\times (0, 63/64)$ and $u \in \Sbound$ it holds 
\[\Eiso(u, [0,2\eps]\times [0,1]) \geq c\eps\delta^{2}. \]
\end{lemma}
\begin{proof}
Let $\eta := 1/(32 \cdot 64^2)$ and assume that $\Eiso (u, \Omega) \leq \eta \eps$ (otherwise, there is nothing to show). 
We introduce the set of vortices along the left boundary
\begin{equation}\begin{split}
 V_0 \coloneqq &\left\{ (0,j) \in \Z^2 \colon j\eps\in (0, 1-\eps)\cap\eps\Z \textup{ and }  \left\vert  \eps\curl((\hor_{0,j},\ver_{0,j}))\right\vert = 2\pi \right\} \\
 \subseteq &\left\{ (0,j) \in \Z^2 \colon j\eps \in (0, 1-\eps)\cap\eps\Z \textup{ and } \max\{\vert\hor_{0, j}\vert, \vert \hor_{0,j+1}\vert, \vert \ver_{1,j}\vert \} \geq  \frac{2}{3}\pi\right\} =:B_0.\end{split} \end{equation}
Lemma \ref{lem:vortices-estimates} and Remark \ref{rem:beta_special} for $\beta = 2\pi/3$ yield the estimate 
\begin{equation}
 \mathcal{H}^0(V) \leq \mathcal{H}^0(B_0) \leq \frac{8\Eiso(u, \Omega)}{\eps^2(1-\delta)^2 }\leq \frac{8 \eta \eps}{\eps^2 (1-\delta)^2} \leq \frac{1}{4\eps}.\label{lower:big_delta:vortices}
\end{equation} 
Let $\mathcal{I}_0$ be the set of cells on the left boundary without vortices i.e., 
\[\mathcal{I}_0\coloneqq \{ (0,j) \in \Z^2 \colon j\eps \in (0,1) \textup{ and } j \notin B_0\}. \]  By \eqref{lower:big_delta:vortices} and $\eps\in (0, 1/2)$, we have 
\begin{equation}
    \mathcal{H}^0(\mathcal{I}_0) \geq \frac{3}{4\eps}-1\geq \frac{1}{4\eps}.
\end{equation}
Now, let $(0, j)\in \mathcal{I}_0$. Since $V_0 \subseteq B_0$, it holds
\begin{equation}
    0 = \hor_{0, j+1} - \ver_{1,j}-\hor_{0,j}=: \theta_1^j -\theta_2^j -\theta_3^j. 
\end{equation} 
We claim that there exists $k\in \{1,2,3\}$ such that
\begin{equation}
    \vert \theta_k^j\vert \notin \left\lbrack\frac34\opt, \frac54 \opt\right\rbrack. \label{lower:spins:smaller_delta:bound:angles}
\end{equation}
Assuming \eqref{lower:spins:smaller_delta:bound:angles} is true for every $j \in \mathcal{I}_0$, we estimate the energy as in \eqref{eq: estimate energy reformulation} in the proof of Lemma \ref{lem:vortices-estimates} 
\begin{equation}\begin{split}
  &\Eiso(u,\Omega) \\ \geq &\frac{1}{2 \cdot 64^2}\eps^2 \sum_{(0,j) \in \mathcal{I}_0} \left\lbrack (1-\delta - \cos(\hor_{0,j}))^2+ (1-\delta - \cos(\hor_{0,j+1}))^2+ (1-\delta - \cos(\ver_{1,j}))^2 \right\rbrack   \\ 
 \geq &\frac{1}{2\cdot 64^2} \eps^2 \mathcal{H}^0(\mathcal{I}_0) \min\left\{ \left\vert \cos(\opt) -\cos\left(\frac34\opt\right) \right\vert^2, \left\vert \cos(\opt) - \cos\left(\frac54\opt\right)\right\vert^2\right\}  \\
 \geq &\frac{\eps}{8\cdot 64^2} \min\left\{ \left\vert \int_{\frac34\opt}^{\opt} \sin(s) \textup{ d}s\right\vert^2, \left\vert \int_{\opt}^{\frac54\opt} \sin(s) \textup{ d}s\right\vert^2\right\} \\ \geq &\sin^2\left( \frac34 \opt \right)\frac{\eps\opt^2 }{8\cdot 16 \cdot 64^2}\\ 
 \geq &\frac14 \left(\frac34\right)^2 \opt^2 \frac{\eps\delta }{8\cdot 16 \cdot 64^2}\geq \frac{9}{4\cdot 8\cdot 16^2 \cdot 64^2}\eps \delta^2,\end{split}
\end{equation}
where we used $\delta^{1/2} \leq \opt= \arccos(1-\delta)$. This yields the lower bound with $c\coloneqq 9/ (4\cdot 8\cdot 16^2 \cdot 64^2) $. \\
Hence, it is left to prove that there is $k\in \{1, \dots , 3\}$ such that \eqref{lower:spins:smaller_delta:bound:angles} holds.
For the sake of a contradiction, assume $\vert\theta^j_k\vert\in [3 \opt/4, 5\opt/4]$ for all $k\in \{1, \dots , 3\}$. 
Without loss of generality, we may assume $ \theta_1^j \geq 0$. 
Since $(\hor,\ver)$ is curl-free, we have 
$\vert \theta_1^j \vert = \theta_1^j = \theta_2^j + \theta_3^j$. 
If $\textup{sgn}(\theta_2^j) = \textup{sgn}(\theta_3^j)$, we estimate 
\[ \frac{5}4 \opt \geq \theta_1^j \geq \frac64 \opt > \frac{5}{4}\opt,
\] which is a contradiction.
Next, assume $\textup{sgn}(\theta_2^j) \neq \textup{sgn}(\theta_3^j).$ 
Then
\[
\frac34 \opt \leq \theta_1^j=\theta_2^j + \theta_3^j \leq -\frac34 \opt + \frac54 \opt = \frac12 \opt < \frac34 \opt,
\] 
which is again a contradiction.
\end{proof}

\subsubsection{Lower bound for small $\delta \in (0, \overline{\delta})$} In this section we prove the lower bound for small interaction parameters $\delta \in (0, \overline{\delta})$. In this setting Lemma \ref{lem:lower_bound_q} allows us to exclude the presence of vortices. The remainder of the proof is inspired by the strategy from \cite[Lemma 4]{GiZw22} for a related continuum model, see also \cite{GiZw23_Martensites,hex}.
 \begin{lemma}
 \label{lem:lower_bound_no_vortices} Let $\eta:= \frac{1}{2\cdot 4\cdot 12^2 \cdot 256}$. Then there is a constant $c_l> 0$ such that for all $\delta \in (0, \eta^2/16)$, $\eps\in (0, \eta/4)$ with $\delta \leq \eps$ and $u \in \Sbound$ it holds
 \[\Eiso(u,\Omega) \geq c_l \min\left\{ \delta^2, \eps\delta^{3/2} \left(\left\vert \ln \left( \frac{\eps}{\delta^{1/2}}\right) \right\vert +1  \right) \right\}.  \] 
 \end{lemma}
 \begin{remark}\label{rem:scaling_continous}
 Note, that 
 \[ \min\left\{ \delta^2, \eps\delta^{3/2} \left(\left\vert \ln \left( \frac{\eps}{\delta^{1/2}}\right) \right\vert +1  \right) \right\}= \begin{cases}
 \delta^2 \quad &\textup{ if } \delta^{1/2} \leq \eps,\\
 \eps\delta^{3/2} \left(\left\vert \ln \left( \frac{\eps}{\delta^{1/2}}\right) \right\vert +1  \right)  \quad & \textup{ if } \eps\leq\delta^{1/2}.
 \end{cases}\] 
 \end{remark}
 \begin{proof}
Let $u \in \Sbound$, $c_1 \coloneqq 1/(32\cdot 12^2 \cdot 256)$ and $c_2 \coloneqq \eta$. \\

 \quad \textbf{Step 1: Preparation.} We may assume that $u$ satisfies 
 \begin{equation}
     \Eiso(u, \Omega) \leq c_1 \min\left\{ \delta^2, \eps\delta^{3/2} \left(\left\vert \ln \left( \frac{\eps}{\delta^{1/2}}\right) \right\vert +1  \right) \right\}.\label{lower:no_vortices:energy_ass}
 \end{equation}
Otherwise there is nothing to show.  
 Next, using \eqref{spin:gamma:reformulation_vertical} and Lemma \ref{lem:lower_bound_q} yields 
 \begin{equation}\begin{split}
      \Eiso(u, \Omega) \geq & \eps^2\sum_{(i,j) \in \INNver} (1-\delta -\cos(\ver_{i,j}))^2   +  2\eps^2 \sum_{(i,j) \in \INNver} p(\ver_{i,j+1}, \ver_{i,j})\left\vert\sin\left( \frac{\ver_{i+1,j}}{2}\right)-\sin\left( \frac{\ver_{i,j}}{2}\right) \right\vert^2 \\
     \geq  & \eps^2\sum_{(i,j) \in \INNver} (1-\delta -\cos(\ver_{i,j}))^2 +\eps^2 \sum_{(i,j) \in \INNver} \left\vert\sin\left( \frac{\ver_{i,j+1}}{2}\right)-\sin\left( \frac{\ver_{i,j}}{2}\right) \right\vert^2  \\ &- 2\pi^2(1-2c_p)\eps^2 \mathcal{H}^0(A^{ver}), \end{split}
    \label{lower:small_spin_refomrulation_argument}
 \end{equation}
 where $A^{ver}= \{(i,j) \in \INNver\colon \max\{\vert\ver_{i, j}\vert,\vert\ver_{i, j+1}\vert\} \geq \pi/8\}$ and $(\hor, \ver) \colon \Omega \cap\eps\Z^2 \rightarrow [-\pi, \pi)^2$ is the associated angular velocity field of $u$. Analogously, the same is true for the horizontal energy contribution, replacing $A^{ver}$ by the corresponding set $A^{hor}.$\\
 
\quad \textbf{Step 2: Reduction to estimate on slices.}  
We claim that for all $i \in \N$ satisfying
\[
c_2 \geq i \eps \geq \frac{c_2}4 \min\left\{ 1, \frac{\eps}{\delta^{1/2}} \left( \left\vert \ln \left( \frac{\eps}{\delta^{1/2}}\right) \right\vert + 1\right) \right\}=: a(\eps,\delta)
\]
one of the following holds 
\begin{equation}\begin{split}
 \textup{(i)} & \  \eps \sum_{j\in \mathcal{J}_{N}} (1-\delta -\cos(\ver_{i,j}))^2 \geq c_2 \delta^2 \nonumber \\
 \textup{(ii)} & \ \eps \sum_{j\in \mathcal{J}_{N}} \max\{ \vert 1-\delta -\cos(\ver_{i,j})\vert, \vert 1-\delta -\cos(\ver_{i,j+1})\vert \} \left\vert\sin\left( \frac{\ver_{i,j+1}}{2}\right)-\sin\left( \frac{\ver_{i,j}}{2}\right) \right\vert  \geq \frac{c_2}{i\eps}\eps\delta^{3/2},\end{split}
 \end{equation}
 where $\mathcal{J}_N \coloneqq [0, 1) \cap \eps \Z$.
We will show first how to conclude from the claim. The rest of the proof will then be devoted to proving the claim.
 
By Young's inequality, (i) and (ii) from above, and \eqref{lower:small_spin_refomrulation_argument}, we have 
\begin{equation}
\Eiso(u,\Omega) \geq \eps \sum_{i\eps \geq a(\eps,\delta)} \min\left\{c_2\delta^2, \frac{c_2}{i\eps} \eps\delta^{3/2} \right\}- 2\pi^2(1-2c_p)\eps^2 \mathcal{H}^0(A^{ver}).\label{lower:ifnot}
\end{equation}
By Lemma \ref{lem:vortices-estimates} we estimate 
\[ 
\mathcal{H}^0(A^{ver}) \leq \frac{2\Eiso(u, \Omega)}{(1-\delta)^2 \eps^2 (1-\delta- \cos(\frac\pi8))^2} \leq \frac{4\cdot20^2c_1 s(\eps,\delta)}{\eps^2 }\leq  4\cdot20^2 c_1 \left( \frac{\delta}{\eps}\right)^2\leq 4 \cdot 20^2 c_1 \leq \frac12, 
\]
where we used $1-\delta - \cos(\pi/8) = \cos(\opt) -\cos(\pi/8)\geq \cos(\pi/16) - \cos(\pi/8) \geq 1/20$,  and $\delta \leq \eps$. Consequently, $\mathcal{H}^0(A^{ver}) = 0$. Similarly, one obtains $\mathcal{H}^0(A^{hor})=0$. Now, we discuss the cases $\delta^{1/2}\leq \eps$ and $\delta \leq \eps\leq \delta^{1/2}$ separately. \\

\quad \emph{\textbf{Step 2.1: $\delta^{1/2} \leq \eps$.}} In this case, inequality \eqref{lower:ifnot} reduces to 
\begin{equation}
\Eiso(u,\Omega) \geq \eps \sum_{c_2 \geq i\eps \geq a(\eps,\delta)} \min\left\{c_2\delta^2, \frac{c_2}{i\eps} \eps\delta^{3/2} \right\}\geq \frac{c_2}{2} \delta^2\left( c_2  - \frac{c_2}{4}- 2 \eps\right)\geq \frac{c_2^2}{8} \delta^2,
\end{equation}
where we used $a(\eps,\delta) = c_2/4$, $c_2\leq 1$, and $\eps\in (0, c_2/4).$\\

\quad \emph{\textbf{Step 2.2: $\delta \leq \eps\leq \delta^{1/2}$.}} By Remark \ref{rem:scaling_continous} and the assumption $\delta^{1/2} \leq \eta/4 = c_2 /4$, we have 
\begin{equation}
    \label{lower:spins:ieps:ifnot}  a(\eps,\delta) = \frac{c_2}{4}s(\eps,\delta)= \frac{c_2}{4} \frac{\eps}{\delta^{1/2}} \left( \ln\left( \frac{\delta^{1/2}}{\eps}\right) +1\right)\geq \frac{c_2}{4}\frac{\eps}{\delta^{1/2}} \geq \eps.
\end{equation}
It follows that
\begin{align*}
\Eiso(u,\Omega) \geq & \frac{\eps}2 \sum_{c_2 \geq  i\eps \geq a(\eps,\delta)} c_2  \min\left\{\delta^2, \frac{\eps \delta^{3/2}}{i\eps} \right\} \\ 
\geq & \frac{\eps}2 \sum_{c_2 \geq  i\eps \geq a(\eps,\delta)} c_2 \frac{\eps \delta^{3/2}}{i\eps} \\
\geq & c_2  \frac{\eps \delta^{3/2}}4 \int_{a(\eps,\delta) + \eps}^{c_2 } \frac1t \, dt \\
= & c_2  \frac{\eps \delta^{3/2}}4 \left( \ln(c_2) - \ln( a(\eps,\delta) + \eps) \right) \\
\geq &c_2  \frac{\eps \delta^{3/2}}4 \left( \ln(c_2) - \ln( 2 a(\eps,\delta) ) \right) \\
= & c_2  \frac{\eps \delta^{3/2}}4 \left( \ln(2) - \ln\left( \frac{\eps}{\delta^{1/2}} ( \ln ( \delta^{1/2} / \eps) + 1 ) \right) \right) \\
 = & \frac{c_2}{4} \eps\delta^{3/2} \left( \ln\left( 2\right) - \ln\left( \frac{\eps}{\delta^{1/2}}\left( \ln\left( \frac{\delta^{1/2}}{\eps}\right) +1\right)   \right)\right)\geq \frac{c_2}{4\cdot 50} \eps\delta^{3/2} \left(\left\vert \ln\left(\frac{\eps}{\delta^{1/2} }\right) \right\vert +1 \right) ,
\end{align*}
where we used $\ln(\vert\ln\sigma\vert+1) \leq \max\{ 24 \ln(2)/25, 3\vert\ln\sigma\vert/4\}$ and $\ln(2)\geq 1/2$ in the last estimate.\\

\quad \emph{\textbf{Step 3: Proof of the claim.}} Assume that (i) and (ii) do not hold for $i\eps$. Define the sets
\begin{equation}\begin{split}
\mathcal{M}_+ & \coloneqq  \left\{j\eps \in (0,1)\cap\eps\Z \colon \ver_{i, j} \in \left(\frac12 \opt, \frac32 \opt\right) \right\}  \text{\qquad and}\\
\mathcal{M}_- & \coloneqq   \left\{j\eps \in (0,1)\cap\eps\Z \colon \ver_{i,j} \in \left(-\frac32 \opt, -\frac12 \opt\right) \right\} .\end{split}
\end{equation}
Note that for $0 \leq \theta < \pi/3$ it holds
\begin{equation}\label{eq: est difference arccos}
|\arccos(1-\delta)^2 - \theta^2| = \left| \int_{\arccos(1-\delta)}^{\theta} \frac12 s \, ds \right| \leq \left| \int_{\arccos(1-\delta)}^{\theta} \sin(s) \, ds \right| = |1 -\delta - \cos(\theta) |.
\end{equation}
Similarly, the same estimate holds for $-\pi/3 < \theta \leq 0$.
Together with $\mathcal{H}^0(A^{ver}) = 0$ and the assumption that (i) does not hold it follows for the set $\mathcal{M}^c \coloneqq ((0,1)\cap\eps\Z)\setminus (\mathcal{M}_+ \cup \mathcal{M}_-)$ that
\begin{equation}\begin{split}
 \eps\mathcal{H}^0(\mathcal{M}^c) = &\eps \sum_{j\eps\in \mathcal{M}_c}1 \leq  \frac{16}{ 9\opt^4 } \eps \sum_{j\eps \in \mathcal{M}^c} (\opt^2-(\ver_{i,j})^2 )^2   \\
&  \leq \frac{4\cdot 16}{9 \opt^4 } \eps \sum_{j\eps \in \mathcal{M}^c} \left( 1-\delta -\cos(\ver_{i,j})\right)^2 \leq \frac{4\cdot 16 c_2\delta^2}{ 9\opt^4 }\leq \frac{64}{9} c_2\leq \frac13. \end{split}
\end{equation}
In particular, since $\eps \in (0,1/2)$
\begin{equation}
   \eps \mathcal{H}^0(\mathcal{M}_+ \cup \mathcal{M}_-)\geq \frac23-\eps \geq \frac16.
\end{equation}
Without loss of generality, we may assume $\eps \mathcal{H}^0(\mathcal{M}_+) \geq 1/12.$   
We introduce the piecewise affine  interpolation of $\ver_{i,\cdot}$ on $(0,1)$ by 
(see Figure \ref{fig:convex_versus_discrete})
\[ 
\theta_i(y) \coloneqq \frac{y-j\eps}{\eps} \ver_{i, j+1}+ \frac{(j+1)\eps -y}{\eps}\ver_{i,j}\quad \textup{ for } y\in [j\eps, (j+1)\eps).
\] 
Note, that $\theta_i(\cdot)$ is Lipschitz-continuous on $(0,1)$ and $\vert \theta_i\vert \leq \pi/3$ since $\mathcal{H}^0(A^{ver})=0$. 

\begin{figure}[h]
    \centering
    \includegraphics[scale = 0.35]{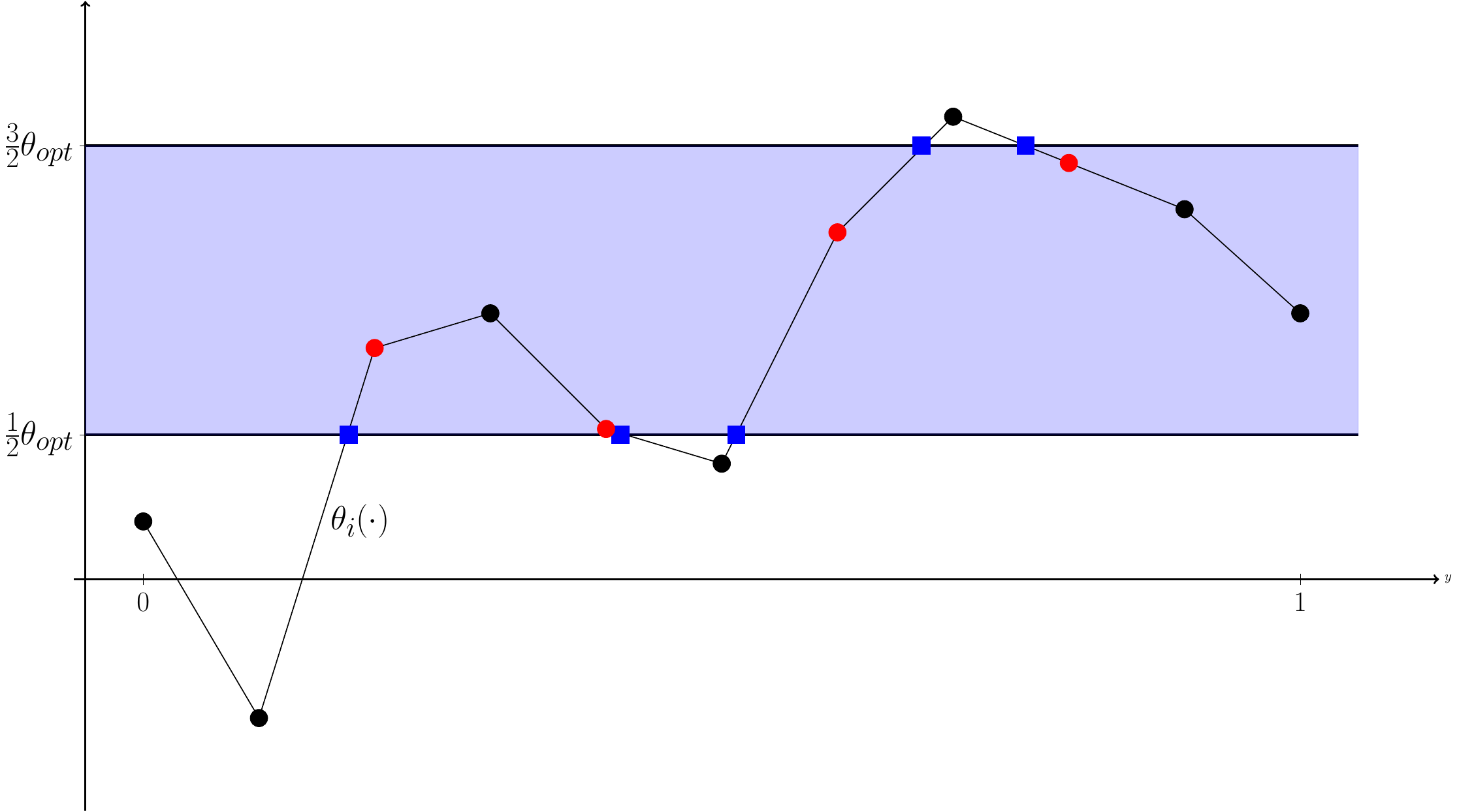}
    \caption{Sketch of the function $\theta_i$. The black dots are the values of $\ver_{i,\cdot}$, the red dots mark the discrete boundary points of $\mathcal{M}_+$, and the blue rectangles indicate where the interpolated function $\theta_i$ leaves the interval $(\opt/2,3\opt/2)$.}
    \label{fig:convex_versus_discrete}
\end{figure}
Furthermore, we introduce the function \[ 
F(\theta) \coloneqq \int_{-\frac{\pi}{3}}^\theta \vert  \opt^2-\gamma^2 \vert \textup{ d}\gamma  \quad \textup{ for }\theta \in [-\frac\pi3, \frac\pi3]. 
\] 
 By the Lipschitz continuity of $F\circ \theta_i$, the coarea formula (see \cite[Theorem 1 in Section 3.4.2]{EvansLawrenceC1992Mtaf}), estimate \eqref{eq: est difference arccos},  $ \vert s-t\vert/2\leq \vert\sin(s)-\sin(t)\vert $ for all $|s|, | t|\leq \pi/3$ and the fact that we assume that (ii) does not hold,  we estimate
\begin{equation}\begin{split}
& \int_\R \mathcal{H}^0\left(\partial \left\{ F\circ \theta_i >s  \right\} \right)  + \mathcal{H}^0(\partial \{F\circ \theta_i  <s   \} ) \textup{ d}s  \\
\leq &2 \int_0^1 \vert (F\circ \theta_i)^\prime(y) \vert \, dy \\
\leq &2\sum_j \int_{j\eps}^{(j+1)\eps} \vert \opt^2 -\theta_i^2(y)\vert \left\vert \frac{\ver_{i,j+1}-\ver_{i,j}}{\eps}\right\vert \textup{ d}y\\
\leq &2\sum_j \int_{j\eps}^{(j+1)\eps} \max\{ \vert \opt^2 - (\ver_{i,j})^2 \vert, \vert \opt^2 - (\ver_{i,j+1})^2 \vert\}  \left\vert \frac{\ver_{i,j+1}-\ver_{i,j}}{\eps}\right\vert \textup{ d}y\\ 
 \leq& 2\eps \sum_j \max\{\vert 1-\delta -\cos(\ver_{i,j+1})\vert, \vert 1-\delta -\cos(\ver_{i,j})\vert \}\left\vert \frac{\ver_{i,j+1}-\ver_{i,j}}{\eps}\right\vert \textup{ d}y\  \\
\leq &8 \sum_j \max\{\vert 1-\delta -\cos(\ver_{i,j+1})\vert, \vert 1-\delta -\cos(\ver_{i,j})\vert \}\left\vert \sin\left(\frac{\ver_{i,j+1}}{2}\right)-\sin\left(\frac{\ver_{i,j}}{2}\right)\right\vert \textup{ d}y \\ 
 \leq &\frac{8c_2\delta^{3/2}}{i\eps}.\end{split} \label{lower:spins:continuous:boundary_estimate}
\end{equation} 
  
Next, by Fubini, there exists $\hat{\theta}\in (0,\opt/4)$ such that 
\begin{align}
& \mathcal{H}^0\left(\partial\left\{ \theta_i >\frac{\opt}{2} - \hat{\theta}\right\}\right) + \mathcal{H}^0\left(\partial \left\{ \theta_i < \frac{3\opt}{2}+\hat{\theta}\right\}\right) \\
 & \quad \leq  \mathcal{H}^0\left(\partial\left\{ F\circ \theta_i > F\left( \frac{\opt}2 - \hat{\theta} \right)\right\}\right) + \mathcal{H}^0\left(\partial\left\{   F\circ \theta_i < F\left(\frac{3\opt}{2}+\hat{\theta}\right)\right\}\right) \label{eq: estimate superlevel}\\
 & \quad \leq \frac{8}{\opt} \int_0^{\frac14\opt}\mathcal{H}^0\left(\partial \left\{ F\circ \theta_i >F\left(\frac{\opt}{2}-s\right)  \right\} \right) \mathcal{H}^0\left(\partial \left\{F\circ \theta_i  <F\left( \frac{3\opt}{2}+s\right)  \right\} \right) \textup{ d}s.\end{align}
 Consider the functions
 \[g_1(s) \coloneqq F\left( \frac{\opt}{2}-s\right) \quad \textup{ and } \quad g_2(s)\coloneqq F\left(\frac{3\opt}{2}+s\right).\quad \] 
 Then, $\min\{  |g_1^\prime (s)| , |g_2^\prime (s)| \}\geq 3\opt^2/4 $ for all $s\in (0, \opt/4).$
Combining \eqref{lower:spins:continuous:boundary_estimate}, \eqref{eq: estimate superlevel} and the estimate $ \delta^{1/2} \leq\arccos(1-\delta) = \opt$ for $\delta \in (0,1)$, we obtain 
\begin{equation}\begin{split}
&\mathcal{H}^0\left(\partial\left\{ \theta_i >\frac{\opt}{2} - \hat{\theta}\right\}\right) + \mathcal{H}^0\left(\partial \left\{ \theta_i < \frac{3\opt}{2}+\hat{\theta}\right\}\right)\\
\leq & \frac{8}{\opt} \int_0^{\frac14\opt}\mathcal{H}^0\left(\partial \left\{ F\circ \theta_i > g_1(s)  \right\} \right)\textup{ d}s + \frac{8}{\opt} \int_0^{\frac14\opt} \mathcal{H}^0\left(\partial \left\{F\circ \theta_i  < g_2(s)  \right\} \right) \textup{ d}t\\
 \leq &\frac{8}{\opt} \left(\int_{F\left(\frac14\opt\right)}^{F\left(\frac12\opt\right)}\frac{4\mathcal{H}^0\left(\partial \left\{ F\circ \theta_i >s  \right\} \right)}{3\opt^2}\textup{ d}s  + \int_{F\left(\frac32\opt\right)}^{F\left(\frac74\opt\right)}\frac{4\mathcal{H}^0(\partial \{F\circ \theta_i  <s   \} )}{3\opt^2} \textup{ d}s\right) \\
\leq &\frac{32}{3\opt^3 }  \int_\R \mathcal{H}^0\left(\partial \left\{ F\circ \theta_i >s  \right\} \right)  + \mathcal{H}^0(\partial \{F\circ \theta_i  <s   \}  \textup{ d}s\leq \frac{256c_2 }{i\eps}.\end{split}\label{lower:number_boundary}
\end{equation}
Hence, the set 
\begin{equation}
  M_+^{\hat{\theta}}\coloneqq   \left\{y\in (0,1) \colon \frac{\opt}{2}-\hat{\theta}\leq \theta_i(y) \leq \frac{3\opt}{2}+\hat{\theta}\right\} 
\end{equation}
has at most  $\mathcal{K}(i\eps) \leq 256 c_2/(i\eps)$ boundary points in $(0,1)$. We want to show that the number of discrete boundary points $\partial^{d, \eps} \mathcal{M}_+^{\hat{\theta}}$ of the set
\begin{equation}
   \mathcal{M}_+^{\hat{\theta}}\coloneqq    M_+^{\hat{\theta}} \cap \eps \Z =  \left\{j\eps\in (0,1)\cap \eps\Z \colon \frac{\opt}{2}-\hat{\theta}\leq \ver_{i,j} \leq \frac{3\opt}{2}+\hat{\theta}\right\}    
\end{equation}
is bounded by $\mathcal{K}(i\eps)$. If $\mathcal{H}^0(\partial M_+^{\hat{\theta}}) = 0$ we either have $\theta_i \in (\opt/2 - \hat{\theta}, \opt/2 + \hat{\theta})$ or $\theta_i \notin (\opt/2 - \hat{\theta}, \opt/2 + \hat{\theta})$ which implies the same property for $\ver_{i, \cdot}$ in $(0,1) \cap \eps\Z.$ Thus, we have $\mathcal{H}^0(\partial^{d, \eps} \mathcal{M}_+^{\hat{\theta}}) = 0$. If $\mathcal{H}^0(\partial M_+^{\hat{\theta}}) = 1$, there is $t \in (0,1)$ such that either $(0, t) = M_+^{\hat{\theta}}\setminus \{ t\} $ or $(t, 1) = M_+^{\hat{\theta}}\setminus\{t\} $. Hence, there is $l\eps \in (0,1)\cap\eps\Z$ such that either $(0, l\eps) \cap\eps\Z = M_+^{\hat{\theta}}\setminus \{l \eps\} $ or $(l\eps, 1)\cap\eps\Z = M_+^{\hat{\theta}}\setminus\{l\eps\}$, which yields $\mathcal{H}^0(\partial M_+^{\hat{\theta}}) = 1$. Finally, let us turn to the case $M\coloneqq \mathcal{H}^0(\partial M_+^{\hat{\theta}}) \geq  2$. Let $0 \eqqcolon y_0 < y_1 < \dots < y_M <  y_{M+1} \coloneqq 1$ be the boundary points of $M_+^{\hat{\theta}}$ and $l = 0, \dots, M-1$. By the definition of $M_+^{\hat{\theta}}$ we either have $(y_l, y_{l+1}) \subseteq M_+^{\hat{\theta}}\setminus\{ y_l, y_{l+1}\} $ and $(y_{l+1}, y_{l+2}) \subseteq (M_+^{\hat{\theta}})^c\setminus\{ y_{l+1}, y_{l+2}\} $ or vice versa. Therefore, $(y_l, y_{l+1})\cap\eps\Z \subseteq \mathcal{M}_+^{\hat{\theta}}\setminus\{ y_l, y_{l+1}\} $ and $(y_{l+1}, y_{l+2})\cap\eps\Z \subseteq (\mathcal{M}_+^{\hat{\theta}})^c\setminus\{ y_{l+1}, y_{l+2}\} $ holds. If $(y_l, y_{l+1}] \cap \mathcal{M}_+^{\hat{\theta}}= \emptyset$ we have $\partial^{d, \eps} \mathcal{M}_+^{\hat{\theta}} \cap (y_l, y_{l+2}) = \emptyset.$ Otherwise, we have $\mathcal{K}(i\eps) \coloneqq \mathcal{H}^0(\partial^{d, \eps} \mathcal{M}_+^{\hat{\theta}} \cap (y_l, y_{l+2})) = 1$, and thus obtain 
\begin{equation}
    \mathcal{H}^0(\partial^{d, \eps} \mathcal{M}_+^{\hat{\theta}})\leq \mathcal{H}^0(\partial M_+^{\hat{\theta}})\leq \mathcal{K}(i\eps) \leq \frac{256 c_2}{i\eps}.
\end{equation}
Hence, there exist pairwise disjoint, non-empty discrete intervals  of the form $I_l \coloneqq [k_l \eps, m_l\eps] \cap \eps \Z$,  for $l = 0, \dots, \mathcal{K}(i\eps)$ and $k_l, m_l \in \N$, satisfying 
\[ \bigcup_{l=1}^{\mathcal{K}(i\eps)} I_l \subseteq \{ j\eps \in (0,1) \cap \eps \Z: \frac{\opt}2 - \hat{\theta} \leq  \ver_{i,j} \leq \frac{3 \opt}2 + \hat{\theta} \} \quad   \textup{ and }\quad \eps \mathcal{H}^0\left( \bigcup_{l=1}^{\mathcal{K}(i\eps)} I_l \right) \geq \frac{1}{12}.
\]

Further, we introduce a discrete potential $u \colon \Omega \cap\eps\Z^2 \rightarrow \R$ of the angular velocity field $(\hor,\ver)$ via 
\begin{equation}
    u(i\eps,j\eps)\coloneqq \frac{\eps}{\opt} \sum_{k=0}^{i-1}\hor_{k,j} \quad \textup{ for } (i\eps, j\eps) \in \lbrack \eps, 1)\cap\eps\Z^2\nonumber 
\end{equation}
and $u(0, \cdot) = 0.$
Since $\mathcal{H}^0(V) \leq \mathcal{H}^0(A^{ver})+ \mathcal{H}^0(A^{hor}) =0$ for the set $V$ of vortices of $(\hor,\ver)$, we have
\begin{equation}
\partial_1^{d,\eps}u(i\eps,j\eps) = \frac{1}{\opt} \hor_{i,j} \quad \textup{ and }\quad  \partial_2^{d,\eps} u(i\eps,j\eps) = \frac{1}{\opt} \ver_{i,j}. 
\end{equation}
Furthermore, we have $\partial_2^{d,\eps} u(i\eps, \cdot) \geq 1/4$ for all $j\eps \in  I_l\cap\eps\Z$ and $l=1, \dots, \mathcal{K}(i\eps).$ 
Consider the piecewise affine interpolation $\tilde{u} \colon [0,1) \rightarrow \R$ given by 
\[ \tilde{u}(y) \coloneqq \frac{y-j\eps}{\eps}u(i\eps, (j+1)\eps) + \frac{(j+1)\eps -y}{\eps}u(i\eps,j\eps) \quad \textup{ for } y\in [j\eps, (j+1)\eps)\] for $j\eps \in [0,1-\eps) \cap\eps\Z.$  The function $\tilde{u}$ satisfies $\partial_2 \tilde{u} = \partial_2^{d,\eps} u\geq 1/4$ for all $y\in \cup_{l=1}^{\mathcal{K}(i\eps)} I_l.$ Using H\"older's inequality  we find for all $l= 1, \dots, \mathcal{K}(i\eps)$
\begin{equation}
\frac{1}{12\cdot 16} \vert I_l\vert^3  \leq \int_{I_l} \vert \tilde{u}(y) \vert^2 \textup{ d}y \leq 2\eps \sum_{j\eps\in I_l}\left(  \eps \sum_{k=0}^{i-1} \vert \partial_1^{d,\eps}u(i\eps,j\eps) \vert \right)^2\leq 2\eps^2 \sum_{j\eps \in I_l} \sum_{k=0}^{i-1} \left\vert \frac{\hor_{i,j}}{\opt}\right\vert^2\cdot i\eps.
\end{equation}
Summing over $l$ leads to 
\begin{equation}\begin{split}
\frac{(i\eps)^2}{(256c_2)^2 16\cdot 12^4}   \leq &\frac{1}{(\mathcal{K}(i\eps))^2 16\cdot 12\cdot 12^3} \\ = &\frac{\mathcal{K}(i\eps)}{16\cdot 12} \left( \frac{1}{\mathcal{K}(i\eps)} \sum_{l=1}^{\mathcal{K}(i\eps)} \vert I_l\vert\right)^3  \\ \leq &\frac{1}{16\cdot 12} \sum_{l=1}^{\mathcal{K}(i\eps)} \vert I_l\vert^3 
\\  \leq &\sum_{l=1}^{K} \int_{I_l} \vert \tilde{u}(y)\vert^2 \textup{ d}y \\
\leq &2\eps \sum_{l=1}^K\sum_{j\eps \in I_l}\left( \eps \sum_{k=0}^{i-1}\left\vert \frac{\hor_{k,j}}{\opt}\right\vert^2\right)  i\eps\\    \leq &\frac{2\eps^2}{\opt^2}\left( \sum_{j\in \mathcal{J}_{NN}} \sum_{k=0}^{i-1} \vert \opt^2 - (\hor_{k,j})^2\vert \cdot i\eps\right)  + 2 (i\eps )^2.\end{split}\end{equation}
Using \eqref{eq: est difference arccos} we find
    \begin{align}
\frac{(i\eps)^2}{(256c_2)^2 16\cdot 12^4}  &\leq \frac{2\eps^2}{\opt^2}\left( \sum_{j\in \mathcal{J}_{NN}}\sum_{k=0}^{i-1} \vert \opt^2 - (\hor_{k,j})^2\vert \cdot i\eps\right)  + 2 (i\eps )^2\\
&\leq  \frac{2\eps^2}{\opt^2} \sum_{j\in \mathcal{J}_{NN}} \sum_{k\leq i} \vert 1-\delta - \cos(\hor_{k,j})\vert \cdot i\eps   + 2 (i\eps )^2 \\ &\leq 2 \left( \frac{\Eiso(u, \Omega)}{\opt^4}\right)^{1/2} (i\eps)^{3/2} + 2(i\eps)^2\\
& \leq  2 \left( \frac{\Eiso(u, \Omega)}{\delta^2}\right)^{1/2} (i\eps)^{3/2} + 2(i\eps)^2,\end{align}
where we used $\delta^{1/2}\leq \arccos(1-\delta)$ for $\delta \in (0,1)$ in the last line.
Since $(256c_2)^2 16\cdot 12^4 = 1/4$ we obtain  
\begin{equation}
    \frac{(i\eps)^2}{2(256c_2)^2 16\cdot 12^4} =  \frac{(i\eps)^2}{(256 c_2)^2 16\cdot 12^4}-2(i\eps)^2\leq 2 \left( \frac{\Eiso(u, \Omega)}{\delta^2}\right)^{1/2} (i\eps)^{3/2}.\label{lemma:lower:spins_almost_last_estimate}
\end{equation}
Rearranging  \eqref{lemma:lower:spins_almost_last_estimate}, using $\Eiso(u, \Omega) \leq c_1 s(\eps,\delta)$ and $4(4\cdot 256^2\cdot 16 \cdot 12^4)^2 c_2^3 c_1  = 1$ yields that $i\eps \leq a(\eps,\delta) $. Hence, (i) or (ii) is satisfied for $i\eps \geq a(\eps,\delta)$.
 \end{proof}
 \begin{remark}\label{rem:lowerbound2}
 Careful inspection of the proof shows that the lower bound from Lemma \ref{lem:lower_bound_no_vortices} holds whenever the angular velocity field $(\hor, \ver)$ of a minimizer $u \in \Sbound$ is curl-free.  
 Recalling that the number of vortices of $u$  is bounded by
 \begin{equation}
 \mathcal{H}^0(V) \leq 64\frac{\Eiso(u, \Omega)}{\eps^2}=64\frac{\min\Eiso}{\eps^2}.
 \label{rem:lower:spins:curl_freeness}  
 \end{equation}
 Hence, if the right-hand side of \eqref{rem:lower:spins:curl_freeness} is smaller than $1$ then $(\hor, \ver)$ is curl-free. Since $\frac{\min\Eiso}{\eps^2}\lesssim \min\left\{ \left(\frac{\delta}{\eps}\right)^2, \frac{\delta^{3/2}}{\eps}\left(\left\vert \ln\left(\frac{\eps}{\delta^{1/2}}\right) \right\vert + 1\right)\right\}$ this proves that a respective lower bound holds for a larger range of parameters than stated in Lemma \ref{lem:lower_bound_no_vortices}. 
 \end{remark}
 \begin{remark}[Existence of vortices] If a minimizer $u\in \Sbound$ is curl-free, we can use the strategy of Lemma \ref{lem:lower_bound_no_vortices} to obtain the same lower bound
 \[ \min\left\{ \delta^2 , \eps\delta^{3/2} \left( \left\vert \ln \left( \frac{\eps}{\delta^{1/2}}\right) \right\vert +1\right) \right\}\lesssim \Eiso(u, \Omega) \lesssim s(\eps, \delta) \leq \eps\delta^{1/2}.\] Fixing $\delta \in (0,1)$ it holds for $\eps \ll \delta$ that $\eps\delta^{1/2} \ll \min\{ \delta^2, \eps \delta^{3/2} \left( |\ln \eps/\delta^{1/2}| + 1 \right)\}$. Hence, there must be minimizers that use vortices to lower their energy.
 \end{remark}

\subsection{Scaling law for the periodic problem} 
In this section, we discuss the periodic problem
\begin{equation}
    \min\{ \Eiso(u, \Omega) \colon u \in \Sboundper\},\label{def:periodic_problem_discrete}
\end{equation}
where $\Sboundper$ consists of all spin fields  $u\in \Sbound$ that additionally satisfy the periodicity condition \eqref{gamma:spins:periodic_boundary}.
Note that 
\[ \argmin_{\Sboundper}(\Eiso(\cdot, \Omega))= \argmin_{\Sboundper}I_{\alpha,\eps}.\]
We prove the following partial scaling law. 
	\begin{theorem}[Scaling Law of the discrete model with periodic boundary conditions]\label{thm:scaling_per} Consider the function $s\colon (0,1/2)\times(0,1) \rightarrow (0,\infty)$ given by
		\[ s(\eps,\delta) \coloneqq \min\left\{ \delta^2, \eps\delta^{3/2} \left( \left\vert \ln\left( \frac{\eps}{\delta^{1/2}}\right) \right\vert +1\right), \eps\delta^{1/2}\right\}. \] 
	Then following statements are true for $\delta := \dver = \dhor$:
	\begin{itemize}
	    \item[(i)] There is a constant $C > 0$ such that for all $(\eps, \delta) \in (0,1/2) \times (0,1)$ there holds
	    \[ \min \left\{ \Eiso(u; \Omega) \colon u \in \Sboundper \right\} \leq C s(\eps,\delta). \]
	    \item[(ii)] There is $\overline{\delta}\in (0,1)$ and a constant $c > 0$, only depending on $\overline{\delta},$ such that for all\\ $(\eps, \delta) \in (0,1/2) \times [\overline{\delta}, 1)$ it holds
	\begin{equation*}
	   c \eps\delta^{1/2} \leq \min \left\{ \Eiso(u; \Omega) \colon u \in \Sboundper \right\}
	\end{equation*}
	and for all $\delta\in (0, \overline{\delta})$ and $\eps \in [\delta, 1/2)$ it holds
 \[ c \min\left\{ \delta^2, \eps\delta^{3/2} \left( \left\vert \ln\left( \frac{\eps}{\delta^{1/2}}\right) \right\vert +1\right)\right\} \leq  \min \left\{ \Eiso(u; \Omega) \colon u \in \Sboundper \right\}. \] 
	\end{itemize}
   \end{theorem}
   \begin{proof}
      The lower bound follows immediately from the lower bound on $\Sbound$. For the upper bound we need to adjust the construction in such a way that the associated spin fields satisfy the periodicity condition \eqref{gamma:spins:periodic_boundary}. The constructions are very similar to the ones used in Proposition \ref{lem:upper}, and we sketch them only for the readers' convenience.\\

      \quad \textbf{Ferromagnet.} As in the proof of Proposition \ref{lem:upper} we choose $u^{(1)} \coloneqq (0,1)^T$ on $\Omega \cap \eps \Z^2.$ This spin field satisfies \eqref{gamma:spins:periodic_boundary} and  $\Eiso(u^{(1)}) \leq 4\delta^2. $\\

      \quad \textbf{Branching-type construction. } We follow the lines of the proof of the upper bound of Proposition \ref{lem:upper}. Recall Steps 2 and 3 from the branching-type construction. Let $N\in \N$ and $(\Psi_N^{hor}, \Psi^{ver}_N)\colon \R^2 \rightarrow [-\pi, \pi)^2$ with $\eps\curl((\Psi_N^{hor}, \Psi^{ver}_N)) =0$ on $\R^2\cap\eps\Z^2.$ Let $\lambda \in (0, 2^{-N}]$. We slightly adjust the previous construction by setting (see Fig.~\ref{fig:enter-label} for a sketch)
      \begin{equation}
          (\Phi_{N}^{hor}, \Phi_{ N}^{ver}) \coloneqq \begin{cases}
              (\Psi^{hor}_N, \Psi^{ver}_N) \quad & \textup{ on } (-\infty, 1-2\lambda]\times \R,\\
              (0, -\opt) \quad & \textup{ on } [1-2\lambda,\infty) \times [1/2, \infty),\\
              (0, \phantom{-}\opt) \quad & \textup{ on } [1-2\lambda,\infty) \times (-\infty, 1/2]
          \end{cases}
      \end{equation}
      Note that $(\Phi_{N}^{hor}, \Phi_{ N }^{ver})$ again satisfies $\eps \curl((\Phi_{N}^{hor}, \Phi_{ N}^{ver})) = 0$ on $\R^2 \cap \eps \Z^2.$ 
      \begin{figure}[h]
    \centering
    \includegraphics[scale = 0.35]{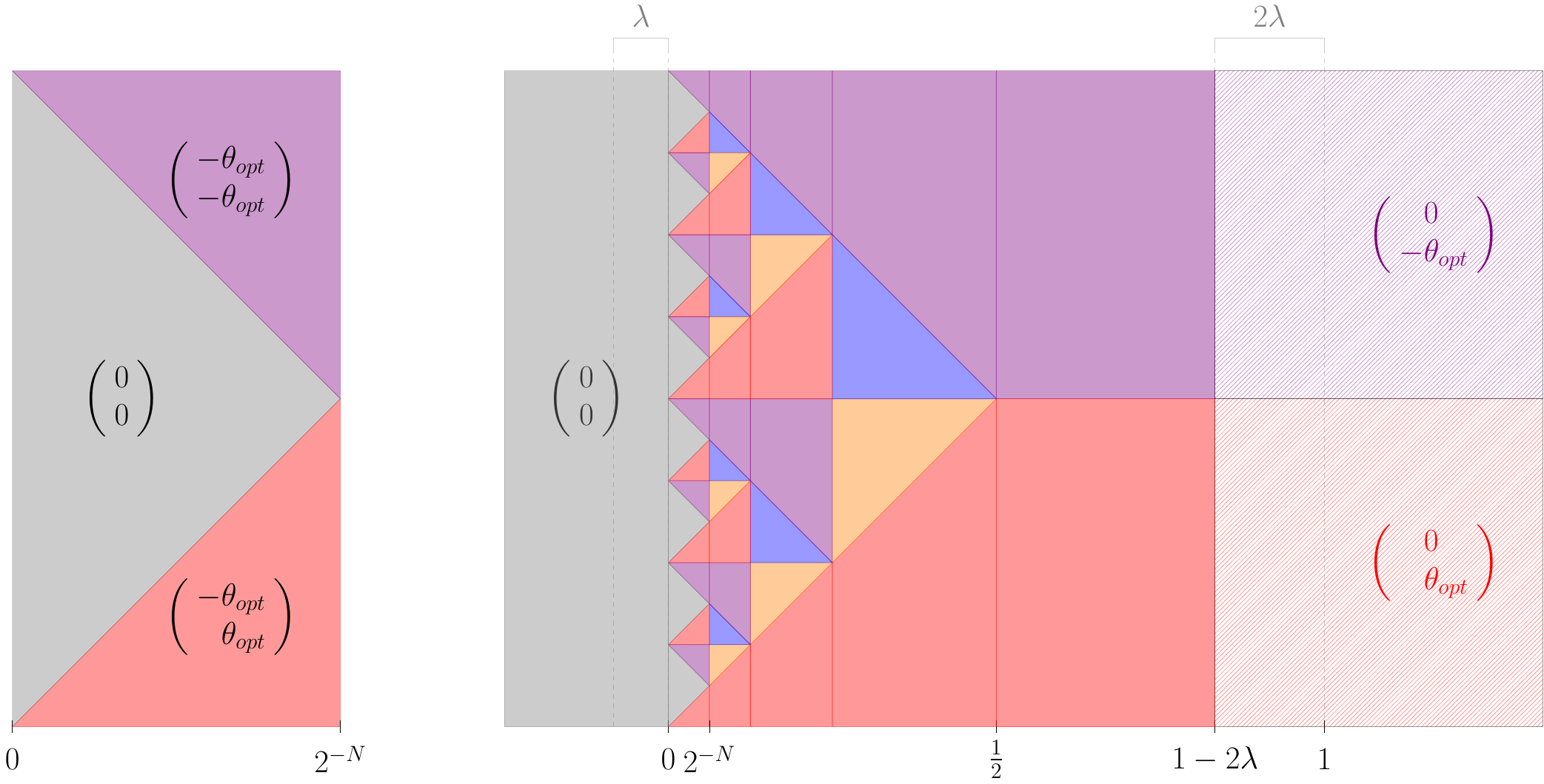}
    \caption{Sketch of $(\Phi_{N}^{hor}, \Phi_{ N }^{ver})$.  A corresponding spin field satisfies \eqref{gamma:spins:periodic_boundary}.}
    \label{fig:enter-label}
    \end{figure}
    We continue by mollifying $(\Phi_{N}^{hor}, \Phi_{ N}^{ver})$ with a suitable convolution kernel, i.e., let $p_\lambda\colon \R \rightarrow \R$ be a standard mollifier satisfying \eqref{prop:mollifier} and consider  $(\Psi_{N,\lambda}^{hor}, \Psi_{N, \lambda}^{ver})\colon \R^2 \rightarrow [-\pi, \pi)^2$, given for $(x,y) \in \R^2$ by (c.f. Figure \ref{fig:smooth_spins})
  \begin{equation}
  \begin{array}{rcl}
  \Psi_{N,\lambda}^{hor} (x,y) &\coloneqq &\int\limits_{\R^2} p_{\lambda}(s) p_\lambda(t) \Phi_N^{hor}(x-\lambda-s, y-t) \textup{ d}s \textup{d}t\\
    \Psi_{N,\lambda}^{ver} (x,y) &\coloneqq & \int\limits_{\R^2} p_{\lambda}(s) p_\lambda(t) \Phi_N^{ver}(x-\lambda-s, y-t) \textup{ d}s \textup{d}t.
  \end{array}
  \end{equation}
Similarly to the proof of Proposition \ref{lem:upper}, we set
  \[ (\theta_{N,\lambda}^{hor}, \theta_{N,\lambda}^{ver}) \coloneqq (\Psi_{N,\lambda}^{hor}, \Psi_{N, \lambda}^{ver}) \quad \textup{ on } \Omega \cap \eps \Z^2 \] and note that $(\theta_{N,\lambda}^{hor}, \theta_{N,\lambda}^{ver})$ induces a spin field $u^{(2)} \in \Sboundper$. 
Estimating the energy, we proceed analogously as for the branching-type construction in the proof of Proposition \ref{lem:upper}, and take care of the additional energy contributions from the extra transition layer that near $\{1\}\times(0,1)$, which yields
\begin{equation}
   \Eisohor(u^{(2)}_{N,\lambda}, \Omega) \lesssim \delta^2 \cdot 2^{-N} + N\delta^2\lambda +  N \frac{\eps^2 \delta}{\lambda} .
\end{equation}
As before one chooses 
  \begin{equation*}
      N \coloneqq \left\lceil\left\vert \frac{ \ln \left( \frac{\eps}{\delta^{1/2}}\right)}{\ln(2)}\right\vert \right\rceil \quad  \textup{   and   }\quad  \lambda \coloneqq \frac{\eps}{2\delta^{1/2}}\in (0,2^{-N})\subseteq(0,1),
  \end{equation*}
  to obtain the desired bound
\begin{equation}
    \Eiso(u^{(2)}_{N, \lambda} , \Omega)\lesssim \eps\delta^{3/2} \left(\left\vert \ln \left( \frac{\eps}{\delta^{1/2}} \right) \right\vert + 1\right).
\end{equation}
\text{ }

\quad \textbf{Vortex structure.} Again, we proceed as for the non-periodic case in Proposition \ref{lem:upper}. Recall the definition of $(\Psi^{hor}, \Psi^{ver})\colon \Omega \cap \eps\Z^2\rightarrow [-\pi, \pi)^2$ in Steps 1 and 2 there (see Figure \ref{fig:building_block_vortices}).
 We define the angular velocity field $(\theta_3^{hor}, \theta^{ver}_3)\colon \Omega \cap \eps \Z^2 \rightarrow [-\pi, \pi)^2$  first on $[0,1/2]^2$ and then extend it to $\Omega\cap\eps\Z^2$ by repeated reflections. We again let $M\coloneqq \lfloor \pi /\opt\rfloor\geq 2$ and set
\begin{equation}
    (\theta_3^{hor}, \theta^{ver}_3) \coloneqq \begin{cases}
        (\Psi^{hor}, \Psi^{ver}) \quad & \textup{ on } [0, \min\{M\eps, 1/2\})\times [0,1/2],\\
        (-\opt, \opt) \quad & \textup{ on } [\min\{M\eps, 1/2\}, 1/2] \times [0, 1/2],
    \end{cases}
\end{equation}
We proceed by reflecting the construction (horizontally ) over $\{ \max([0, 1/2]\cap \eps\Z)\} \times [0, 1/2)$ and afterwards (vertically) over $[0,1) \times \{ \max([0, 1/2]\cap \eps \Z)\}$, see Figure \ref{fig:vortex_structure_per}.\\
\begin{figure}[ht!]
    \centering
    \includegraphics[scale = 0.07]{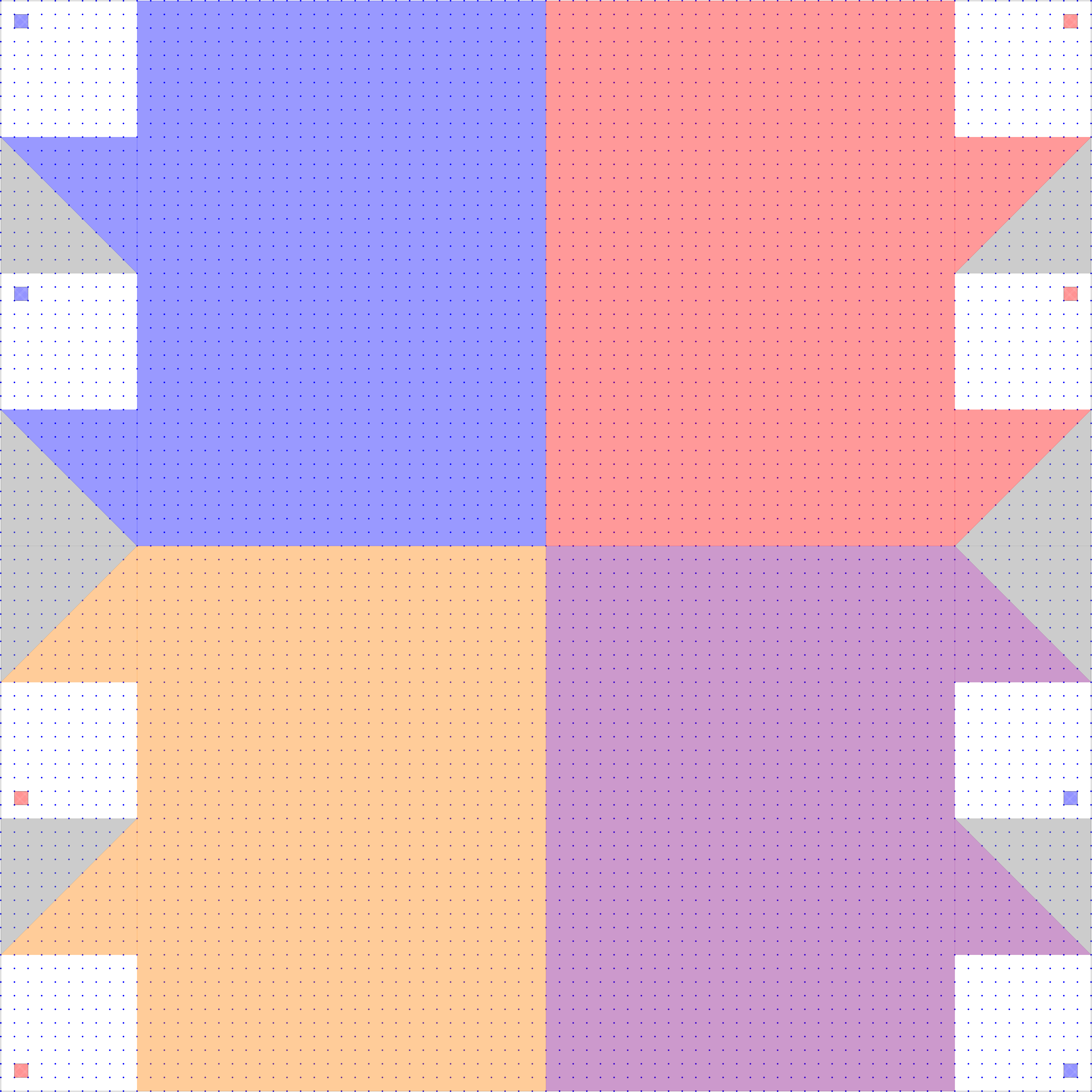}
    \caption{Sketch of the construction $(\theta_3^{hor}, \theta^{ver}_3)$ satisfying a periodicity condition. }
    \label{fig:vortex_structure_per}
\end{figure}

Let $k_{hor} \coloneqq \lfloor 1/(2\eps)\rfloor$. We extend $(\theta_3^{hor}, \theta^{ver}_3)$ to $([0,1)\times [0, 1/2])\cap\eps\Z^2$ by setting 
\begin{equation} 
\theta_{3,k_{hor} + i, j}^{hor}\coloneqq -\theta_{3,k_{hor} - i, j}^{hor}  \quad \textup{ and } \quad \theta^{ver}_{3, k_{hor}+ i, j} \coloneqq \theta^{ver}_{3, k_{hor}- i, j} \quad \textup{ for } (i\eps, j\eps) \in [0, 1/2) \times [0, 1/2].
\end{equation}
Note that this extension still satisfies a discrete curl condition on $([1/2, 1)\times [1/2, 1))\cap \eps\Z^2$ since 
\begin{equation}
    \eps \curl((\theta_3^{hor}, \theta^{ver}_3)_{k_{hor}+ i, j}) =  - \eps \curl((\theta_3^{hor}, \theta^{ver}_3)_{k_{hor}- i, j})\in \{ 2\pi, 0, -2\pi\}
\end{equation}
for $(i\eps, j\eps)\in [0, 1/2)^2\cap \eps \Z^2$. Analogously, we extend $(\theta_3^{hor}, \theta^{ver}_3)$ vertically to the whole domain $\Omega \cap\eps\Z^2$. Since the discrete curl-condition is satisfied, there is a spin field $u^{(3)} \colon \Omega \cap \eps\Z^2 \rightarrow \mathbb{S}^1$ whose associated angular velocity field is given by $(\theta_3^{hor}, \theta_3^{ver})$. The energy can be bounded similarly to the non-periodic case, including the second vertical layer of vortices on the left boundary and the additional horizontal and vertical transition layers, see Figure \ref{fig:vortex_structure_per}. Hence, we obtain the bound
\begin{equation}
    \Eiso(u^{(3)}, \Omega) \lesssim \eps\delta^{1/2} + \eps \delta \lesssim \eps\delta^{1/2}.
\end{equation}
This concludes the proof of the upper bound of Theorem \ref{thm:scaling_per}.
\end{proof}

\section*{Acknowledgement}
The authors gratefully acknowledge funding of the Deutsche Forschungsgemeinschaft (DFG, German Research Foundation) via project 195170736 - TRR 109. In addition, the authors would like to thank Marco Cicalese and Marwin Forster for helpful discussions. Furthermore, Melanie Koser would also like to thank Franz Bethke for illuminating numerical experiments achieved in a collaboration.  

\section*{Appendix}\label{sec:appendix}

In the following we recall a heuristic argument based on \cite{B11-2015-CS, B11-2019-CFO} for the Gamma-convergence result in Section \ref{sec:Gamma}. We adjust the presentation to the anisotropic model for the interaction parameters $\dhor, \dver \in (0,1).$
For $\eps > 0$ small enough one formally has $(u_{i+2, j} - 2u_{i+1, j} + u_{i,j})/\eps^2 \approx \partial_{11} u $ and similarly $(u_{i, j+2} - 2u_{i, j+1} + u_{i,j})/\eps^2 \approx \partial_{22} u $ for some smooth enough function $u \colon \Omega \rightarrow S^1$ (see below). This consideration and a formal computation \cite{B11-2015-CS, B11-2019-CFO} lead to
\begin{equation}
\begin{split}
& \E(u, \Omega)  = \frac{\eps^2}{2} \sum_{(i,j)} \vert u_{i,j} -2(1-\dhor) u_{i+1, j} + u_{i+2, j}\vert^2 + \frac{\eps^2}{2} \sum_{(i,j)} \vert u_{i,j} -2(1-\dver) u_{i, j+1} + u_{i, j+2}\vert^2 \\
 & \quad = \frac{\eps^2}{2} \sum_{(i,j)} \left\vert \eps^2 \frac{u_{i+2, j} - 2u_{i+1, j} + u_{i,j}}{\eps^2} + 2 \dhor u_{i+1,j} \right\vert  + \frac{\eps^2}{2} \sum_{(i,j)} \left\vert \eps^2 \frac{u_{i, j+2} - 2u_{i, j+2} + u_{i,j}}{\eps^2} + 2 \dver u_{i,j+2} \right\vert\\
 & \quad \simeq \frac12\int \vert \eps^2 \partial_{11}u + 2\dhor u\vert^2  + \left\vert \eps^2 \partial_{22}u  + 2 \dver u\right\vert \dx\dy \\
 & \quad \simeq \frac12\int 4(\dhor)^2 + 4\dhor\eps^2 u \partial_{11} u + \eps^4 \vert \partial_{11} u \vert^2 + 4(\dver)^2 + 4\dver\eps^2 u \partial_{22} u + \eps^4 \vert \partial_{22} u \vert^2  \dx\dy \\
 & \quad \simeq \frac12\int 4(\dhor)^2 - 4\dhor\eps^2 \vert \partial_1 u \vert^2  + \eps^4 \vert \partial_{11} u \vert^2 + 4(\dver)^2 - 4\dver\eps^2 \vert \partial_2 u \vert^2 + \eps^4 \vert \partial_{22} u \vert^2  \dx\dy
\end{split}
\end{equation}
Using polar coordinates, we write $u$ in the form $u = (\cos(\Psi), \sin(\Psi))$ for some suitable $\Psi \colon \Omega \to [-\pi, \pi)$. Assuming that $\Psi$ if differentiable and using $\vert \partial_i u \vert^2 = \vert \partial_i \Psi\vert^2$ and $\vert \partial_{ii} u \vert^2 = \vert \partial_{ii} \Psi\vert^2 + \vert \partial_{i} \Psi\vert^4$ for $i=1,2$, we formally obtain
\begin{equation}
    \begin{split}
        & \E(u, \Omega) \\ \simeq &\frac12\int 4(\dhor)^2 - 4(\dhor)\eps^2 \vert \partial_1 u \vert^2  + \eps^4 \vert \partial_{11} u \vert^2 + 4(\dver)^2 - 4\dver\eps^2 \vert \partial_2 u \vert^2 + \eps^4 \vert \partial_{22} u \vert^2  \dx\dy \\
        &  =  \frac12\int (2 \dhor - \eps^2 \vert \partial_1 \Psi\vert^2)^2   + \eps^4 \vert \partial_{11} \Psi \vert^2 + (2 \dver - \eps^2 \vert \partial_2 \Psi\vert^2)^2 + \eps^4 \vert \partial_{22} \Psi \vert^2  \dx\dy \\
        &  = \frac12 \int 4 (\dver)^2 \left(\frac{\dhor}{\dver} -  \left\vert \frac{\eps}{\sqrt{2\dver}}\partial_1 \Psi\right\vert^2\right)^2   + \eps^4 \vert \partial_{11} \Psi \vert^2 \dx\dy \\
        &  \phantom{ = } + \frac12 \int 4 (\dver)^2 \left(1 -  \left\vert \frac{\eps}{\sqrt{2\dver}}\partial_2 \Psi\right\vert^2\right)^2  + \eps^4 \vert \partial_{22} \Psi \vert^2  \dx\dy 
    \end{split}
\end{equation}
Substituting $\varphi = \eps \Psi / \sqrt{2\dver}$ and $\gamma^2 \coloneqq \dhor/\dver$ finally leads to 
\begin{equation}
    \begin{split}
        & \E(u, \Omega) \\ \simeq & \frac12\int 4 (\dver)^2 \left(\gamma^2 -  \left\vert\partial_1 \varphi\right\vert^2\right)^2   + 2\dver\eps^2 \vert \partial_{11} \varphi \vert^2 + 4 (\dver)^2 \left(1 -  \left\vert \partial_2 \varphi\right\vert^2\right)^2  + 2\dver\eps^2 \vert \partial_{22} \varphi \vert^2  \dx\dy\\ 
        = &\sqrt{2} \eps (\dver)^{3/2} \int \bigg[ \frac{\sqrt{2\dver}}{\eps}\left(\gamma^2 -  \left\vert\partial_1 \varphi\right\vert^2\right)^2 + \frac{\eps}{\sqrt{2\dver}} \vert \partial_{11} \varphi\vert^2 + \frac{\sqrt{2\dver}}{\eps}\left(1 -  \left\vert\partial_2 \varphi\right\vert^2\right)^2 \\ & \hspace{3cm} +   \frac{\eps}{\sqrt{2\dver}} \vert \partial_{22} \varphi\vert^2 \bigg] \dx \dy 
    \end{split}
\end{equation}
Assuming small angular changes of the spin field $u$ (see Section \ref{sec:Gamma} for a rigorous proof) in the horizontal and vertical direction formally yields 
\[\partial_1 \varphi =  \frac{\eps}{\sqrt{2\dver}} \partial_1 \Psi \simeq \frac{1}{\sqrt{2\dver}} \hor = \gamma \frac{1}{\sqrt{2\dhor}} \hor \simeq \gamma \sqrt{\frac{2}{\dhor}} \sin\left( \frac{\hor}{2}\right) \eqqcolon \gamma w^\Delta(u)  \]
and similarly 
\[ \partial_2 \varphi = \frac{\eps}{\sqrt{2\dver}} \partial_2 \Psi  \simeq  \sqrt{\frac{2}{\dver}} \sin \left( \frac{\ver}{2}\right)\eqqcolon z^\Delta(u),\] where $w^\Delta(u)$ and $z^\Delta(u)$ are the chirality order parameters introduced in Section \ref{sec:model}.
Setting $\sigma \coloneqq \eps/\sqrt{2\dver}$ and the formal equality $\nabla \varphi = (\gamma w^\Delta(u), z^\Delta(u))$, we obtain the renormalized energy 
\begin{equation}
    \begin{split}
        &\H(u, \Omega) \\ \coloneqq &\frac1\sigma \int (\gamma^2 - \vert\partial_1\varphi\vert^2)^2 \dx \dy + \sigma \int \vert \partial_{11} \varphi\vert^2 \dx\dy + \frac1\sigma \int (1- \vert \partial_2 \varphi\vert^2 )^2 \dx\dy + \sigma \int \vert\partial_{22} \varphi \vert^2 \dx\dy \\
         \simeq &\frac{1}{\sqrt{2} \eps (\dver)^{3/2}} \E(u, \Omega).
    \end{split}
\end{equation}

\bibliography{spins}
\bibliographystyle{alpha}
\end{document}